\documentclass[10 pt]{amsart}
\usepackage{bbm, mathrsfs, enumerate, amssymb}
\usepackage[foot]{amsaddr}

\usepackage[latin1]{inputenc}
\usepackage{amsfonts,amssymb,amsmath,amsthm}
\usepackage[headinclude,DIV13]{typearea}
\usepackage{hyperref}
\usepackage{mathrsfs,bbm}

\usepackage{geometry}
\usepackage{graphicx}

 \usepackage{amsmath,amsfonts,amssymb}
 \usepackage{bbold}
 \usepackage{setspace}

\usepackage{color}
\usepackage{ulem}

\newtheorem{rem}{Remark}
\newtheorem{lem}{Lemma}[section]
\newtheorem{pro}{Proposition}
\newtheorem{defi}{Definition}[section]

\newtheorem{theo}{Theorem}

\newtheorem{ass}{Assumption}

\theoremstyle{definition} 
\theoremstyle{definition} 

\renewcommand{\P}{\mathbb{P}}
\newcommand{\R}{\mathbb{R}}

\newcommand{\E}{\mathbb{E}}

\newcommand{\N}{\mathbb{N}}
\newcommand{\Z}{\mathbb{Z}}

\newcommand{\eps}{\varepsilon}

\DeclareMathOperator{\var}{Var}
\DeclareMathOperator{\cov}{Cov}
\numberwithin{equation}{section}

\newtheorem{ex}{Example}

\begin{document}

\title[Genealogies of two neutral loci after a selective sweep]
{Genealogies of two linked neutral loci after a selective sweep in a large population of stochastically varying size}

\author{Rebekka Brink-Spalink}
\address{Institute for Mathematical Stochastics,
Georg-August-Universit\"at G\"ottingen,
Goldschmidtstr. 7,
37077 G\"ottingen, Germany}
\email{rbrink@math.uni-goettingen.de}

\author{Charline Smadi}
\address{Irstea, UR LISC, Laboratoire d'Ingénierie des Systèmes Complexes, 9 avenue Blaise Pascal-CS 20085, 63178 Aubière, France and
Department of Statistics, University of Oxford, 1 South Parks Road, Oxford OX1 3TG, UK}
\email{charline.smadi@polytechnique.edu}


\maketitle

\begin{abstract}
We study the impact of a hard selective sweep on the genealogy of partially linked neutral loci in the vicinity of the positively
selected allele. We consider a sexual population of stochastically
varying size and, focusing on two neighboring loci, derive an approximate formula for the neutral genealogy of a sample of individuals
taken at the end of the sweep. 
Individuals are characterized by ecological parameters depending on their genetic type and governing their growth rate and interactions 
with other individuals (competition). As a consequence, the "fitness" of an individual depends on the population state and is not 
an intrinsic characteristic of individuals.
We provide a
deep insight into the dynamics of the mutant and wild type populations during the different stages of a selective sweep.
\end{abstract}

 \section*{Introduction}
We study the hitchhiking effect of a beneficial mutation in a sexual haploid population of stochastically varying size. We assume
that a mutation occurs in one individual of a monomorphic population
 and that individuals carrying the new allele $a$ are better adapted to the current environment and spread in the population.
We suppose that the mutant allele $a$ eventually replaces the resident one, $A$, and
study the influence of this fixation on the neutral gene genealogy of a sample taken
at the end of the selective sweep. That is, in each sampled individual we consider the same set of partially
linked loci including the locus where the advantageous mutation occurred. We then trace back the ancestral
lineages of all loci in the sample until the beginning of the sweep and update the genetic relationships whenever a
coalescence or a recombination changes the ancestry of one or several loci.
Our main result is the derivation of a sampling formula for the ancestral partition of two neutral loci situated in the vicinity of
the selected allele.

The first studies of hitchhiking, initiated by Maynard Smith and Haigh \cite{smith1974hitch}, have modeled the mutant population size as
the solution of a deterministic logistic equation \cite{ohta1975effect,kaplan1989hitchhiking,stephan1992effect,
stephan2006hitchhiking}. Barton \cite{barton1998effect}
 was the first to point out the importance of the stochasticity of the mutant population size.
Following this paper, a series of works took into account this randomness during the sweep.
In \cite{durrett2004approximating,schweinsberg2005random} Schweinsberg and Durrett based their analysis on a Moran model with selection and
recombination, while Etheridge and coauthors \cite{etheridge2006approximate} worked with the diffusion limit of such discrete population models.
Then Brink-Spalink \cite{brink2014multsweep}, Pfaffelhuber and Studeny \cite{pfaffelhuber2007approximating} and Leocard
\cite{stephanie2009selective} extended the respective findings of these two approaches for the ancestry
of one neutral locus to the two-locus (resp. multiple-locus) case.

However, in all these models, the population size was constant and each individual had a ``fitness'' only dependent
on its type and not on {the population state}.
The fundamental idea of Darwin is that the individual traits have an influence on the interactions between individuals,
which in turn generate selection on the different traits.
In this paper we aim at modeling precisely these interactions by extending the model introduced in \cite{smadi2014eco} where the author considered only one neutral locus.
Such an eco-evolutionary approach has been introduced by Metz and coauthors \cite{metz1996adaptive} and has been made rigorous in
the seminal paper of Fournier and M\'el\'eard \cite{fournier2004microscopic}. Then it was further developed by
Champagnat, M\'el\'eard and coauthors (see
\cite{champagnat2006microscopic,
champagnat2011polymorphic,champagnat2014adaptation} and references therein) for
the haploid asexual case and by Collet, M\'el\'eard and Metz \cite{collet2011rigorous} and Coron and coauthors 
\cite{coron2013slow} for the diploid sexual case.

The population dynamics, described in Section \ref{modelandresults}, is a multitype birth and death Markov process
with competition.
We represent the carrying capacity of the underlying environment by a scaling parameter $K\in \N$ and state results in the limit for
large $K$. In \cite{champagnat2006microscopic} it was shown that such kind of invasion processes can be divided into three phases
(see Figure \ref{fig3loci}): an
 initial phase in which the fraction of $a$-individuals does not exceed a fixed value $\eps>0$ and where the dynamics of the wild type
 population is nearly undisturbed by the invading type. A second phase where both types account for a non-negligible percentage of the
 population and where the dynamics of the population can be well approximated by a deterministic competitive Lotka-Volterra system.
 And finally a third phase where the roles of the types are interchanged and the wild type population is near extinction.
The durations of the first and third phases of the selective sweep are of order $\log K $ whereas the second phase only
lasts an amount of time of order 1. This three phases decomposition is commonly encountered in population genetics models
and dates back to \cite{ohta1975effect}.

In Section \ref{sectioncoupling} we precisely describe these three phases and introduce two couplings of the population process, key tools to study the
dynamics of the $A$- and $a$-populations.
Section \ref{sectionproofmain} is devoted to the proofs of the main theorems on the ancestral partition of the two neutral alleles.
Sections \ref{technicalsection} to \ref{secondphase} are dedicated to the proofs of auxiliary statements.
In Section \ref{sectioncomparison} we compare our findings with previous results.
Finally, we state technical results needed in the proofs in the Appendix.

\section{Model and results} \label{modelandresults}

We consider a three locus model: one locus under selection, $SL$, with alleles in
$\mathcal{A}:=\{A,a\}$ and two neighboring neutral loci $N1$ and $N2$ with alleles in the finite sets $\mathcal{B}$ and
$\mathcal{C}$ respectively. We denote by $\mathcal{E}=\mathcal{A}\times \mathcal{B} \times \mathcal{C}$ the type space.
Two geometric alignments are possible: either the two neutral loci are adjacent (geometry $SL-N1-N2$), or they are separated by the
selected locus (geometry $N1-SL-N2$).
We introduce the model and notations for the adjacent geometry, their analogs for the separated one can be deduced
in a straightforward manner.

Whenever a reproduction event takes place, recombinations between $SL$ and $N1$ or between $N1$ and $N2$
occur independently with probabilities $r_{1}$ and $r_{2}$, respectively. These probabilities depend on the parameter $K$,
representing the environment's carrying capacity, but for the purpose of readability we do not indicate this dependence. We assume a regime of weak recombination:
  \begin{align}\label{assrK}
 \limsup_{K \rightarrow \infty}\ r_{j} \log K < \infty, \ j=1,2.
\end{align}
This is motivated by Theorem 2 in \cite{smadi2014eco} which states that
this is the good scale to observe a signature on the neutral allele distribution.
If the recombination
probabilities are larger (neutral loci more distant from the selected locus), there are many recombinations
and the sweep does not modify the neutral diversity at these sites.
Recombinations may lead to a mixing of the parental genetic material in the newborn, and hence, parents with types
$\alpha \beta \gamma$ and $\alpha' \beta' \gamma'$ in $\mathcal{E}$ can generate the following offspring:
\begin{align*}
\renewcommand{\arraystretch}{1.4}{
\begin{array}{ccc}
 \text{possible genotype} & \text{event} & \text{{probability}}\\
 \alpha \beta \gamma,\alpha' \beta' \gamma' & \text{no recombination } & (1-r_{1})(1-r_{2})\\
\alpha \beta' \gamma',\alpha' \beta \gamma & \text{one recombination {between $SL$ and $N1$}} &  r_{1}(1-r_{2})\\
 \alpha \beta \gamma',\alpha' \beta' \gamma & \text{one recombination {between $N1$ and $N2$}} & (1-r_{1}) r_{2}\\
\alpha \beta' \gamma,\alpha' \beta \gamma' &  \text{two recombinations} & r_{1} r_{2}
\end{array}}
\end{align*}
We will see in the sequel that the probability to witness a birth event with two simultaneous recombinations in the neutral genealogy of a uniformly
chosen individual is very small.

As we assume the loci $N1$ and $N2$ to be neutral,
the ecological parameters of an individual
only depend on the allele $\alpha$ at the locus under selection.
Let us denote by $f_\alpha$ the fertility of an individual with type $\alpha$.
In the spirit of \cite{collet2011rigorous}, such an individual gives birth at rate $f_\alpha$ (female role), and has a probability
proportional to $f_\alpha$ to be chosen as the father in a given birth event (male role).
Denoting the complementary type of the
allele $\alpha$
by $\bar{\alpha}$ we get the following result for the birth rate of individuals of
type $\alpha \beta\gamma \in \mathcal{E}$:
\begin{multline}\label{birthrate}
 b^{K}_{\alpha \beta \gamma }(n) = (1-r_1)(1-r_2) f_{\alpha}n_{\alpha \beta \gamma} + r_1(1-r_2)f_{\alpha}n_{\alpha}\frac{f_{\alpha}n_{\alpha \beta \gamma} +
 f_{\bar{\alpha}}n_{\bar{\alpha}\beta \gamma}}{f_{a}n_a+f_An_A} +
\\(1-r_1)r_2f_{\alpha} \frac{\sum_{(\beta',\gamma')\in (\mathcal{B},\mathcal{C})}n_{\alpha \beta \gamma'}( f_{\alpha}n_{\alpha \beta' \gamma}+ f_{\bar{\alpha}}n_{\bar{\alpha}\beta' \gamma})}{f_{a}n_a+f_An_A} 
+r_1r_2 f_{\alpha}
\frac{\sum_{(\beta',\gamma')\in (\mathcal{B},\mathcal{C})}n_{\alpha \beta' \gamma }( f_{\alpha}n_{\alpha \beta\gamma' }+ f_{\bar{\alpha}}n_{\bar{\alpha}\beta \gamma'})}{f_{a}n_a+f_An_A},
\end{multline}
where $n_{\alpha \beta \gamma}$ (resp. $n_\alpha$) denotes the current number of $\alpha \beta \gamma$-individuals
(resp. $\alpha$-individuals) and
$n=(n_{\alpha \beta \gamma}, (\alpha, \beta, \gamma) \in \mathcal{E})$
is the current state of the population.
An $\alpha$-individual can die either from a natural death (rate $D_\alpha$), or from type-dependent competition:
 the parameter $C_{\alpha,\alpha'}$ models the impact an
individual of type $\alpha'$ has on an individual of type $\alpha$, where $(\alpha,\alpha') \in \mathcal{A}^2$.
The strength of the competition also depends on
the carrying capacity $K$.
This
results in the total death rate of individuals carrying the alleles $\alpha \beta \gamma \in \mathcal{E}$:
\begin{align}\label{deathrate}
 d^{K}_{\alpha \beta \gamma }(n) =\; & \left( D_{\alpha} +  \frac{C_{\alpha,A}}{K}n_{A} + \frac{C_{\alpha,a}}{K} n_{a} \right) n_{\alpha \beta \gamma}.
\end{align}
Hence the population process
\begin{align*}
N^K= (N^{K}(t), t \geq 0)=\Big((N_{\alpha \beta \gamma}^{K}(t))_{(\alpha, \beta, \gamma)\in \mathcal{E}}, t \geq 0\Big),
\end{align*}
where $N_{\alpha \beta \gamma}^{K}(t)$ denotes the number of $\alpha \beta \gamma$-individuals
at time $t$, is a multitype birth and death process with rates given in \eqref{birthrate} and \eqref{deathrate}.
We will often work with the trait population process $( (N_A^K(t),N_a^K(t)),t \geq 0)$, where $N_{\alpha}^{K}(t)$ denotes the number of $\alpha$-individuals at time $t$.
This is also a birth and death process with birth
and death rates given by:
\begin{align}\label{def:totalbd}
 b^{K}_{\alpha }(n) &=\sum_{(\beta,\gamma) \in \mathcal{B}\times \mathcal{C}}  b^{K}_{\alpha \beta \gamma }(n) = f_{\alpha} n_{\alpha}
 \\
 d^{K}_{\alpha }(n) &=\sum_{(\beta,\gamma) \in \mathcal{B}\times \mathcal{C}}  d^{K}_{\alpha  \beta \gamma }(n) =\Big( D_{\alpha} + \frac{C_{\alpha,A}}{K} n_{A} +  \frac{C_{\alpha,a}}{K}n_{a}\Big) n_{\alpha}.\nonumber
\end{align}
As a quantity summarizing the advantage or disadvantage a mutant with allele type $\alpha$ has in an $\bar{\alpha}$-population at equilibrium, we introduce the so-called invasion fitness $S_{\alpha \bar{\alpha}}$ through
\begin{equation} \label{deffitinv3loci}
 S_{\alpha \bar{\alpha}} := f_{\alpha} -D_{\alpha} - C_{\alpha,\bar{\alpha}}\bar{n}_{\bar{\alpha}},
\end{equation}
where the equilibrium density $\bar{n}_{{\alpha}}$ is defined by
\begin{align}\label{defnbar}
\bar{n}_{{\alpha}}: =\frac{f_{\alpha} -D_{\alpha}}{C_{\alpha,{\alpha}}}.
\end{align}
The role of the invasion fitness $S_{\alpha \bar{\alpha}}$ and the
definition of the equilibrium density $\bar{n}_{{\alpha}}$ follow from the properties of the two-dimensional competitive Lotka-Volterra
system:
\begin{equation} \label{S}
 \dot{n}_\alpha^{(z)}=(f_\alpha-D_\alpha-C_{\alpha,A}n_A^{(z)}-C_{\alpha,a}n_a^{(z)})n_\alpha^{(z)},\quad
z \in \R_+^\mathcal{A}, \quad
n_\alpha^{(z)}(0)=z_\alpha,\quad  \alpha \in \mathcal{A}.
 \end{equation}
If we assume
\begin{equation}\label{defnbara}
 \bar{n}_{A}>0,\quad \bar{n}_{a}>0,\quad \text{and} \quad S_{Aa}<0<S_{aA},
\end{equation}
then $\bar{n}_{\alpha}$ is the equilibrium size of a monomorphic $\alpha$-population and
the system \eqref{S} has a unique stable equilibrium $(0,\bar{n}_a)$ and two unstable steady states $(\bar{n}_A,0)$ and $(0,0)$.
Thanks to Theorem 2.1 p. 456 in \cite{ethiermarkov} we can prove that
if $N_A^K(0)$ and $N_a^K(0)$ are of order $K$ and $K$ is large, the rescaled process $(N_A^K/K,N_a^K/K)$ is very close to the solution of
\eqref{S} during any finite time interval.
The invasion
fitness $S_{aA}$ corresponds to the \textit{per capita} initial growth rate of the mutant $a$ when it appears in a monomorphic
population of individuals $A$ at their equilibrium size $\bar{n}_AK$.
Hence the dynamics of the allele $a$ is very dependent on the properties of the system \eqref{S} and it is proven in
\cite{champagnat2006microscopic}
that under Condition \eqref{defnbara} one mutant $a$ has a positive probability to fix in the population and replace a wild type $A$.
More precisely, if we use the convention
\begin{equation}\label{convP} \P^{(K)}(.):=\P(.|N_A^K(0)=\lfloor \bar{n}_A K \rfloor,N_a^K(0)=1), \end{equation}
Equation (39) in \cite{champagnat2006microscopic} states that
\begin{equation}\label{probafix}
\lim_{K \to \infty}\P^{(K)}(\text{Fix}^K)=\frac{S_{aA}}{f_a}=:s,
\end{equation}
where $s$ is called the rescaled invasion fitness, and the extinction time of the $A$-population and the event of fixation of the $a$-allele are rigorously defined as follows:
\begin{align}\label{def:extfix}
  T_{\text{ext}}^K:=\inf \big\{t \geq 0 : N_{A}^{K}(t) = 0 \big\}\quad \text{and}\quad \text{Fix}^K:=\big\{T_{\text{ext}}^K<\infty, N_{a}^{K}(T_{\text{ext}}^K)>0\big\}.
\end{align}
From this point onward, we fix $d$ in $\N$.
We aim at quantifying the effect of the selective sweep on the neutral diversity. Our method consists in tracing back the neutral genealogies
of $d$ individuals sampled uniformly at the end of the sweep (time $T_{\text{ext}}^K$) until time $0$.
Two event types (see Definition \ref{defcoalreco}) may affect the relationships of the sampled neutral alleles: coalescences correspond to the merging of
the neutral genealogies of two individuals
at one or two neutral loci, and recombinations redistribute the selected and neutral alleles of one individual into two groups carried by its two parents.
We will represent the neutral genealogies by a partition $\Theta^K_d$ which belongs to the set $\mathcal{P}_d^*$ of marked partitions of
$\{ (i,k), i \in \{1,...,d\}, k \in \{1,2\} \}$ with (at most) one
block distinguished by the mark $*$, which will correspond to the descendants of the original mutant $a$.
In this notation $(i,1)$
and $(i,2)$ are the neutral alleles at loci $N1$ and $N2$ of the $i$th sampled individual.
Let us define rigorously the random partition $\Theta^K_d$:

\begin{defi}\label{deftheta}
Sample $d$ individuals uniformly and without replacement at the end of the sweep (time $T_{\text{ext}}^K$). Follow the genealogies of the first and second neutral
alleles of the $i$-th sampled individual, $(i,1)$ and $(i,2)$ for $i \in \{1, ...,d\}$.
Then the partition $\Theta^K_d \in \mathcal{P}_d^*$ is defined as follows:
each block of the partition $\Theta^K_d$ is composed of all those neutral alleles which originate from the same individual alive at the beginning of the sweep;
the block containing the descendants of the mutant $a$ (if such a block exists) is distinguished by the mark $*$.
\end{defi}

We will show in Theorems \ref{mainresult} and \ref{mainresult2} that when $K$ is large the partition $\Theta^K_d$ belongs
with a probability close to one to a subset $\Delta_d$ of $\mathcal{P}_d^*$, which
is defined as follows:

\begin{defi}\label{defdelta}
$\Delta_d$ is the subset of $\mathcal{P}_d^*$ consisting of those partitions whose unmarked blocks (if there are any)
 are either singletons or pairs of the form $\{(i,1),(i,2)\}$ for one $i\in \{1,...,d\}$.
\end{defi}


\begin{ex}
In the example represented in Figure \ref{schemagen}, the marked partition $\pi^{(ex)}$ belongs to $\Delta_d$:
$$\pi^{(ex)}=\Big\{\{ (1,1), (1,2), (2,1), (5,2) \}^*,
\{(2,2)\},\{(3,1),(3,2)\},\{(4,1)\},\{(4,2)\},\{(5,1)\} \Big\}.$$
\end{ex}

\begin{figure}[h]
  \centering
 \includegraphics[width=11cm,height=4.5cm]{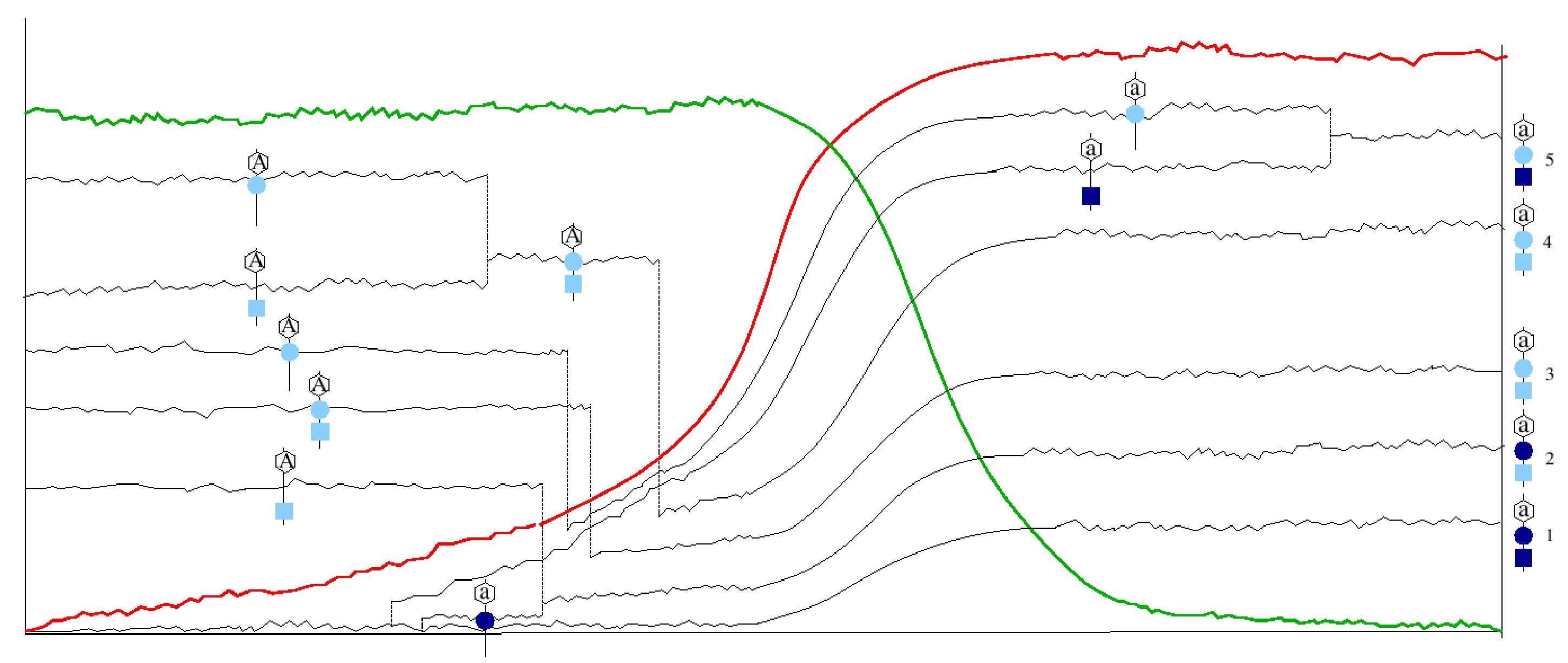}
\caption{Example of genealogy for a 5-sample:
dark blue neutral alleles originate from the mutant and light blue ones from an $A$-individual.
We indicate the selected allele, $A$ or $a$, associated with the neutral alleles during the sweep.
It can change when a recombination occurs.
Bold lines represent the $A$(green)- and $a$(red)-population sizes.
In this example,
 the two neutral alleles of the first individual, the first neutral allele of the second
individual and the second neutral allele of the fifth individual originate from the mutant;
the two neutral alleles of the third individual originate from the same $A$-individual, whereas the
two neutral alleles of the fourth individual originate from two distinct $A$-individuals. }
  \label{schemagen}
\end{figure}

For a partition $\pi \in \mathcal{P}_d^*$, we define for some possible ancestral relationships the number of individuals in the sample
whose two neutral loci are related in that particular way:

\begin{defi}\label{defpartition}
 Let $d\in \N$ and $\pi\in \mathcal{P}_d^*$. Then we set:
 \begin{enumerate}
  \item[] $|\pi|_1=\#\{$\small$1\leq i \leq d$  such that $(i,1)$ and $(i,2)$ belong to the marked block \normalsize  $\}$
  \item[] $|\pi|_2=\#\{$\small$ 1\leq i \leq d$ such that $(i,1)$ belongs to the marked block and $\{(i,2)\}$ is an unmarked block\normalsize  $\}$
  \item[] $|\pi|_3=\#\{$\small$1\leq i \leq d$ such that $(i,2)$ belongs to the marked block and $\{(i,1)\}$ is an unmarked block\normalsize  $\}$
  \item[] $|\pi|_4=\#\{$\small$1\leq i \leq d$ such that $\{(i,1),(i,2)\}$ is an unmarked block\normalsize  $\}$
  \item[] $|\pi|_5=\#\{$\small$1\leq i \leq d$ such that $\{(i,1)\}$ and $\{(i,2)\}$ are two distinct unmarked blocks\normalsize  $\}$
 \end{enumerate}
\end{defi}

To express the limit distribution of the partition $\Theta_d^{K}$ we need to introduce:
\begin{equation}
\label{defq1q2}
q_1:= e^{- \frac{f_ar_1\log K}{S_{aA}}},\ \ q_2:=e^{-\frac{f_ar_2\log K}{S_{aA}}}, \ \
\bar{q}_2:=e^{-\frac{f_ar_2\log K}{|S_{Aa}|}} \ \ \text{and} \ \
q_3:=\frac{r_1(q_2^{f_A/f_a}-q_1q_2)}{r_1+r_2(1-f_A/f_a)},
\end{equation}
where the invasion fitnesses have been defined in \eqref{deffitinv3loci}.
We did not make any assumption on
the sign of $f_a(r_1+r_2)-f_Ar_2$, but $q_3$ can be written in the form
$\delta(e^{-\mu}- e^{-\nu})/(\nu-\mu)$ for $(\delta,\mu,\nu) \in \R_+^3$
so that it is well defined and non-negative. It is easy to check that $q_3\leq 1$.
The forms of $q_1$, $q_2$ and $\bar{q}_2$ are intuitive (see comments of Proposition \ref{prop:1phase_probg1}).The form of $q_3$ is more complex to explain and results from 
a combination of different possible genealogical scenarios during the first phase.
We now define five non-negative numbers $(p_k, 1\leq k \leq 5)$
which will quantify the law of $\Theta_d^K$ for large $K$ in Theorem \ref{mainresult}:
\begin{eqnarray}\label{defpi} p_1 := q_1q_2 [1-(1-q_1)(1-
\bar{q}_2)], \quad  p_2 := q_1[(1-q_1q_2)
- q_2\bar{q}_2(1-q_1) ], \quad \\
 p_3 := q_1q_2(1-\bar{q}_2)(1-q_1), \quad
p_4 :=\bar{q}_2 q_3\quad \text{and} \quad  p_5 := (1-q_1)(1-q_1q_2(1-\bar{q}_2))
-\bar{q}_2q_3.\nonumber\end{eqnarray}

Note that $\sum_{1 \leq k \leq 5}p_k=1$. Finally, we introduce an assumption which summarizes all the assumptions made in this work:
\begin{ass}\label{asstotale}
$(N_A^K(0),N_a^K(0))=(\lfloor\bar{n}_AK\rfloor,1)$ and Conditions \eqref{assrK} on the recombination probability and \eqref{defnbara}
on the equilibrium densities and fitnesses hold.
\end{ass}
With Definitions  \ref{deftheta}, \ref{defdelta} and \ref{defpartition} in mind, we can now state our main results:

\begin{theo}[Geometry $SL-N1-N2$] \label{mainresult}
Under Assumption \ref{asstotale}, we have for every $ \pi \in \mathcal{P}_d^*$
$$ \underset{K \to \infty}{\lim} \
\Big| \P^{(K)}( \Theta^K_d=\pi |{\textnormal{Fix}}^K)-\mathbf{1}_{\{\pi \in \Delta_d\}}
{p_1}^{|\pi|_1}{p_2}^{|\pi|_2}{p_3}^{|\pi|_3}{p_4}^{|\pi|_4}{p_5}^{|\pi|_5} \Big|=0.$$
\end{theo}

Notice that when $K$ is large, $\Theta^K_d$ belongs to $\Delta_d$ with a probability close to one, and that
$$({p_1}^{|\pi|_1}{p_2}^{|\pi|_2}{p_3}^{|\pi|_3}{p_4}^{|\pi|_4}{p_5}^{|\pi|_5}, \pi \in \Delta_d)$$
 is a probability on
$\Delta_d$ (depending on $K$).
Moreover, this result implies that the $d$ sampled individuals have asymptotically independent neutral genealogies.
With high probability, the neutral alleles of a given sampled individual $i$ either originate from the first mutant $a$
and belong to the marked block,
 or escape the sweep and originate from an $A$ individual. In this case they belong to an unmarked block which is of the
 form $\{(i,1)\}$, $\{(i,2)\}$ or $\{(i,1),(i,2)\}$, according to Definition \ref{defpartition}.
As a consequence, if some neutral alleles of two distinct sampled individuals escape the sweep,
they originate from
distinct $A$-individuals with high probability.
However, the genealogies of the two neutral alleles of a given individual are not independent.
For example the probability that $(i,1)$ and $(i,2)$ escape the sweep is $p_4+p_5$; the probability that
$(i,1)$ (resp. $(i,2)$) escapes the sweep is $p_3+p_4+p_5$ (resp. $p_2+p_4+p_5$), {and for every $K \in \N$ such that $r_1\neq 0$
$$ (p_3+p_4+p_5)(p_2+p_4+p_5)=(1-q_1)(1-q_1q_2)<(1-q_1)(1-q_1q_2+q_1q_2\bar{q}_2)=p_4+p_5. $$
 This is due to the fact that if {(backwards in time)} a recombination first occurs between $SL$ and $N1$, the neutral allele at $N2$, linked to $N1$, also escapes the sweep.
As the term $q_1q_2\bar{q}_2$ does not tend to $0$ when $K$ goes to infinity under Condition \eqref{assrK}, the only possibility to have an equality in the limit is the case where
$r_1 \log K\ll 1$ or in other words when the probability to see a recombination between $SL$ and $N1$ is negligible.\\

Let us now consider the separated geometry, $N1-SL-N2$:

\begin{theo}[Geometry $N1-SL-N2$] \label{mainresult2}
Under Assumption \ref{asstotale}, we have for every $ \pi \in \mathcal{P}_d^*$
$$
 \underset{K \to \infty}{\lim} \ \Big| \P^{(K)}( \Theta^K_d=\pi |\textnormal{Fix}^K)-
\mathbf{1}_{\{\pi \in \Delta_d\}}{[q_1q_2]}^{|\pi|_1}{[q_1(1-q_2)]}^{|\pi|_2}{[(1-q_1)q_2]}^{|\pi|_3}{[(1-q_1)(1-q_2)]}^{|\pi|_5}
\Big|=0. $$
\end{theo}

Again the neutral genealogies of the $d$ sampled individuals are asymptotically independent.
Furthermore, we have independence between the neutral loci. Indeed Theorem \ref{mainresult2}
means that a neutral allele at locus $Nk$ escapes the sweep with probability $1-q_k$ independently of all other neutral alleles,
including the allele at the other neutral locus of the same individual.
This is due to the fact that in the separated geometry a recombination between $SL$ and one neutral locus has no impact on the genetic background
of the allele at the other neutral locus.
Note in particular that there is no block of the form $\{(i,1),(i,2)\}$ in the limit partition, as the two neutral alleles have a very small probability
to recombine at the same time.

\section{Comparison with previous work} \label{sectioncomparison}

In \cite{schweinsberg2005random} the authors gave an approximate sampling formula for the genealogy of one neutral locus
during a selective sweep. The population evolved as a two-locus modified Moran model
with recombination, selection, and in particular constant population size.
They introduced the fitness $s^{SD}$ of the mutant $a$ as follows: when one of the iid exponential clocks
of the living individuals rings, one picks two individuals uniformly at random (with replacement),
one dies, and the other one gives birth. A replacement of an $a$-individual by an
$A$-individual is rejected with probability $s^{SD}$. In this case, nothing happens.
In \cite{smadi2014eco}, the author studied the one neutral locus version of the here presented model.
It was shown that the ancestral relationships in a sample taken at the end of the sweep correspond to
the ones derived in \cite{schweinsberg2005random}
when we equal the fitness of \cite{schweinsberg2005random} and the rescaled invasion fitness
$s^{SD}=S_{aA}/f_a$ and when we have the equality $|S_{Aa}|/f_A=S_{aA}/f_a$ (in this case the first and third phases have the same duration,
$S_{aA}\log K/f_a$).\\

In \cite{brink2014multsweep}, the author generalized the model introduced in \cite{schweinsberg2005random} towards two neutral loci and used
similar methods to derive a corresponding statement for the genealogy of a sample taken at the end of the sweep.
If we however make the analogous comparison and try to
match our result for the adjacent geometry with the statement from \cite{brink2014multsweep},
we observe an interesting phenomenon:
the probabilities of the different types of ancestry only coincide if the birth rates of $a$- and $A$-individuals are the same,
that is, if $f_a=f_A$.
In biology, the fitness describes the ability to both survive and reproduce, and can be defined by the average contribution of an individual with a given genotype to the gene pool of
the next generation. Hence a mutation which affects the fitness of an individual
in a given environment can either act on the fertility ($f_\alpha$ in our model), or on the death rate,
intrinsic ($D_\alpha$) or by competition ($C_{\alpha,\alpha'}$), or on both.
Our result is comparable to that of \cite{schweinsberg2005random} if the mutation only affects the death rate
(and still if $s^{SD}=S_{aA}/f_a=|S_{Aa}|/f_A$).\\

In \cite{pfaffelhuber2007approximating}, instead of a birth and death process,
the authors modeled the population with a structured coalescent. It is shown that this process can be approximated by a marked Yule
tree where the different marks are realized by Poisson processes and indicate a recombination of one or two loci into the wild type
background. The impact of the third phase is taken into account by a certain refinement prior to the beginning of the coalescent which
leads to the same effect of splitting of the two neutral loci as it is seen here.
We again find similarities with our results when $f_A=f_a$.
In contrast, the techniques and precision used in \cite{pfaffelhuber2007approximating} yield that coalescent
events with $A$-individuals cannot be ignored, that is,
there are
neutral loci of different individuals from the sample which have the same type-$A$-ancestor. The structure of the sample is therefore
different from our results here.
Notice that it is also the case of the second approximate sampling formula stated in \cite{schweinsberg2005random}, which is more precise than
the first one.

\section{Dynamics of the sweep and couplings}
\label{sectioncoupling}

\subsection{Description of the three phases}\label{sectionP1P2P3}

We only need to focus on the trajectories of the population
process where the mutant allele $a$ goes to fixation and replaces the resident allele $A$. Champagnat has
described these trajectories in \cite{champagnat2006microscopic} and in particular divided the sweep into three phases
with distinct $A$- and $a$-population dynamics (see Figure \ref{fig3loci}).
In the sequel, $\eps$ will be a positive real number independent of $K$, as small as needed for the different approximations to hold.
Moreover, from this point onward we will write $N_\alpha$ (resp. $N_{\alpha\beta\gamma}$) instead of $N_\alpha^K$ (resp. $N^K_{\alpha\beta\gamma}$)
and $\P$ instead of $\P^{(K)}$ for the sake of readability.

\begin{figure}[h]
  \centering
 \includegraphics[width=10cm,height=4.5cm]{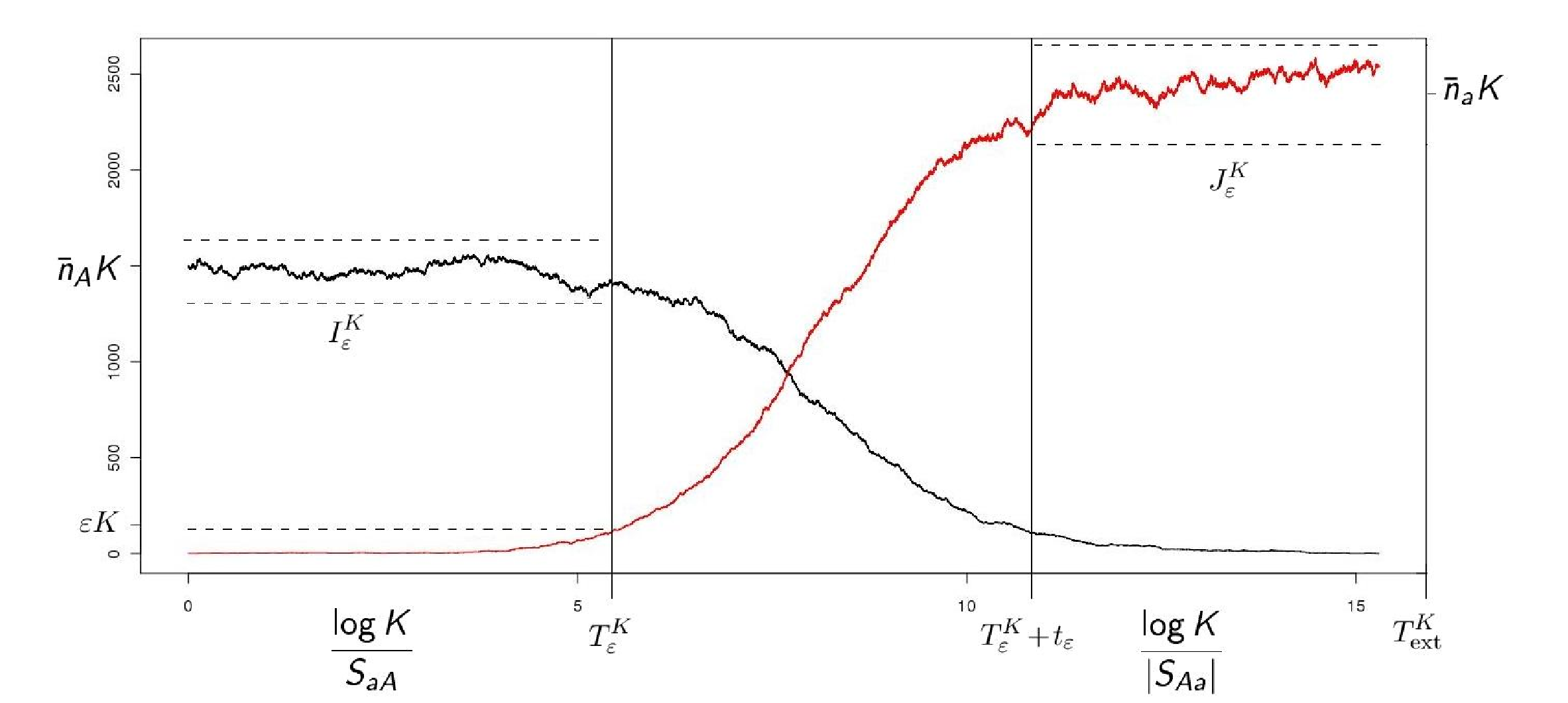}
\caption{The three phases of a selective sweep
The y-axis corresponds to population sizes ($A$ in black, $a$ in red), and the x-axis to the time.
In this simulation,
$K=1000$, $(f_A, f_a) =(2,3), D_\alpha=0.5, \alpha \in \mathcal{A}$, $C_{\alpha,\alpha'}=1, (\alpha,\alpha') \in \mathcal{A}^2$. 
We have also indicated
 some of the notations introduced in Section \ref{sectionP1P2P3}}
  \label{fig3loci}
\end{figure}

\subsection*{First phase} The resident population size stays close to its equilibrium value $\bar{n}_AK$ as long as the mutant population size
has not hit $\lfloor \eps K \rfloor $: if we introduce the finite subset of $\N$
\begin{equation} \label{compact}I_\eps^K:= \Big[K\Big(\bar{n}_A-2\eps \frac{C_{A,a}}{C_{A,A}}\Big),K\Big(\bar{n}_A+2\eps \frac{C_{A,a}}{C_{A,A}}\Big)\Big]
\cap \N, \end{equation}
and the stopping times $T^K_\eps$ and $S^K_\eps$, which denote respectively the hitting time of $\lfloor\eps K \rfloor$ by the mutant population size
 and the exit time of $I_\eps^K$ by the resident population size,
\begin{equation} \label{TKTKeps} T^K_\eps := \inf \{ t \geq 0, N_a(t)= \lfloor \eps K \rfloor \}\quad \text{and} \quad S^K_\eps := \inf \{ t \geq 0,
N_A(t)\notin I_\eps^K \},  \end{equation}
then we can deduce from \cite{champagnat2006microscopic} (see Equations (A.5) and (A.6) in \cite{smadi2014eco} for the details of the derivation) that
the events $\text{Fix}^K$, $\{T_\eps^K \leq S_\eps^K\}$ and $\{T_\eps^K<\infty\}$ are very close:
\begin{equation}\label{A6frommywork}
 \limsup_{K \to \infty}   \P^{(K)}(\{ T_\eps^K \leq S_\eps^K \} \bigtriangleup \text{Fix}^K )\leq c \eps,\quad \text{and} \quad
 \limsup_{K \to \infty}   \P^{(K)}(\{ T_\eps^K <\infty \} \bigtriangleup \text{Fix}^K )\leq c \eps,
\end{equation}
for a finite $c$ and $\eps$ small enough, where we recall convention \eqref{convP}.
In this context, $\bigtriangleup$ is the symmetric difference: for two sets $B$ and $C$,
$B \bigtriangleup C=(B\cap C^c) \cup (C\cap B^c)$.
From this point onwards, "first phase" will denote the time interval $[0,T_\eps^K]$ when the $a$-population size is smaller than $\lfloor \eps K \rfloor$.

\subsection*{Second phase} When $N_A$ and $N_a$ are of order $K$, the rescaled population process
$(N_A/K,N_a/K)$ is well
approximated by the Lotka-Volterra system \eqref{S}.
Moreover, under {Condition \eqref{defnbara}}
the system \eqref{S} has a unique attracting equilibrium $(0,\bar{n}_a)$ for initial conditions $z$ satisfying $z_a>0$,
where $\bar{n}_a$ has been defined in \eqref{defnbar}.
In particular,
if we introduce for $(n_A, n_a) \in \N^2$ the notation,
\begin{equation}\label{convP_} \P_{(n_A,n_a)}(.):=\P(.|N_A(0)=n_A,N_a(0)=n_a), \end{equation}
then Theorem 3 (b) in \cite{champagnat2006microscopic} implies:
\begin{equation}\label{result_champa} \underset{K \to \infty}{\lim}\ \underset{z \in \Gamma}{\sup}\
\P_{{(\lfloor z_AK\rfloor,\lfloor z_aK\rfloor)}}\Big(\underset{0\leq t \leq t_\eps,\alpha \in \mathcal{A}}
{\sup} \Big|\textstyle{\frac{N_\alpha(t)}{K}}-{n}^{(z)}_\alpha(t)\Big|\geq \delta \Big)=0,\end{equation}
for every $\delta>0$, where
\begin{align}
 \label{tetaCfini}&\Gamma:=\Big\{ z \in \R_+^\mathcal{A}, \lfloor z_AK \rfloor
\in I_\eps^K,
z_a\in [{\eps}/{2},\eps]  \Big\},\\
\label{deftepsz} &t_{\eps}(z):=\inf \big\{ s \geq 0,\forall t \geq s, {n}_A^{(z)}(t)\in
[0,\eps^2/2], {n}_a^{(z)}(t) \in[\bar{n}_a-\eps/2,\bar{n}_a+\eps/2] \big\},\\
 &t_\eps:=\sup \{ t_\eps(z),z \in \Gamma\}<\infty. \nonumber
\end{align}

In the sequel, "second phase" will denote the time interval $[T_\eps^K,T_\eps^K+t_\eps]$ when the population process
is close to the solution of the system \eqref{S}.

\subsection*{Third phase} Equation \eqref{result_champa} also implies that
\begin{equation} \label{boundsecondphase}
\underset{K \to \infty}{\lim}\
\P\Big( \textstyle{\frac{N_A(T_\eps^K+t_\eps)}{K}} \in [\omega_1,\omega_2],\Big|\textstyle{\frac{N_a(T_\eps^K+t_\eps)}{K}}
-\bar{n}_a\Big| \leq \eps , \
\Big|  \Big(\textstyle{\frac{N_A(T_\eps^K)}{K},\frac{N_a(T_\eps^K)}{K}} \Big) \in \Gamma \Big)=1,\end{equation}
where
\begin{equation}\label{defomega12} 2\omega_1:= \inf \ \{ n_A^{(z)}(t_\eps), z \in \Gamma \}>0, \quad \text{and}
\quad \omega_2:= 2 \sup \ \{ n_A^{(z)}(t_\eps), z \in \Gamma \}\leq \eps^2.  \end{equation}
The "third phase", which corresponds to the time interval $[T_\eps^K +t_\eps,T_{\text{ext}}^K]$, can be seen as the symmetric counterpart of the first phase, where the roles of $A$ and $a$ are interchanged: during the extinction
of the $A$-population, the $a$-population size stays close to its equilibrium value $\bar{n}_aK$.

Let us introduce the positive real number
 $M'':=3+(f_a+C_{a,A})/C_{a,a}$ and the finite subset of $\N$
\begin{equation} \label{compact2}J_\eps^K:= \Big[K\Big(\bar{n}_a-M''\eps \Big),K\Big(\bar{n}_a+M''\eps \Big)\Big]\cap \N. \end{equation}
The times
$T_u^{(K,A)}$ and ${S}^{(K,a)}_\eps$
are two stopping times for the process restarted after the second phase and
denote respectively the hitting times of $\lfloor uK \rfloor$ by the $A$-population
for $u \in \R_+$,
and the exit time of $J_\eps^K$ by the $a$-population during the third phase,
\begin{equation} \label{T0K} T_u^{(K,A)} := \inf \{ t \geq 0, N_A(T_\eps^K+t_\eps+t)= \lfloor uK \rfloor \},\quad
 {S}^{(K,a)}_\eps := \inf \{ t \geq0, N_a(T_\eps^K+t_\eps+t)\notin J_\eps^K \}.  \end{equation}
If we define the event
\begin{equation}\label{defNepsK} \mathcal{N}_\eps^K:=\{ T_\eps^K \leq S_\eps^K \} \cap \Big\{  \textstyle{\frac{N_A(T_\eps^K+t_\eps)}{K} }
\in [\omega_1,\omega_2],\Big|\textstyle{\frac{N_a(T_\eps^K+t_\eps)}{K}} -\bar{n}_a\Big| \leq \eps  \Big\}, \end{equation}
we get from {the proof of Lemma 3 in} \cite{champagnat2006microscopic} that for a finite $c$ and $\eps$ small enough,
\begin{equation}\label{diffprobaphase3}
\limsup_{K \to \infty}\Big\{\P(\text{Fix}^K
\bigtriangleup [\mathcal{N}_\eps^K \cap \{T_0^{(K,A)} <T_\eps^{(K,A)}\wedge {S}^{(K,a)}_\eps \} ])+\P(\text{Fix}^K
\bigtriangleup [\mathcal{N}_\eps^K \cap \{T_0^{(K,A)} <T_\eps^{(K,A)} \} ] )\Big\}\leq c \eps. \end{equation}
To summarize, the fixation event $\text{Fix}^K$ is very close to the following succession of events:
\begin{enumerate}
 \item[$\bullet$] The $a$-population size hits $\lfloor \eps K\rfloor$ before the $A$-population size has escaped the vicinity of its equilibrium $I_\eps^K$ (first phase)
 \item[$\bullet$] The rescaled {population} process $N/K$ is close to the deterministic competitive Lotka-Volterra system during the second phase
 \item[$\bullet$] The $A$-population size gets extinct before hitting $\lfloor \eps K\rfloor$ and before the
 $a$-population size has escaped the vicinity of its equilibrium $J_\eps^K$ (third phase)
\end{enumerate}

\subsection{Couplings for the first and third phases}

We are interested in the law of the neutral genealogies on the event $\text{Fix}^K$. Equations \eqref{A6frommywork} and \eqref{diffprobaphase3}
imply that it is enough to concentrate our attention on the event $\mathcal{N}_\eps^K \cap \{T_0^{(K,A)} <T_\eps^{(K,A)}\}$,
but the dynamics of the population process $N$ conditionally on this event is complex to study.
Indeed it boils down to studying the dynamics of a process conditioned on a future event ($\{ T_\eps^K \leq S_\eps^K \}$
for the first phase and $\{T_0^{(K,A)} <T_\eps^{(K,A)}\}$ for the third one).
Hence the idea is to couple the population process with two processes, $\tilde{N}$ and $\tilde{\tilde{N}}$,
whose laws are easier to study. 
These processes will satisfy:
\begin{equation}\label{couplage2.1}\underset{K \to \infty}{\limsup} \ \P(\{\exists t \leq T_{\eps}^K, N(t) \neq \tilde{N}(t)\}, T_{\eps}^K<\infty)\leq c \eps.\end{equation}
and
\begin{multline}\label{couplage2.2}\underset{K \to \infty}{\limsup} \
\P(\{\exists \ 0\leq  t-(T_\eps^K+t_\eps) \leq T_0^{(K,A)}, N(t) \neq \tilde{\tilde{N}}(t)\},
 T_0^{(K,A)}<T_\eps^{(K,A)}
| \mathcal{N}_{\eps}^K<\infty)\leq c \eps.\end{multline}

Let $\alpha$ be in $\mathcal{A}$ and $n$ be in $\N^\mathcal{E}$. Denote $n^{(\alpha)}$ the $\alpha$ component of the
population state:
\begin{equation}\label{defn(alpha)}
 n^{(\alpha)}=\sum_{(\beta,\gamma) \in \mathcal{B}\times \mathcal{C} }n_{\alpha \beta \gamma}e_{\alpha \beta\gamma},
\end{equation}
where $(e_{\alpha \beta \gamma}, (\alpha, \beta ,\gamma)\in \mathcal{E})$ is the canonical basis of $\R^\mathcal{E}$.
We are now able to introduce a process needed to describe the couplings:

\begin{defi}\label{defMR}
We denote by \textup{Moran process of type $\alpha$ with recombination $r_2$} a
 process $MR_\alpha^{(n^{(\alpha)})}$
with values in $\N^{\alpha \times \mathcal{B}\times \mathcal{C}}$, initial state $n^{(\alpha)}$, and the following dynamics:
\begin{enumerate}
 \item[$\bullet$] After an exponential time with parameter $f_\alpha\bar{n}_\alpha K$ we
pick uniformly and with replacement three individuals and draw a Bernoulli variable $R$ with parameter $r_2$
\item[$\bullet$] The first individual dies, the second one gives birth to an individual carrying its alleles at loci $SL$ and $N1$, 
the third one is the potential second parent
\item[$\bullet$] If $R=0$, there is
no recombination and the allele at locus $N2$ of the newborn
is also inherited from the second individual; if $R=1$ there is a recombination {between} $N1$ and $N2$
and the newborn inherits its second neutral allele from the third individual
\item[$\bullet$] We again draw an exponential variable with parameter $f_\alpha\bar{n}_\alpha K$ and restart the procedure
\end{enumerate}
\end{defi}

\noindent \textbf{Coupling with $\tilde{N}$}: $N$ and $\tilde{N}$ are equal up to time $S_\eps^K$;
after this time the $A$ individuals in the
population process $\tilde{N}$ follow a Moran process with recombination independent of the $a$-individuals.
Let $\underline{\mathfrak{b}\mathfrak{c}}$ be in $\mathcal{B}\times \mathcal{C}$.
We let the $a$-population evolve as if the $a$-individuals were interacting with $\tilde{N}_A(s)$ individuals with genotype 
$A\underline{\mathfrak{b}\mathfrak{c}}$:
\begin{multline}\label{deftildeN}
 \tilde{N}(t) = \mathbf{1}_{t < S_\eps^K}N(t)+
\mathbf{1}_{t \geq  S_\eps^K} \Big(MR_A^{({N}^{(A)}(S_\eps^K))}(t-S_\eps^K)+
 \underset{(\beta,\gamma)\in \mathcal{B}\times\mathcal{C}}{ \sum} e_{a\beta\gamma}\\
\int_{S_\eps^K}^t\int_{\R_+} \Big[
 Q_{\beta\gamma}^{(1)}(ds,d\theta) \mathbf{1}_{\{0<\theta\leq b^K_{a \beta\gamma}(\tilde{N}_{A}(s-) 
 e_{A\underline{\mathfrak{b}\mathfrak{c}}},\tilde{N}^{(a)}(s^-))\}}
 - 
 Q_{\beta \gamma}^{(2)}(ds,d\theta)
 \mathbf{1}_{\{0<\theta
\leq d^K_{a \beta\gamma}(\tilde{N}_{A}(s^-) e_{A\underline{\mathfrak{b}\mathfrak{c}}},\tilde{N}^{(a)}(s^-))\}}\Big]\Big),
\end{multline}
where $MR_A^{({N}^{(A)})}$ has been defined in Definition \ref{defMR} and $(Q_{\beta \gamma}^{(i)}, i \in \{1,2\},(\beta,\gamma)\in \mathcal{B}\times\mathcal{C} )$ 
are independent Poisson Point processes with
density $dsd\theta$, also independent of $MR^{({N}^{(A)})}$.
The reason for the construction of such a coupling is that we need to control the $A$-population size and the number of
births of $A$-individuals during the first phase in Section \ref{technicalsection}. With the process $\tilde{N}$ such control is achieved easier.

\noindent \textbf{Coupling with $\tilde{\tilde{N}}$}:
 we assume that $\mathcal{N}_\eps^K$ from \eqref{defNepsK} holds;
 $N$ and $\tilde{\tilde{N}}$ are equal up to time $T_\eps^K+t_\eps+S_\eps^{(K,a)}\wedge T_\eps^{(K,A)}$.
Then the $a$-individuals in the
population process $\tilde{\tilde{N}}$ follow a Moran process with recombination independent of the $A$-individuals, and
each $A\beta\gamma$-population evolves as a birth and death process with
individual birth and death rates $f_A$ and $f_A+|S_{Aa}|$, independent of the $a$-individuals and the $A\beta'\gamma'$-populations with
$(\beta,\gamma)\neq (\beta',\gamma')$:
\begin{multline}\label{deftildetildeN}
 \tilde{\tilde{N}}(T_\eps^K+t_\eps+t) = \mathbf{1}_{t < S_\eps^{(K,a)}}N(T_\eps^K+t_\eps+t)+ \mathbf{1}_{t \geq  S_\eps^{(K,a)}} \Big(
MR_a^{({N}^{(a)}(S_\eps^{(K,a)}))}(t-S_\eps^{(K,a)})+\\
 +\underset{(\beta,\gamma)\in \mathcal{B}\times\mathcal{C}}{ \sum}
e_{A\beta\gamma}\Big[\int_{T_\eps^K+t_\eps+S_\eps^{(K,a)}}^{T_\eps^K+t_\eps+t}\int_{\R_+}
 Q_{\beta\gamma}(ds,d\theta)
\Big\{\mathbf{1}_{\{0<\theta
\leq f_A\tilde{\tilde{N}}_{A\beta\gamma}(s^-)\}}
-\mathbf{1}_{\{0<\theta
- f_A\tilde{\tilde{N}}_{A\beta\gamma}(s^-)\leq  (f_A+|S_{Aa}|)\tilde{\tilde{N}}_{A\beta\gamma}(s^-) \}}\Big\}\Big]\Big),
\end{multline}
where $MR_a^{({N}^{(a)})}$ has been defined in Definition \ref{defMR} and is independent of the sequence of
independent Poisson
measures $(Q_{\beta\gamma}, (\beta,\gamma) \in \mathcal{B}\times\mathcal{C})$, with intensity $dsd\theta$.
The $a$-population size and the number of
births of $a$-individuals will be easy to control for the process $\tilde{\tilde{N}}$ during the third phase, and again we will need such control in 
Section \ref{technicalsection}.\\

Inequality \eqref{couplage2.1} follows from \eqref{A6frommywork}.
Moreover, from the proof of Lemma 3 in \cite{champagnat2006microscopic} we know that
$$\liminf_{K \to \infty} \P(T_0^{(K,A)}<T_\eps^{(K,A)}\wedge S_\eps^{(K,a)} |\mathcal{N}_\eps^K)\geq 1-c\eps $$
for a finite $c$ and $\eps$ small enough. Adding \eqref{diffprobaphase3} we get that \eqref{couplage2.2} is also satisfied.
Hence we will study the processes $\tilde{N}$ and $\tilde{\tilde{N}}$ and deduce properties of the dynamics of the process $N$ during
the first and third phases.

\section{Proofs of the main results}
 \label{sectionproofmain}

As the proof of Theorem \ref{mainresult2} is simpler than this of Theorem \ref{mainresult}, and follows essentially the same ideas, we only prove
Theorem \ref{mainresult}.
 
\subsection{Events impacting the genealogies in each phase}\label{subsec:eventsallphases}

Let us now summarize the results on the genealogies for the three successive phases of the sweep that we will derive in
Sections \ref{firstphase} and \ref{secondphase}.\\

\noindent  \textbf{First phase:} As explained in the previous section, we work with the process $\tilde{N}$ to study the first phase.
Let us introduce the jump times of $\tilde{N}$:
\begin{equation}\label{deftpssauts}
\tau_0^K=0 \quad \text{and} \quad \tau_m^K= \inf \{ t> \tau_{m-1}^K, \tilde{N}(t)\neq \tilde{N}(\tau_{m-1}^K) \}, \quad m \geq 1.
\end{equation}
The number of jumps during the first phase is denoted by $J^K(1)$:
\begin{equation}\label{defJK1}
 J^K(1):= \inf \{ m \in \N, \tilde{N}_a(\tau_m^K)=\lfloor \eps K \rfloor \}.
\end{equation}
Coalescence and recombination events are defined as follows (see Figure \ref{coalevent}):

\begin{defi}\label{defcoalreco}
Sample two distinct individuals at time $\tau_m^K$ and denote $\alpha\beta\gamma$
 and $\alpha'\beta'\gamma'$ their type.

We say that $\beta$ and $\beta'$ coalesce at time $\tau_m^K$ if they are carried by two distinct individuals at time $\tau_{m}^K$ and 
by the same
individual at time $\tau_{m-1}^K$. Seen forwards in time it corresponds to a birth and hence a copy of the neutral allele. Seen backwards in time it corresponds to the
fusion of two neutral alleles into one, carried by one parent of the newborn. We define in the same way coalescent events at locus $N2$
(resp. loci $N1$ and $N2$) for alleles $\gamma$ and $\gamma'$ (resp. allele pairings $(\beta, \gamma)$ and $(\beta', \gamma')$).

We say that $\beta$ (and/or $\gamma$) recombines at time $\tau_m^K$ from the $\alpha$- to the $\alpha'$-population if the individual carrying the allele $\beta$ (and/or $\gamma$)
 at time $\tau^K_m$ is a newborn,
 carries the allele $\alpha$ inherited from it first parent, and has inherited its allele $\beta$ (and/or $\gamma$) from a different individual
carrying allele $\alpha'$.
\end{defi}

We are only interested in
recombinations which entail new associations of alleles.
In particular we will not consider the simultaneous recombinations of a pair $(\beta,\gamma)$
within the $\alpha$-population.

\begin{figure}[h]
  \centering
 \includegraphics[width=7cm,height=3cm]{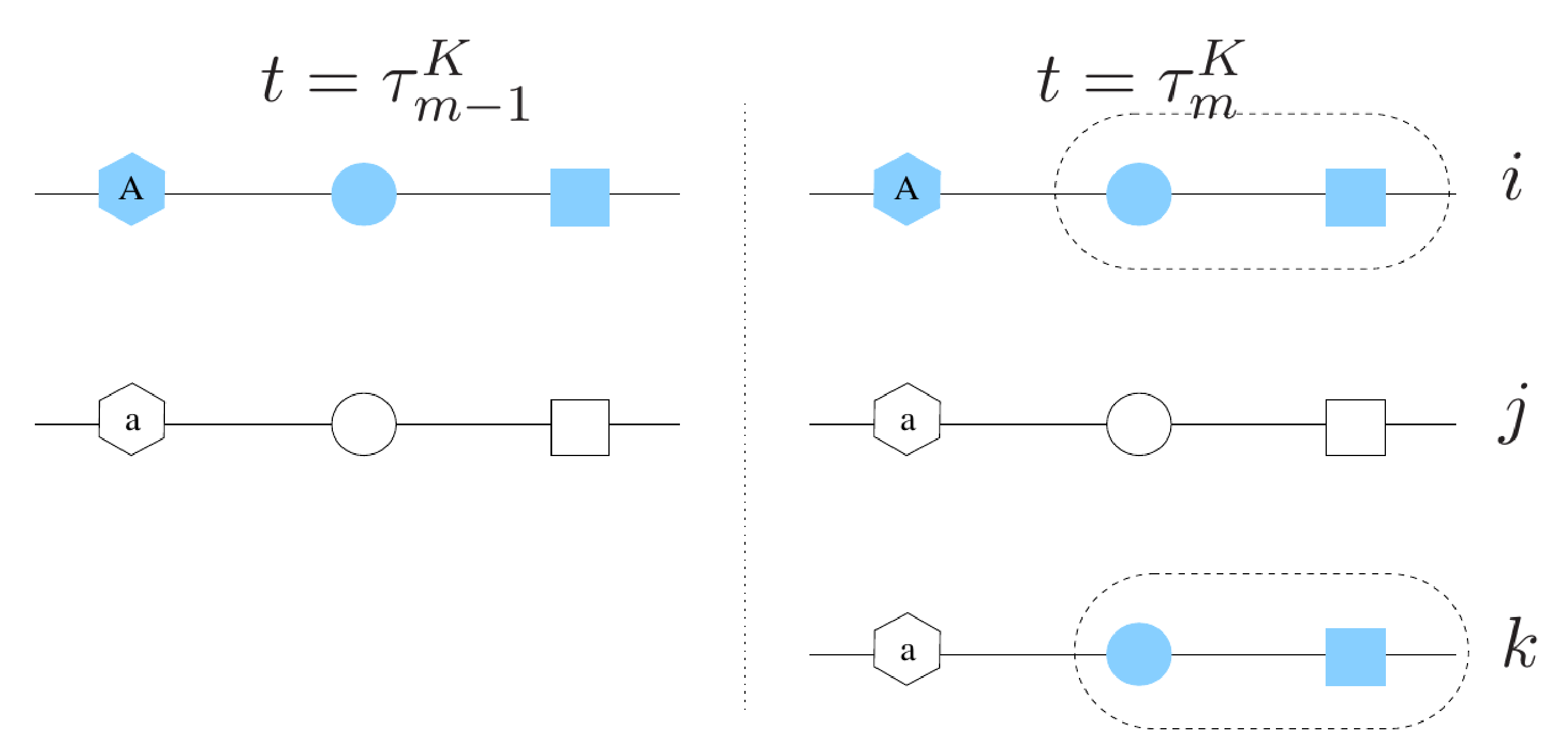}
\caption{Illustration of Definition \ref{defcoalreco}: the newborn (individual $k$) has inherited the selected allele from its
"white" parent and the two neutral alleles from its "blue" parent; hence the encircled neutral loci (of individuals $i$ and $k$) coalesce at time $\tau_m^K$.
In terms of recombinations, the two neutral loci of the newborn individual recombine at time $\tau_m^K$ from the $a$- to the
$A$-population}
  \label{coalevent}
\end{figure}

Let us now describe the genealogical scenarios which modify the ancestral relationships between the neutral alleles of one individual
and occur with positive probability when $K$ is large.
We first focus on the first phase and pick uniformly an individual
$i$ from the $a$-population at time $\tilde{T}_\eps^K$. We introduce:
\label{defgenealogy}\begin{align*}
\begin{aligned}
 NR(i)^{(1)}: &\quad \text{there is no recombination into the $A$-population affecting $(i,1)$ or $(i,2)$}\\
&\quad \text{and both neutral loci of the $i$-individual
 originate from the first mutant,}\\
 R2(i)^{(1)}: &\quad \text{only the neutral allele $(i,2)$ is affected by a recombination with the $A$-population,}\\
 &\quad \text{hence $(i,1)$ originates from the first mutant and $(i,2)$ from an $A$-individual,}\\
 R12(i)^{(1)}: &\quad \text{one recombination between $SL$ and $N1$ from the $a$- into the $A$-population occurs}\\
 &\quad \text{and both neutral alleles $(i,1)$ and $(i,2)$ originate from the same $A$-individual,}\\
 [2,1]^{rec}_{A,i}: &\quad
 \text{first (backwards in time) $(i,2)$ recombines into the A-population, then $(i,1)$}\\
 &\quad \text{recombines into the A-population and connects to a different individual than $(i,2)$.}\\
 [12,2]^{rec}_{A,i} : &\quad
 \text{first (backwards in time) the tuple $\{(i,1),(i,2)\}$ recombines into the $A$-population, }\\
&\quad \text{then a second recombination splits the two neutral loci inside the $A$-population.}\\
 R1|2(i)^{(1,ga)}: &\quad [2,1]^{rec}_{A,i}\cup [12,2]^{rec}_{A,i} \text{ (see Figure \ref{schemagendefevent})}
\end{aligned}
\end{align*}

\begin{figure}[h]
  \centering
 \includegraphics[width=6cm,height=3.5cm]{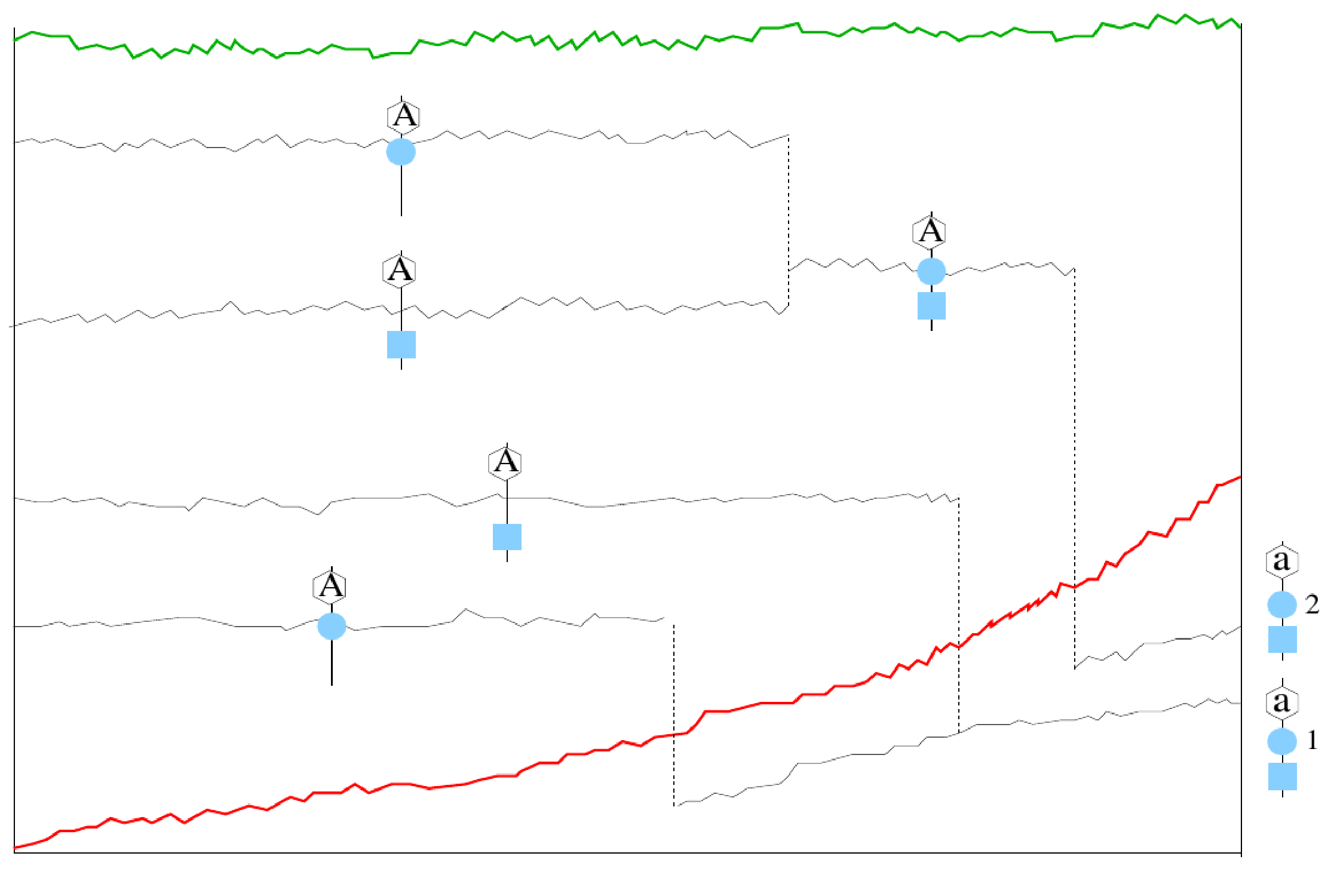}
\caption{Illustration of events $[2,1]^{rec}_{A,i}$ (individual $1$) and $[12,2]^{rec}_{A,i}$ (individual $2$)}
  \label{schemagendefevent}
\end{figure}

Finally, we introduce a conditional probability for the process $\tilde{N}$:
\begin{equation}\label{defP1}
 \P^{(1)}(.)=\P(.|J^K(1)<\infty),
\end{equation}
where $J^K(1)$ has been defined in \eqref{defJK1}.
Hence, recalling the definition of $(q_1,q_2,q_3)$ in \eqref{defq1q2} we will prove in Section \ref{firstphase}:

\begin{pro}[Neutral genealogies during the first phase]\label{prop:1phase_probg1}
Let $i$ be an $a$-individual sampled uniformly at the end of the first phase (time $\tilde{T}_\eps^K$). Under Assumption \ref{asstotale}, there
exist two finite constants $c$ and $\eps_0$ such that for every $\eps\leq \eps_0$,
\begin{multline*}
 \limsup_{K \to \infty} \left\{\Big|\P^{(1)}(NR(i)^{(1)})-q_1 q_2\Big|+ \Big|\P^{(1)}(R2(i)^{(1)})-q_1
(1-q_2)\Big|\right.\\
 \left. + \Big|\P^{(1)}(R12(i)^{(1)})-q_3\Big|+ \Big|\P^{(1)}(R1|2(i)^{(1,ga)})-
(1-q_1 - q_3)\Big|\right\}\leq  c\eps.
\end{multline*}
\end{pro}

For large $K$, the sum of the four probabilities of Proposition \ref{prop:1phase_probg1} equals one up to a constant times $\eps$.
 Hence, in the limit we only
observe the events described on page
\pageref{defgenealogy}.
The probabilities of the first two events are quite intuitive: broadly speaking, the probability to have
no recombination at a birth
event is $1-r_1-r_2$, the birth rate is $f_a$ and the duration of the first phase is $\log K/S_{aA}$. Hence under $\P^{(1)}$,
the probability of the event $NR(i)^{(1)}$ is approximately
$$ (1-(r_1+r_2))^{f_a\log K/S_{aA}}\sim \exp(-(r_1+r_2))^{f_a\log K/S_{aA}}=q_1q_2. $$
Similarly the probability to have no recombination between $SL$ and $N1$ is close to $q_1$ and subtracting the probability
of $NR(i)^{(1)}$ we get this of $R2(i)^{(1)}$. The probabilities of $R12(i)^{(1)}$ and $R1|2(i)^{(1,ga)}$ are more involved.
 The proofs rely on a fine study of the different possible scenarios.\\

\noindent \textbf{Second phase:} We work with the process $N$ to study the second phase. The latter one
has a duration of order $1$,
and the recombination probabilities are negligible with respect to one (Condition \eqref{assrK}).
 Consequently, no event impacting the genealogies of the neutral loci occurs during the
second phase. More precisely, let us sample uniformly two distinct $a$-individuals $i$ and $j$ at the end of the second phase
(time $T_\eps^K+t_\eps$) and introduce
the events:\label{defgenealogy3}
\begin{align*}
\begin{aligned}
  NR(i)^{(2)}: &\quad \text{there is no recombination affecting $(i,1)$ or $(i,2)$,}\\
 NC(i,j)^{(2)}: &\quad \text{there is no coalescence  between the neutral genealogies of $i$ and $j$.}
\end{aligned}
\end{align*}

Then we have the following result, which will be proven in Section \ref{secondphase}.

\begin{pro}[Neutral genealogies during the second phase]\label{prosecondphase}
Let $i$ and $j$ be two distinct $a$-individuals sampled uniformly at the end of the second phase (time $T_\eps^K+t_\eps$).
 Then under Assumption \ref{asstotale},
 \begin{equation*}\label{reconegli2}
 \lim_{K \to \infty}  \P(NR(i)^{(2)} \cap NC(i,j)^{(2)}|T_\eps^K\leq S_\eps^K)=1.
 \end{equation*}
\end{pro}

\noindent \textbf{Third phase:} Finally, we focus on the process $\tilde{\tilde{N}}$.
When $K$ is large, there is only one event occurring with positive probability during the
third phase which may modify the ancestry of the neutral alleles of an individual $i$ sampled at the end of the sweep in the adjacent
geometry:

\begin{align}\label{defgenealogy4}
\begin{aligned}
 R2(i)^{(3,ga)}: &\quad \text{a recombination between loci $N1$ and $N2$ occurs and separates}\\
 &\quad \text{$(i,1)$ and $(i,2)$ within the $a$-population,}
\end{aligned}
\end{align}
Indeed, if we also define the events
\label{defgenealogy5}\begin{align*}
\begin{aligned}
 NR(i)^{(3)}: &\quad \text{there is no recombination affecting $(i,1)$ or $(i,2)$ and they both}\\
&\quad \text{originate
  from the same $a$-individual at the end of the second phase }\\
   NC(i,j)^{(3)}: &\quad \text{defined as $NC(i,j)^{(2)}$ for two distinct individuals sampled}\\
&\quad \text{uniformly at the end of the sweep},
\end{aligned}
\end{align*}
and the conditional probability for the process $\tilde{\tilde{N}}$::
\begin{equation}
 \P^{(3)}(.):=\P(.|\mathcal{N}_\eps^K ,\tilde{\tilde{T}}_0^{(K,A)}<\tilde{\tilde{T}}_\eps^{(K,A)}),
\end{equation}
where $\tilde{\tilde{T}}_0^{(K,A)}$ and $\tilde{\tilde{T}}_\eps^{(K,A)}$ are the analogs of
 ${T}_0^{(K,A)}$ and $T_\eps^{(K,A)}$ (defined in \eqref{T0K}) for the process $\tilde{\tilde{N}}$,
then we will prove in Section \ref{secondphase}:

\begin{pro}[Neutral genealogies during the third phase]\label{prothirdphase}
Let $i$ and $j$ be two distinct $a$-individuals sampled uniformly at the end of the sweep.
Under Assumption \ref{asstotale}, there
exist two finite constants $c$ and $\eps_0$ such that for every $\eps\leq \eps_0$,
$$
\underset{K \to \infty}{\limsup}\left\{\Big| \P^{(3)}(R2(i)^{(3,ga)}) -
( 1- \bar{q}_2 ) \Big|+\Big| \P^{(3)}(NR(i)^{(3)}) -
 \bar{q}_2 \Big|+\Big|\P^{(3)}(NC(i,j)^{(3)})-1\Big| \right\}\leq c \sqrt{\eps}.$$
\end{pro}
In particular, there is no recombination with the $A$-population during the third phase. As for the probabilities of the first two events 
in the Proposition \ref{prop:1phase_probg1} this result is quite intuitive, as the duration of the third phase is close to $\log K/|S_{Aa}|$.\\

\noindent \textbf{Independence:} Finally we again consider the population process $N$ and state a proposition which enables us to give the
statement of Theorem \ref{mainresult} independently for all sampled individuals, that is, jointly for the whole sample.
To this aim, let us introduce a partition $ \Theta^{(K,1)}_d \in \mathcal{P}_d^*$ which is the analog of $\Theta^K_d$ where
the $d$ individuals are sampled at the end of the first phase and not at the end of the sweep.
Recall Definitions \ref{defdelta} and \ref{defpartition},
and denote by $|R2^{(3,ga)}|_d$ (resp. $|NR^{(3)}|_d$) the number of $a$-individuals in a $d$-sample taken at the end of the sweep
whose neutral alleles originate from two distinct $a$-individuals (resp. from the same $a$-individual)
at the beginning of the third phase.
Then we have the following result:

\begin{pro}\label{proindi}
Let Assumption \ref{asstotale} hold. Then there
exist two finite constants $c$ and $\eps_0$ such that for every $\eps\leq \eps_0$,
the ancestral relationships of a $d$-sample
taken at the end of the first phase (time $T_\eps^K$) satisfy for every  $(m_k,1\leq k \leq 4) \in \Z_+^4$:
\begin{multline*}
 \limsup_{K\rightarrow \infty} \Big| \P( |\Theta^{(K,1)}_d|_k=m_k,1\leq k \leq 4|T_\eps^K\leq S_\eps^K)\\
- \mathbf{1}_{\{m_1+m_2+m_3+m_4=d\}}\frac{d!}{m_1!m_2!m_3!m_4!}(q_1q_2)^{m_1}(q_1(1-q_2))^{m_2}q_3^{m_3}(1-q_1-q_3)^{m_4} \Big|
\leq c {\eps}.
\end{multline*}
In the same way, the neutral genealogy of a $d$-sample taken at the end of the sweep satisfies for every  $(m_k,1\leq k \leq 2) \in \Z_+^2$:
\begin{equation*}
 \limsup_{K\rightarrow \infty} \Big| \P((|R2^{(3,ga)}|_d,|NR^{(3)}|_d)=(m_1,m_2)|\mathcal{N}_\eps^K) -\mathbf{1}_{\{m_1+m_2=d\}}
 \frac{d!}{m_1!m_2!}(1-\bar{q}_2)^{m_1}\bar{q}_2^{m_2} \Big| \leq c{\eps}.
\end{equation*}
\end{pro}
The Proposition \ref{proindi} is a key result: we only need to focus on individual neutral genealogies to get general results
on the genealogy of a $d$-sample with respect to the neutral loci.
 It will be proven in Section \ref{proofindep}.

\subsection{Proof of Theorem \ref{mainresult}}

Let $i$ be an individual sampled uniformly at the end of the sweep.
The idea of the proof is the following: in a first step, we list certain compositions of coalescent and recombination events leading to
specific ancestral relationships {which could be described by blocks
of a partition of $\Delta_d$}. Then we approximate the probabilities of the described events and prove that these probabilities sum to one up to a constant times $\sqrt{\eps}$
for some fixed small $\eps$. This shows that in the limit for large $K$ the
neutral genealogy of the individual $i$ belongs to those described on page \pageref{defgenealogy} with a probability close to one.
In a second step we use Proposition \ref{proindi} to treat the neutral genealogies
of the $d$ sampled individuals independently.

\begin{enumerate}
 \item[i)] We consider two possible trajectories such that the alleles at both neutral loci originate from the mutant:
either the two neutral loci separate inside the $a$-population during the third phase and coalesce during the first phase,
 or they stay in the $a$-population and do not separate during the whole sweep (see individual $1$ in Figure
\ref{schemagen}):

\begin{multline*}
 \Big(R2(i)^{(3,ga)}\cap NR(i1)^{(2)} \cap NR(i2)^{(2)} \cap NC(i1,i2)^{(2)} \cap   [ NR(i1)^{(1)}\sqcup R2(i1)^{(1)}]  \cap NR(i2)^{(1)}\Big)\\
\bigsqcup \Big(NR(i)^{(3)} \cap NR(i)^{(2)} \cap NR(i)^{(1)} \Big),
\end{multline*}
where {$\sqcup$ is the disjoint union and} we denote by $i1$ and $i2$ the labels of the parents of the first and second neutral loci of $i$, respectively,
 at the end of the second phase (the way we label the $a$-individuals has no
importance as they are exchangeable).
 \item[ii)] We consider two possible trajectories such that $(i,1)$ originates from the mutant and $(i,2)$ originates
 from some $A$-individual
\begin{multline*}  \Big(R2(i)^{(3,ga)}\cap  NR(i1)^{(2)} \cap NR(i2)^{(2)} \cap NC(i1,i2)^{(2)} \cap  [NR(i1)^{(1)}\cup R2(i1)^{(1)}] \\
 \cap [ R12(i2)^{(1)}\sqcup R1|2(i2)^{(1)}\sqcup R2(i2)^{(1)}] \Big)
\bigsqcup \Big(NR(i)^{(3)} \cap NR(i)^{(2)} \cap R2(i)^{(1)} \Big).\end{multline*}
The first bracket considers a separation of the two neutral loci during the third phase.
As a consequence, the fate of the first neutral locus of individual $i2$ during the first phase has no consequence on the neutral genealogy of $i$.
This is why we consider the event $\{R12(i2)^{(1)}\sqcup R1|2(i2)^{(1)}\sqcup R2(i2)^{(1)}\}$ and not only $\{ R2(i2)^{(1)} \}$.
The second bracket corresponds to individual $2$ in Figure
\ref{schemagen}.
\item[iii)] We consider one possible trajectory such that $(i,1)$ originates from some $A$-individual and $(i,2)$ originates
from the mutant (see individual $5$ in Figure
\ref{schemagen})
$$   R2(i)^{(3,ga)}\cap  NR(i1)^{(2)} \cap NR(i2)^{(2)} \cap NC(i1,i2)^{(2)}
\cap  [R12(i1)^{(1)}\sqcup R1|2(i1)^{(1,ga)}] \cap  NR(i2)^{(1)}$$
 \item[iv)] We consider one possible trajectory such that $(i,1)$ and $(i,2)$ originate from the same $A$-individual
 (see individual $3$ in Figure
\ref{schemagen})
$$ NR(i)^{(3)} \cap NR(i)^{(2)} \cap R12(i)^{(1)} $$
 \item[v)] Finally, we consider two possible trajectories such that $(i,1)$ and $(i,2)$ originate from
distinct $A$-individuals (see individual $4$ in Figure
\ref{schemagen} for the second bracket):
\begin{multline*}  \Big(R2(i)^{(3,ga)}\cap NR(i1)^{(2)} \cap NR(i2)^{(2)} \cap NC(i1,i2)^{(2)}\cap  [R12(i1)^{(1)}\sqcup R1|2(i1)^{(1)}] \\
\cap [ R12(i2)^{(1)}\sqcup R1|2(i2)^{(1)}\cup R2(i2)^{(1)}] \Big)
\bigsqcup \Big(NR(i)^{(3)}  \cap NR(i)^{(2)} \cap R1|2(i)^{(1,ga)} \Big).\end{multline*}
\end{enumerate}
\vspace{.5cm}

Thanks to \eqref{A6frommywork}, and \eqref{diffprobaphase3} to \eqref{couplage2.2} we know that for all non negligible
measurable events $C^{(1)}$, $C^{(2)}$ and $C^{(3)}$ occurring during the first, second and third phase respectively,
{
\begin{equation}\label{decompoC123}
 \P( C^{(1)}, C^{(2)}, C^{(3)}, \text{Fix}^K )=
\P( C^{(1)}, C^{(2)}, C^{(3)}, \mathcal{N}_\eps^K,\{T_0^{(K,A)}<T_\eps^{(K,A)} \wedge S_\eps^{(K,A)}\}) +O_K(\eps)
\end{equation}
where $O_K(\eps)$ is a function of $K$ and $\eps$ satisfying
\begin{equation}\label{defOKeps}\limsup_{K \to \infty}|O_K(\eps)|\leq c \eps,\end{equation}
for $\eps\leq \eps_0$
where $\eps_0$ and $c$ are finite.
Using the same inequalities we can decompose the right hand side of \eqref{decompoC123} as follows
\begin{multline*}
 \P(C^{(1)},\{T_\eps^K<\infty\})+\P(C^{(2)},\{  \textstyle{\frac{N_A(T_\eps^K+t_\eps)}{K} }
\in [\omega_1,\omega_2],|\textstyle{\frac{N_a(T_\eps^K+t_\eps)}{K}} -\bar{n}_a| \leq \eps  \}|C^{(1)},\{T_\eps^K\leq S_\eps^K\})\\
+ \P(C^{(3)},\{T_0^{(K,A)}<T_\eps^{(K,A)} \wedge S_\eps^{(K,A)}\}|C^{(1)},C^{(2)}, \mathcal{N}_\eps^K)+O_K(\eps).
\end{multline*}
Then from \eqref{couplage2.1} we get
$$ \P(C^{(1)},\{T_\eps^K<\infty\})=\P^{(1)}(\tilde{C}^{(1)})\P(T_\eps^K<\infty)+O_K(\eps) ,$$
from \eqref{boundsecondphase}
$$ \P(C^{(2)},\{  \textstyle{\frac{N_A(T_\eps^K+t_\eps)}{K} }
\in [\omega_1,\omega_2],|\textstyle{\frac{N_a(T_\eps^K+t_\eps)}{K}} -\bar{n}_a| \leq \eps  \}|C^{(1)},\{T_\eps^K\leq S_\eps^K\})=
\P(C^{(2)}
|C^{(1)},\{T_\eps^K\leq S_\eps^K\})+O_K(\eps),$$
and from \eqref{diffprobaphase3} and \eqref{couplage2.2}
$$  \P(C^{(3)},\{T_0^{(K,A)}<T_\eps^{(K,A)} \wedge S_\eps^{(K,A)}\}|C^{(1)},C^{(2)}, \mathcal{N}_\eps^K)=
\P^{(3)}(\tilde{\tilde{C}}^{(3)}|C^{(1)}, C^{(2)})+O_K(\eps),$$
where $\tilde{C}^{(1)}$ (resp. $\tilde{\tilde{C}}^{(3)}$) corresponds to the event $C^{(1)}$ (resp. $C^{(3)}$)
 expressed in terms of the process $\tilde{N}$
(resp. $\tilde{\tilde{N}}$).
Putting everything together we finally obtain
\begin{equation}\label{decomproba123}
 \P( C^{(1)}, C^{(2)}, C^{(3)}| \text{Fix}^K )=
\P^{(3)}(\tilde{\tilde{C}}^{(3)}|C^{(1)}, C^{(2)})\P( C^{(2)}|C^{(1)},\{T_\eps^K\leq S_\eps^K\})\P^{(1)}(\tilde{C}^{(1)})+O_K(\eps).
\end{equation}}
By applying Propositions \ref{prop:1phase_probg1}, \ref{prosecondphase},  \ref{prothirdphase} and \ref{proindi}
we then can calculate the probabilities corresponding to the five sets from Definition \ref{defpartition} and get the probabilities 
$(p_k, 1 \leq k \leq 5)$ defined in \eqref{defpi}, which sum to one.
Let us detail the calculations for the case $i)$: by applying \eqref{decomproba123}, Proposition \ref{prosecondphase}
and the Markov property,
the probability to see one of the two trajectories described in $i)$ is
\begin{multline}\label{mathcalPi1}
\mathcal{P}(i,1)=  \P^{(3)}(R2(i)^{(3,ga)})   \P^{(1)}([ NR(i)^{(1)}\sqcup R2(i)^{(1)}]  \cap NR(j)^{(1)})\\
+\P^{(3)}(NR(i)^{(3)})\P^{(1)}( NR(i)^{(1)})+O_K(\eps) ,
\end{multline}
where $i$ and $j$ are two distinct individuals (exchangeability).
But thanks to Proposition \ref{proindi} we know that the neutral genealogies of individuals $i$ and $j$ are nearly independent. Hence
adding Proposition \ref{prop:1phase_probg1} leads to
$$ \P^{(1)}([ NR(i)^{(1)}\sqcup R2(i)^{(1)}]  \cap NR(j)^{(1)})=(q_1q_2+q_1(1-q_2))q_1q_2+ O_K({\eps}).$$
Applying Propositions \ref{prop:1phase_probg1} and \ref{prothirdphase} in \eqref{mathcalPi1} yields
$$ \mathcal{P}(i,1)= (1-\bar{q}_2)q_1^2q_2+\bar{q}_2 q_1q_2+ O_K(\sqrt{\eps})=p_1+ O_K(\sqrt{\eps}),$$
where we recall the definition of $p_1$ in \eqref{defpi}.\\

Finally, we get the asymptotic independence of the neutral genealogies of the $d$ sampled individuals during the first and third phases
by applying the multinomial version of the de Finetti Representation Theorem (see \cite{diniz2014simple} Chapter 4 for a simple proof) to the result of Proposition
\ref{proindi}.
The asymptotic independence during the second phase follows from  Proposition \ref{prosecondphase} as, with high probability,
nothing happens.


\section{Number of births and deaths during the selective sweep}\label{technicalsection}

 In this section we derive some results on birth and death numbers of the population processes $\tilde{N}$ and $\tilde{\tilde{N}}$,
 needed in Sections \ref{firstphase} and \ref{secondphase}
 to prove Propositions \ref{prop:1phase_probg1}, \ref{prosecondphase} and \ref{prothirdphase}.

\subsection{Coupling with supercritical birth and death processes during the first phase}

 We are interested in the dynamics of the process $\tilde{N}_a$ during the first phase, that is, before the time
$\tilde{T}^K_\eps$. The idea is to couple this process with two supercritical
birth and death processes, and deduce its
dynamics  from well known results on birth and death processes.
Recall the definition of  the rescaled invasion fitness $s$
in \eqref{probafix},
and for $\eps < S_{aA}/( 2 {C_{a,A}C_{A,a}}/{C_{A,A}}+C_{a,a} )$ define the two approximations,
\begin{equation}\label{def_s_-s_+}
s-\frac{ 2 {C_{a,A}C_{A,a}}+C_{a,a}{C_{A,A}}}{f_a{C_{A,A}}}\eps=:s_-(\eps)\leq s \leq s_+(\eps):=s
+ 2 \frac{C_{a,A}C_{A,a}}{f_aC_{A,A}}\varepsilon.
\end{equation}
Then for $t < \tilde{T}_\eps^K \wedge S_\eps^K$ the death rate of $a$-individuals in the process $\tilde{N}$ equals that of the process $N$,
 defined in \eqref{def:totalbd} and satisfies
\begin{equation}\label{ineqtxmort}
1-s_+(\eps)\leq \frac{{d}_{a}(\tilde{N}(t))}{f_a\tilde{N}_a(t)}= 1-s+\frac{C_{a,A}}{f_aK}
(\tilde{N}_A(t)-\bar{n}_AK)+\frac{C_{a,a}}{f_aK}\tilde{N}_a(t)\leq
1-s_-(\eps).
\end{equation}
For $S_\eps^K\leq t <\tilde{T}_\eps^K$, according to the definition of $\tilde{N}$ in \eqref{deftildeN},
the death rate of $a$-individuals also satisfies
\begin{equation}\label{ineqtxmort}
1-s_+(\eps)\leq \frac{{d}_{a}^K(\tilde{N}_{A}{({S_\eps^K}^-)} e_{A\underline{\mathfrak{b}\mathfrak{c}}},\tilde{N}^{(a)}(t))}{f_a\tilde{N}_a(t)}\leq
1-s_-(\eps).
\end{equation}
Hence, following Theorem 2 in \cite{champagnat2006microscopic} we can construct the processes
$Z^-_\eps$, $(\tilde{N}_A,\tilde{N}_a)$ and $Z^+_\eps$ on the same probability space such that almost surely:
\begin{equation}\label{couplage1}
 Z^-_\eps(t) \leq \tilde{N}_a(t) \leq Z^+_\eps(t), \quad \text{for all } t <  \tilde{T}^K_\eps,
\end{equation}
where for $*\in \{-,+\}$, $Z^*_\eps$ is a birth and death process with initial state $1$, and individual birth  and death rates $f_a$ and
$f_a (1-s_*(\eps))$.\\

Let $\sigma_u^K$ denote the time of the first hitting of $\lfloor u \rfloor$ by the process $\tilde{N}_a$:
\begin{equation}\label{tpssigma}
 \sigma^K_u:=\inf \{ t\geq 0, \tilde{N}_a(t)= \lfloor u \rfloor\}, \quad u \in \R_+.
\end{equation}
If for $-1<s<1$, $\tilde{Z}^{(s)}$ is a random walk with jumps $\pm 1$ where up-jumps occur
with probability $1/(2-s)$ and down-jumps with probability $(1-s)/(2-s)$, we introduce 
\begin{equation} \label{defPronde} \mathcal{P}_i^{(s)}:= \mathcal{L} \left( \tilde{Z}^{(s)} |  \tilde{Z}^{(s)}(0)=i \right), \quad i \in \N.
\end{equation}
the law of $\tilde{Z}^{(s)}$ when the initial state is $i \in \N$
and for every $\rho \in \R_+$ the stopping time
\begin{equation} \label{deftauas}
 \tau_\rho:=\inf \ \{ n \in \Z_+, \tilde{Z}^{(s)}_n= \lfloor \rho \rfloor \}.
\end{equation}

\subsection{Number of jumps of $\tilde{N}_a$ during the first phase}


\subsubsection{Expectation of the number of upcrossings}

Let us recall Equation
\eqref{deftpssauts} and consider $k<\lfloor \eps K \rfloor $. Then
the number of upcrossings from $k$ to $k+1$ during the first phase is:
\begin{equation} \label{Umjk} U^K_{k}(1):=\# \{ m , \tau_m^K< \tilde{T}_\eps^K, (\tilde{N}_a(\tau_{m}^K),\tilde{N}_a(\tau_{m+1}^K))=(k,k+1)\}, \end{equation}
where $(1)$ stands for the first phase.
Recall \eqref{compact} and \eqref{def_s_-s_+}, and introduce a real number $\lambda_\eps$
\begin{equation}\label{deflambda}
 \lambda_\eps:=(1-s_-(\eps))^3(1-s_+(\eps))^{-2},
\end{equation}
which belongs to $(0,1)$ for $\eps$ small enough. We have the following result:

\begin{lem}\label{lemmaexpup}
There exist three positive finite constants $c$, $K_0$ and $\eps_0$ such that for $K\geq K_0$ and $\eps\leq \eps_0$:
\begin{enumerate}
 \item[] If $j\leq k<\lfloor \eps K \rfloor$ and $n_A \in I_\eps^K±1$,
\begin{equation} \label{expupa2}
\Big|\E^{(1)}_{(n_A,j)}[U^K_{k}{(1)}]-\frac{1-(1-s)^{\lfloor \eps K \rfloor -k}-(1-s)^{k+1}}{s} \Big|  \leq c \eps . \end{equation}
\item[] If $k<j<\lfloor \eps K \rfloor$ and $n_A \in I_\eps^K±1$,
\begin{equation}  \label{expupa3}
 \E^{(1)}_{(n_A,j)}[U^K_{k}(1)]\leq \frac{(1-s_-(\eps))^{j-k}}{s_+(\eps)s_-^2(\eps)}.
\end{equation}
\item[] If $k'\leq k<\lfloor \eps K \rfloor$ and $n_A \in I_\eps^K±1$,
\begin{equation}\label{covUmlk}  \Big|\cov^{(1)}_{(n_A,j)}(U_{k}^{K}(1),U_{k'}^{K}(1))\Big|
 \leq  c \Big( \lambda_\eps^{(k-k')/2} +\eps\Big).
\end{equation}
\end{enumerate}
\end{lem}

 \begin{proof}
The idea, {which comes} from \cite{schweinsberg2005random} and
will be used several times throughout Section \ref{technicalsection}, is to compare the number of upcrossings with
geometric random variables.
Suppose first that $j\leq k$. Then on the event $\{\tilde{T}_{\eps}^K<\infty\}$ the process $\tilde{N}_a$ necessarily jumps from
$k$ to $k+1$. Being in $k+1$, it either reaches $\lfloor \eps K \rfloor$ before $k$, or it goes back and then again from
$k$ to $k+1$ and so on. We first approximate the probability that there is only one jump from $k$ to $k+1$.  As we do not know the
value of $\tilde{N}_A$ when $\tilde{N}_a$ hits $k$ for the first time, we bound the probability
using the extreme values it can take. Recall Definitions
\eqref{tpssigma} and \eqref{deftauas}. The upper bound is derived as follows:
\begin{eqnarray} \label{onlyonejump}
 \P^{(1)}_{(n_A,j)}(U^K_{k}(1)=1)&\leq & \underset{n_A\in I_\eps^K±1}{\sup}\P^{(1)}_{(n_A,k+1)}(\tilde{T}_{\eps}^K<\sigma_{k}^K)\\
&=&\underset{n_A\in I_\eps^K±1}{\sup} \frac{\P_{(n_A,k+1)}(\tilde{T}_{\eps}^K<\sigma_{k}^K)}
{ \P_{(n_A,k+1)}(\tilde{T}_{\eps}^K<\infty)}
\leq
 q^{(s_+(\eps),s_-(\eps))}_{k},\nonumber
\end{eqnarray}
where we use \eqref{defP1} and for $(s_1,s_2) \in (0,1)^2$
\begin{equation}\label{defqks1s2} q^{(s_1,s_2)}_{k}:=\frac{\mathcal{P}_{k+1}^{(s_1)}(\tau_{  \eps K }<\tau_k)}{\mathcal{P}_{k+1}^{(s_2)}(\tau_{  \eps K }<\tau_0)} .\end{equation}
Similarly, we show that $ \P^{(1)}_{(n_A,j)}(U^K_{k}(1)=1)\geq  q^{(s_-(\eps),s_+(\eps))}_{k}$.
In the same way, we can approximate the probability that there are least three jumps from $k$ to $k+1$ knowing that there
are at least two jumps, and so on. We deduce that we can construct
two geometric random variables $G_1$ and $G_2$, possibly on an enlarged space,
with respective parameters $q^{(s_+(\eps),s_-(\eps))}_{k}\wedge 1$ and $q^{(s_-(\eps),s_+(\eps))}_{k}$ such that
\begin{equation}\label{compaU'|Uj} G_1\leq U_{k}^{K}(1) \leq G_2, \quad \text{a.s.}  \end{equation}
In particular, taking the expectation we get from \eqref{hitting_times}
\begin{multline} \label{encadexpUmjk}
  \frac{(1-(1-s_+(\eps))^{\lfloor \eps K \rfloor -k})(1-(1-s_-(\eps))^{k+1})}{s_+(\eps)(1-(1-s_-(\eps))^{\lfloor \eps K \rfloor})}
\leq \E^{(1)}_{(n_A,j)}[U^K_{k}(1)]
\\ \leq 
\frac{(1-(1-s_-(\eps))^{\lfloor \eps K \rfloor -k})(1-(1-s_+(\eps))^{k+1})}{s_-(\eps)(1-(1-s_+(\eps))^{\lfloor \eps K \rfloor})}.
\end{multline}
 According to \eqref{probafix} and \eqref{def_s_-s_+}, $0<s<1$ and
$|s_+(\eps)-s_-(\eps)|\leq (4C_{a,A}C_{A,a}+C_{a,a}C_{A,A})\eps/(f_aC_{A,A})$. Hence the last inequality and straightforward calculations
lead to
\eqref{expupa2}.\\

Let us now assume that $k<j$. Then  we have
\begin{multline*}
 \P^{(1)}_{(n_A,j)}(U^K_{k}{(1)}\geq 1)\leq \underset{n_A\in I_\eps^K±1}{\sup}\P_{(n_A,j)}^{(1)}(\sigma_k^K<\tilde{T}_\eps^K)
 =  \underset{n_A\in I_\eps^K±1}{\sup}\frac{\P_{(n_A,j)}(\tilde{T}_\eps^K<\infty|\sigma_k^K<\tilde{T}_\eps^K)
\P_{(n_A,j)}(\sigma_k^K<\tilde{T}_\eps^K)}
{\P_{(n_A,j)}(\tilde{T}_\eps^K<\infty)}\\
 \leq  \frac{\mathcal{P}_{k}^{(s_+(\eps))}(\tau_{\eps K}<\tau_0)\mathcal{P}_{j}^{(s_-(\eps))}(\tau_k<\tau_{\eps K})}
{\mathcal{P}_{j}^{(s_-(\eps))}(\tau_{\eps K}<\tau_0)} \leq \frac{(1-s_-(\eps))^{j-k}}{s_+(\eps)s_-(\eps)},
\end{multline*}
where we again used \eqref{defP1} and \eqref{hitting_times}. Moreover, the same proof as for \eqref{compaU'|Uj} leads to:
\begin{equation*}
 \E^{(1)}_{(n_A,j)}[U^K_{k}{(1)}|U^K_{k}{(1)}\geq 1]\leq   \Big(q^{(s_-(\eps),s_+(\eps))}_{k}\Big)^{-1} \leq  s_-^{-1}(\eps) ,
 \end{equation*}
where we used Equation \eqref{minqk}. This ends the proof of \eqref{expupa3}.
The last inequality, \eqref{covUmlk}, has been stated in \cite{smadi2014eco} (Equation $(7.26)$).
\end{proof}

%

\subsubsection{Expectation of hitting numbers}

Let us recall \eqref{Umjk} and introduce for $0< j\leq k<\lfloor  \varepsilon K \rfloor$
the total number of downcrossings from $k$ to $k-1$,
\begin{equation} \label{Dk} D_{k}^K(1):=\# \{m, \tau_m^K\leq \tilde{T}_\eps^K, (\tilde{N}_a(\tau_m^K),\tilde{N}_a(\tau_{m+1}^K))=(k,k-1) \},\end{equation}
and the number of hittings
of the state $k$ by the process $\tilde{N}_a$ before the time $\tilde{T}_\eps^K$:
\begin{equation} \label{Vk} {V}_{k}^K (1):=U_{k-1}^K(1)+D_{k+1}^K(1)=\# \{m ,\tau_m^K\leq \tilde{T}_\eps^K,
\tilde{N}_a(\tau_{m-1}^K)\neq k , \tilde{N}_a(\tau_m^K)=k\}.\end{equation}


Recall the definition of $\lambda_\eps \in (0,1)$ in \eqref{deflambda}. We can state the following Lemma, which will be useful to get
 bounds on the number of upcrossings of the $A$-population during the first phase (see Lemma \ref{lemespvarmathcalU}):

\begin{lem}\label{lemV}
There exist three finite constants $c$, $K_0$ and $\eps_0$ such that for $K\geq K_0$, $\eps\leq \eps_0$ and $k'< k <\lfloor \eps K \rfloor $:
 \begin{equation*}
 \Big| \E^{(1)}[{V}_{k}^K (1)] -\frac{(2-s)(1-(1-s)^{\lfloor \eps K\rfloor -k}-(1-s)^k)}{s}\Big|\leq c\eps,
 \ \ \text{and} \ \
  |\cov^{(1)}({V}_{k'}^K (1),{V}_{k}^K (1))|\leq c(\eps + \lambda_\eps^{(k-k')/2}).
 \end{equation*}
\end{lem}

\begin{proof}

Under $\P^{(1)}$ the $a$-population size goes from $1$ to $\lfloor \eps K \rfloor$, thus
the number of downcrossings from $k+1$ to $k$ is equal to the number of upcrossings from $k$ to $k+1$ minus $1$.
Adding \eqref{Vk} yields
\begin{equation*}  {V}_{k}^{K} (1)={U}_{k-1}^{K} (1)+{U}_{k}^{K} (1)-1, \quad \P^{(1)}-a.s. \end{equation*}
We get the first part of the Lemma by taking the expectation and applying \eqref{expupa2}.
The proof of the second part follows that of \eqref{covUmlk}, and once again we can find the details in the proof of
 Equation (7.26) in \cite{smadi2014eco}.
\end{proof}

\subsubsection{Number of upcrossings during an excursion above or below a given level} \label{sectionexcu}
We now focus on the number of upcrossings from $k$ to $k+1$ during an excursion above or below $l$.
Let us denote by $\sigma_l^K(1)$ the jump number of the first hitting of $l$ before the end of the first phase: for $l<\lfloor\eps K \rfloor$,
\begin{equation}\label{defsigma}
 \sigma_l^K(1):=\inf \{ m, \tau_m^K \leq \tilde{T}_\eps^K, \tilde{N}_a(\tau_m^K)=l \},
\end{equation}
and for $1\leq k,l<\lfloor \eps K \rfloor$ and $n_A \in I_\eps^K±1$,
\begin{equation}\label{defjumpexc}
 U^K_{n_A,l,k}(1):= \# \Big\{ m<\sigma_l^K(1), (\tilde{N}_a(\tau_{m}^K),\tilde{N}_a(\tau_{m+1}^K))=(k,k+1)
\Big\}. \end{equation}
Then, if we denote by $\mu_\eps$ the real number
\begin{equation}\label{defmu}
\mu_\eps:= (1-s_-(\eps))^2(1-s_+(\eps))^{-1},
\end{equation}
which belongs to $(0,1)$ for $\eps$ small enough, we can derive the following bounds:

\begin{lem}\label{excunder}
There exist three positive, finite constants $c$, $K_0$ and $\eps_0$
such that for $K\geq K_0$, $\eps\leq \eps_0$, $1\leq k<l<\lfloor \eps K \rfloor$  and  $n_A \in I_\eps^K±1$,
\begin{equation*}
  \E^{(1)}_{(n_A,k+1)}[U^K_{n_A,k,l}(1)|\sigma_k^K(1)<\infty] \vee \E^{(1)}_{(n_A,l-1)}[U^K_{n_A,l,k}(1)]\leq  c\mu_\eps^{l-k}.
 \end{equation*}
 \end{lem}

\begin{proof}
Equations (B.5) and (B.6) in \cite{smadi2014eco} state that for $k<l<\lfloor \eps K \rfloor$  and  $n_A \in I_\eps^K±1$,
$$ \P^{(1)}_{(n_A,k+1)}(U^K_{n_A,k,l}(1)\geq 1|\sigma_k^K(1)<\infty)\leq c(1-s_-(\eps))^{l-k},$$
and
$$ \P^{(1)}_{(n_A,k+1)}(U^K_{n_A,k,l}(1)= 1|U^K_{n_A,k,l}(1)\geq 1,\sigma_k^K(1)<\infty)
\geq c\Big(\frac{1-s_+(\eps)}{1-s_-(\eps)}\Big)^{l-k}
$$
for a finite $c$. By comparing $U^K_{n_A,k,l}(1)$ with a geometric random variable we get the first inequality.
To bound the expectation of upcrossings from $k$ to $k+1$ during an excursion below $l$ we first bound the probability to have at least one jump from $k$ to $k+1$ during
 such an excursion. By definition, $\tilde{N}_a$ necessarily hits $l-1$ during the excursion below $l$. Recall Definitions  \eqref{defP1}, \eqref{tpssigma}
and \eqref{deftauas}. Then for every $n_A$
in $I_\eps^K±1$,
\begin{eqnarray*}
 \P^{(1)}_{(n_A,l-1)}(\sigma^K_k<\sigma^K_l|\sigma^K_l<\infty)&=& \frac{{\P}_{(n_A,l-1)}(\tilde{T}_\eps^K<\infty |\sigma^K_k<\sigma^K_l)
{\P}_{(n_A,l-1)}(\sigma^K_k<\sigma^K_l)}
{{\P}_{(n_A,l-1)}(\tilde{T}_\eps^K<\infty)}\\
& \leq & \frac{\mathcal{P}^{(s_+(\eps))}_k(\tau_{\eps K}<\tau_0)\mathcal{P}^{(s_-(\eps))}_{l-1}(\tau_{k}<\tau_l)}{\mathcal{P}^{(s_-(\eps))}_{l-1}(\tau_{\eps K}<\tau_0)}
\leq  \frac{(1-s_-(\eps))^{l-k-1}}{s_-(\eps)},  \qquad \qquad \qquad
\end{eqnarray*}
where we used \eqref{hitting_times}. The next step consists in bounding the number of upcrossings from $k$ to $k+1$ during the excursion
knowing that this number is larger than one: for $n_A
\in I_\eps^K±1$,
\begin{eqnarray*}
 \P^{(1)}_{(n_A,k+1)}(\sigma_l^K<\sigma_k^K)
&=& \frac{{\P}_{(n_A,k+1)}(\tilde{T}_\eps^K<\infty |\sigma_l^K<\sigma_k^K)
{\P}_{(n_A,k+1)}(\sigma_l^K<\sigma_k^K)}
{{\P}_{(n_A,k+1)}(\tilde{T}_\eps^K<\infty)}\\
& \geq & \frac{\mathcal{P}^{(s_-(\eps))}_l(\tau_{\eps K}<\tau_0)\mathcal{P}^{(s_-(\eps))}_{k+1}(\tau_{l}<\tau_k)}{\mathcal{P}^{(s_+(\eps))}_{k+1}(\tau_{\eps K}<\tau_0)}
\geq  s_-^2(\eps),
\end{eqnarray*}
where we again used \eqref{hitting_times}. Hence on the event $\{U^K_{n_A,l,k}(1)\geq 1\}$, $U^K_{n_A,l,k}(1)$
is smaller than a geometric random variable with parameter $s_-^2(\eps)$ and we get:
\begin{eqnarray*} \E^{(1)}_{(n_A,l-1)}[U^K_{n_A,l,k}(1)]&\leq & s_-^{-2}(\eps)\P^{(1)}_{(n_A,l-1)}(U^K_{n_A,l,k}(1)\geq 1)\ \leq \
\frac{(1-s_-(\eps))^{l-k-1}}{s_-^3(\eps)},\end{eqnarray*}
which ends the proof of Lemma \ref{excunder}.
\end{proof}

\subsection{Number of jumps $\tilde{N}_A$ during the first phase}\label{subsec:bdA1p}

We introduce for $k<\lfloor \eps K \rfloor$ the number of upcrossings of the $A$-population
when the $a$-population is of size $k$:
\begin{equation} \label{mathcalUk} \mathcal{U}_{k}^K {(1)}:=\# \{m, \tau_m^K\leq \tilde{T}_\eps^K,
\tilde{N}_A(\tau_{m+1}^K)-\tilde{N}_A(\tau_m^K)=1 , \tilde{N}_a(\tau_m^K)=k \}.\end{equation}
We are now able to get bounds for the expectations and covariances of these quantities:

\begin{lem}\label{lemespvarmathcalU}
 There exist three finite constants $c$, $K_0$ and $\eps_0$ such that for $K \geq K_0$, $\eps\leq \eps_0$ and $k<\lfloor \eps K \rfloor$,
$$\label{approxsummathcalUk}
\Big| \E^{(1)}\Big[ \sum_{i=1}^k \mathcal{U}_i^K(1) \Big]- \frac{f_A\bar{n}_AK\log k}{sf_a} \Big|\leq  cK(1+ {\eps}\log k)
\quad \text{and} \quad \var^{(1)}\Big( \sum_{i=1}^k \mathcal{U}_i^K(1) \Big)\leq cK^2(1+{\eps} \log^2 k).
$$
\end{lem}

\begin{proof}
 The proof is based on the comparison of the $A$- and $a$-population jump rates.
Let us first focus on the $a$-population. For $k\leq \lfloor \eps K \rfloor$ and $n_A \in I_\eps^K±1$,
\begin{align}\label{majkk+1}
   \P^{(1)}_{(n_A,k)}(\tilde{N}_a(\delta t)\neq k) &= 
\P_{(n_A,k)}(\tilde{T}_\eps^K<\infty |\tilde{N}_a(\delta t)=k\pm 1) 
\P_{(n_A,k)}(\tilde{N}_a(\delta t)=k\pm 1)/\P_{(n_A,k)}(\tilde{T}_\eps^K<\infty) \nonumber \\
&\leq 
\mathcal{P}_{k\pm 1}^{(s_+(\eps))}(\tilde{T}_\eps^K<\infty)
 \P_{(n_A,k)}(\tilde{N}_a(\delta t)=k\pm 1)/ \mathcal{P}^{(s_{-}(\eps))}_k(\tilde{T}_\eps^K<\infty) \nonumber \\
 & \leq  \frac{(1+c\eps)}{1-(1-s)^k}\Big((1-(1-s)^{k+1})f_ak+(1-s)(1-(1-s)^{k-1})(D_a+C_{aA}\bar{n}_A)k\Big) \delta t\nonumber\\
&= (1+c\eps)f_a(2-s)k \delta t ,
\end{align}
for a finite constant $c$ and $\eps$ small enough, where
 $\delta t$ is a small time step and by abuse of notation we did not indicate the $o(\delta t)$'s.
 We used the definition of $\P^{(1)}$ in \eqref{defP1} for the equality, Coupling \eqref{couplage1} for the first inequality,
\eqref{hitting_times} for the second one, and the equality $ S_{aA}=f_a-D_a - C_{a,A}\bar{n}_A $ for the last one.
Reasoning similarly we get:
\begin{equation}\label{encaderatea}
 (1-c\eps)f_a(2-s)k \delta t\leq \P^{(1)}_{(n_A,k)}(\tilde{N}_a(\delta t)\neq k).
\end{equation}
Let us now focus on the number of upcrossings of the $A$-population. The definition of $\tilde{N}$ in \eqref{deftildeN} and Bayes' Theorem yield
\begin{equation}\label{majstsA} (1-c\eps)f_A\bar{n}_AK \delta t \leq \P^{(1)}_{(n_A,k)}(\tilde{N}_A( \delta t )= n_A+1)
\leq (1+c\eps)f_A\bar{n}_AK \delta t ,
\end{equation}
for a finite $c $ and $\eps$ small enough. Indeed, from Coupling \eqref{couplage1} and Equation \eqref{hitting_times} we get the following bound, independent of $n_A$ in $I_\eps^K±1$:
$$ \frac{1-(1-s_-(\eps))^k}{1-(1-s_-(\eps))^{\lfloor \eps K \rfloor}}\leq \P_{(n_A,k)}(\tilde{T}_\eps^K <\infty )\leq
\frac{1-(1-s_+(\eps))^k}{1-(1-s_+(\eps))^{\lfloor \eps K \rfloor}}. $$

 Hence there exist two finite constants $c$ and $\eps_0$ such that for every $\eps\leq \eps_0$,
 if we introduce the parameters
\begin{equation} \label{defqkeps} \frac{1}{q^{(1)}_k(\eps)}:=1+(1-c\eps)\frac{f_A\bar{n}_AK}{(2-s)f_ak}
  \leq 1+(1+c\eps)\frac{f_A\bar{n}_AK}{(2-s)f_ak}=:\frac{1}{q^{(2)}_k(\eps)} ,\end{equation}
 we can deduce from \eqref{majkk+1} to \eqref{majstsA} that
 for $k <\lfloor \eps K \rfloor$
\begin{equation}\label{tildemajV2} \sum_{V_k^K{(1)}}\Big({G}^i_{q^{(1)}_k(\eps)}-1 \Big)\leq
{\mathcal{U}}_k^K{(1)}
\leq \sum_{V_k^K{(1)}}\Big({G}^i_{q^{(2)}_k(\eps)}-1 \Big) ,\end{equation}
where for $j \in \{1,2\}$, $({G}^i_{q^{(j)}_k(\eps)}, i \in \N)$ is a sequence of geometric random variables
with parameter
$q^{(j)}_k(\eps)$
 independent of $V_l^K{(1)}$ (defined in \eqref{Vk}) for all $l< \lfloor \eps K \rfloor$.
Hence a direct application of Lemmas \ref{lemV} and \ref{lemgeom} leads to
\begin{equation}\label{majexpmathcalUDtilde}
\Big|  \E^{(1)}\Big[  \mathcal{U}_k^K(1)\Big]-
\frac{f_A\bar{n}_AK}{sf_ak}(1-(1-s)^k-(1-s)^{\lfloor \eps K \rfloor -k})\Big| \leq c\eps \frac{K}{k}, \end{equation}
for a finite $c$ and $\eps$ small enough. This implies the first inequality of Lemma \ref{lemespvarmathcalU}.\\

Let us now bound the second moment of $ \mathcal{U}_k^K(1)$ and the expectation of
${\mathcal{U}}_k^K(1) {\mathcal{U}}_l^K(1)$ for $k\neq l$. The first upper bound follows again
from a direct application of Lemmas \ref{lemV} and \ref{lemgeom}. We get
\begin{equation}\label{majcarremathcalU} \E^{(1)}\Big[  (\mathcal{U}_k^K(1))^2\Big]\leq \E^{(1)}\Big[ \Big(\sum_{V_k^K{(1)}}{G}^i_{q^{(2)}_k(\eps)} \Big)^2\Big]
\leq \frac{2(\E^{(1)}[V_k^K{(1)}])^2}{(q^{(2)}_k(\eps))^2}\leq 2(1+c\eps)\Big(\frac{f_A\bar{n}_AK}{sf_ak} \Big)^2 ,\end{equation}
for a finite $c$ and $\eps$ small enough.
A new application of the same Lemmas yields, for $k < l <\lfloor \eps K \rfloor $
\begin{equation} \label{majprod}
  \E^{(1)}\Big[ {\mathcal{U}}_k^K(1) {\mathcal{U}}_l^K(1)\Big]
 \leq  \frac{ \E^{(1)}[ V_k^K{(1)}V_l^K{(1)} ]}{q^{(2)}_k(\eps)q^{(2)}_l(\eps)}
 \leq   c(1+\eps+  \lambda_\eps^{(l-k)/2}) \frac{(f_A\bar{n}_AK)^2}{(f_as)^2 kl} ,
\end{equation}
where we used that $\E^{(1)}[XY]=\E^{(1)}[X]\E^{(1)}[Y]+\cov^{(1)}(X,Y)$ for any real random variables $(X,Y)$.
From \eqref{tildemajV2} to \eqref{majprod} and \eqref{lemma3.5} we deduce that there exists a finite $c$ such that for $\eps$ small enough and $k <\lfloor \eps K \rfloor $,
\begin{equation} \label{majprod2}  \E^{(1)}\Big[ \Big( \sum_{i=1}^{k} \mathcal{U}_i^K(1)\Big)^2\Big]
 \leq
   (1+c\eps) \Big(\frac{f_A\bar{n}_AK\log k}{f_as }\Big)^2+cK^2 . \end{equation}
Reasoning similarly to get the lower bound, we obtain
\begin{equation} \label{majprod3} \Big| \E^{(1)}\Big[ \Big( \sum_{i=1}^{k} \mathcal{U}_i^K(1)\Big)^2\Big]
- \Big(\frac{f_A\bar{n}_AK\log k}{f_as }\Big)^2\Big| \leq  cK^2(1+\eps \log^2 k) . \end{equation}
Adding the first inequality of Lemma \ref{lemespvarmathcalU} we conclude the proof.
\end{proof}

\subsection{Coupling with subcritical birth and death processes during the third phase}\label{subsec:newprob3phase}

We couple the process $\tilde{\tilde{N}}_a$ with two subcritical birth and death processes to control
its dynamics.
We recall the definition of ${\mathcal{N}}_\eps^K$ in \eqref{defNepsK} and introduce
\begin{equation}\label{defsbar} \bar{s}:={|S_{Aa}|}/{f_A}. \end{equation}
Let us define for $\eps$ small enough,
\begin{equation}\label{def_bars_-s_+}
\bar{s}-\frac{ M''C_{A,a}}{f_A}\eps=:\bar{s}_-(\eps)<\bar{s}<
\bar{s}_+(\eps):=\bar{s}
+ \frac{C_{A,A}+M''C_{A,a}}{f_A}\varepsilon,
\end{equation}
where $M''$ has been defined just before Definition \eqref{compact2}.
Then, according to the definition of $\tilde{\tilde{N}}$ in \eqref{deftildetildeN},
 we can follow Theorem 2 in \cite{champagnat2006microscopic} and construct the processes
$Y^+_\eps$, $\tilde{\tilde{N}}$ and $Y^-_\eps$ on the same probability space such that on the event $\mathcal{N}_\eps^K$
\begin{equation}\label{couplage3}
 Y^+_\eps(t) \leq \tilde{\tilde{N}}_A(t) \leq Y^-_\eps(t), \quad \text{for all }
T_\eps^K+t_\eps\leq t< T_\eps^K+t_\eps+\tilde{\tilde{T}}_0^{(K,A)},\quad \text{a.s.},
\end{equation}
where for $*\in \{-,+\}$, $Y^*_\eps$ is a birth and death process with initial state $N_A(T_\eps^K+t_\eps)$ and
individual birth and death rates $f_A$ and  $f_A (1+\bar{s}_*(\eps))$,
and we recall that $\tilde{\tilde{T}}_0^{(K,A)}$ is the analog of ${{T}}_0^{(K,A)}$ (defined in \eqref{T0K})
for the process $\tilde{\tilde{N}}$.\\

Recall Definition \eqref{defPronde}, and let us introduce for $i \in \N$ $\mathcal{Q}_i^{(s)}= \mathcal{P}_i^{(-s)}$ and
 for $\rho \in \R_+$ the stopping time
\begin{equation} \label{defnuuas}
 \nu_\rho:=\inf \{ n \in \Z_+, \tilde{Z}^{(-s)}_n= \lfloor \rho \rfloor \}.
\end{equation}

\subsection{Number of jumps of $\tilde{\tilde{N}}_A$ during the third phase}
Similarly as in \eqref{Vk} we introduce for $1 \leq k <\lfloor \eps K \rfloor$ the random variable
$\mathcal{V}_k^K(3)$ which corresponds to the number of hittings of state $k$ by the process $\tilde{\tilde{N}}_A$ during the third phase.
Recall Definitions \eqref{defomega12}, \eqref{compact2} and \eqref{def_bars_-s_+}. We have the following approximations:

\begin{lem}\label{jumpA3}
Let $u$ be in $[\omega_1,\omega_2]$.
 There exist three finite constants $c$, $K_0$ and $\eps_0$ such that for $K \geq K_0$, $\eps\leq \eps_0$
 and $n_a$ in $J_\eps^K±1$,
if $\lfloor u K \rfloor< k<\lfloor \eps K \rfloor,$
 \begin{equation*}
\E^{(3)}_{(\lfloor u K \rfloor,n_a)}[\mathcal{V}_{k}^K(3)]
\leq
(1+c\eps)
\frac{2+\bar{s}}{\bar{s}}(1+\bar{s}_-(\eps))^{\lfloor u K \rfloor-k} ,
\end{equation*}
and if $k\leq \lfloor u K \rfloor$,
\begin{eqnarray*}
\Big|\E^{(3)}_{(\lfloor u K \rfloor,n_a)}[\mathcal{V}_{k}^K(3)]-\frac{2+\bar{s}}{\bar{s}}(1-(1+\bar{s})^{-k}
-(1+\bar{s})^{k-\lfloor \eps K \rfloor})\Big|
\leq
c\eps.
\end{eqnarray*}
\end{lem}

\begin{proof}
The proof is very similar to that of \eqref{expupa2}, hence we do not detail all the calculations and refer to the proof
of Lemma \ref{lemmaexpup}. First we consider
$\lfloor u K \rfloor< k<\lfloor \eps K \rfloor$ and  approximate under $\P^{(3)}$ the probability for $\tilde{\tilde{N}}_A$ to hit
$k$ before the extinction of the $A$-population. Indeed, if $k\leq \lfloor u K \rfloor$, we know that
$\tilde{\tilde{N}}_A$ hits $k \ \P^{(3)}$-a.s.
Let $\lfloor u K \rfloor<k<\lfloor \eps K \rfloor$. Then for every $n_a \in J_\eps^K±1$,
Equation \eqref{hitting_times} implies
\begin{equation} \label{probahitk}
 \P^{(3)}_{(\lfloor u K \rfloor,n_a)}(\tilde{\tilde{N}}_A \text{ hits }k)
\leq  \frac{\mathcal{Q}^{(\bar{s}_+(\eps))}_{k}(\nu_0<\nu_{\eps K})
\mathcal{Q}^{(\bar{s}_-(\eps))}_{\lfloor u K \rfloor}(\nu_{k}<\nu_0)}
{\mathcal{Q}^{(\bar{s}_-(\eps))}_{\lfloor u K \rfloor}(\nu_0<\nu_{\eps K})}\leq \frac{1+c\eps}
{(1+\bar{s}_-(\eps))^{k-\lfloor u K \rfloor}},
\end{equation}
for a finite $c$, $\eps$ small enough and $K$ large enough.
The second step consists in counting how many times the process $\tilde{\tilde{N}}_A$ hits $k$ during the third phase knowing that it happens
at least once. Once again we will compare
this number with geometric random variables, by approximating the probability to have only one jump.
The following inequality follows the spirit of \eqref{onlyonejump}. The only difference is that in the third phase $\tilde{\tilde{N}}_A$ is coupled with
subcritical birth
and death processes, whereas in the first phase $\tilde{N}_a$ was coupled with supercritical birth and death processes.
For every
$n_a\in J_\eps^K±1$ and $k<\lfloor \eps K \rfloor$,
\begin{multline*}
  \P^{(3)}_{(k,n_a)}(\tilde{\tilde{N}}_A(t)\leq k,
 \forall  t \geq 0 )
 \geq  \frac{\mathcal{Q}^{(\bar{s}_-(\eps))}_{k-1}(\nu_0<\nu_{k})\mathcal{Q}^{(\bar{s}_-(\eps))}_{k}(\nu_{k-1}<\nu_{k+1})}
{\mathcal{Q}^{(\bar{s}_+(\eps))}_{k}(\nu_0<\nu_{\eps K})}
\geq  \frac{(1-c\eps)\bar{s}}{(2+\bar{s})(1-(1+\bar{s})^{-k}-(1+\bar{s})^{k-\lfloor \eps K \rfloor})}.
\end{multline*}
We derive the upper bound similarly and end the proof by comparing the hitting numbers with geometric random variables.
For $\lfloor u K \rfloor <k<\lfloor \eps K \rfloor$ we have to multiply the expectation of the geometric random variables
by the probability to hit $k$ at least once, approximated in \eqref{probahitk}.
\end{proof}

\subsection{Number of births of $a$-individuals during the third phase}

Recall \eqref{defn(alpha)} and
let ${U}_{k}^K {(3)}$ be the number of births in the $a$-population during the third phase
when $\tilde{\tilde{N}}_A$ equals $k\leq \lfloor \eps K \rfloor$
\begin{multline} \label{UDk3phase} U_k^K {(3)}:=\# \{m, T_\eps^K+t_\eps<\tau_m^K\leq T_{\text{ext}}^K,
 \tilde{\tilde{N}}_A(\tau_m^K)=k , \text{ and }
\{\{\tilde{\tilde{N}}_a(\tau_{m+1}^K)-\tilde{\tilde{N}}_a(\tau_m^K)=1\}\\
\text{ or } \{\tilde{\tilde{N}}_a(\tau_{m+1}^K)=\tilde{\tilde{N}}_a(\tau_m^K),\tilde{\tilde{N}}^{(a)}(\tau_{m+1}^K)\neq \tilde{\tilde{N}}^{(a)}(\tau_m^K) \} \}.\end{multline}
We now state an approximation for the expectation of ${U}_{k}^K(3)$. We do not prove this result as it
is obtained in the same way as Lemma \ref{lemespvarmathcalU}: the birth rate of the $a$-population is close to $f_a\bar{n}_aK$,
the jump rate of the $A$-population is of order $(2+\bar{s})f_Ak$ when $\tilde{\tilde{N}}_A=k$ and the expectations of the hitting numbers
for the $A$-population
are given in Lemma \ref{jumpA3}. The only difference is that
the $A$-population size can hit values bigger than the initial value of the third phase,
$\tilde{\tilde{N}}_A(T_\eps^K+t_\eps)$. However the probabilities to hit such values
decrease geometrically (see Lemma \ref{jumpA3}) and they have a negligible influence on the final result. Thus we get

\begin{lem}\label{lemespvarU}
 There exist three finite constants $c$, $\eps_0$ and $K_0$ such that for $\eps\leq \eps_0$, $K \geq K_0$ and $k \leq \lfloor \eps K \rfloor$
$$
\Big| \E^{(3)}\Big[ \sum_{i=1}^k {U}_i^K(3) \Big]- \frac{f_a\bar{n}_aK\log k}{\bar{s}f_A} \Big|\leq  cK(1+\eps\log k)
\quad \text{and} \quad \var^{(3)}\Big( \sum_{i=1}^k {U}_i^K(3) \Big)\leq cK^2(1+\eps \log^2 K).
$$
\end{lem}

\section{First phase}\label{firstphase}

This section is dedicated to the proof of Proposition \ref{prop:1phase_probg1}.
We prove that there are only four different possible ancestral relationships of
the two neutral loci and calculate the probabilities for the non-negligible possibilities.

\subsection{Coalescence and recombination probabilities, negligible events}\label{subsec:RecombinationProbNegligible}
Recall Definition \ref{defcoalreco} and define, for $j \in \{1,2\}$
$$ r^*_{j}:= r_1+\mathbf{1}_{\{j=2\}}(r_2-2r_1r_2),  \quad \text{and} \quad r^*_{(1,2)}:=r_1r_2,$$
which denote the probability to have one (resp. two) recombination(s) somewhere before the locus $Nj$ (resp. before the
locus $N2$) at a birth event.

\begin{defi}\label{defpalphaalpha}
For $(\alpha,\alpha') \in \mathcal{A}^2$, $j \in \{1,2\}$ and $n=(n_A,n_a)\in \N^\mathcal{A}$ we define:
\begin{enumerate}
  \item[$p_{\alpha\alpha'}^{(c,j)}(n)$]:=
probability that two randomly chosen neutral alleles, located at locus $Nj$ and associated respectively with alleles $\alpha$ and
$\alpha'$
at time $\tau_m^K$, coalesce at this time conditionally on $(N_A,N_a)(\tau_{m-1}^K)=n$ and on the birth of an individual
carrying allele $\alpha$ at time $\tau_{m}^K$.
\item[$p_{\alpha \alpha'}^{(j)}(n)$]:= probability to have one
(and only one) recombination from the $\alpha$- into the $\alpha'$-population before locus $Nj$ conditionally on
$(N_A,N_a)(\tau_{m-1}^K)=n$ and on the birth of an individual
carrying allele $\alpha$ at time $\tau_{m}^K$.
\item[$p_{\alpha \alpha'}^{(1,2)}(n)$]:= probability
to have a double recombination under the same conditions
 \end{enumerate}
\end{defi}

Then we have the following result:

\begin{lem}\label{lempcoal}
Let $\alpha \in \mathcal{A}$, $n=(n_A,n_a)\in \N^\mathcal{A}$ such that $n_a \leq \lfloor \eps K \rfloor$, $n_A \in I_\eps^K ±1$ and $j \in \{1,2\}$. Then
there exists a finite $c$ such that,
\begin{equation*}
p_{aa}^{(c,j)}(n)=\frac{2}{n_a(n_a+1)}\Big( 1-\frac{r_j^* f_A n_A}
{f_An_A+f_an_a} \Big),\quad
p_{a A}^{(c,j)}(n) = \frac{r_j^*f_A}{(n_a+1)(f_An_A+f_an_a)} \quad \text{and} \quad   p_{A\alpha}^{(c,j)}(n)\leq \frac{c}{K^2}.
 \end{equation*}
 \end{lem}

\begin{proof}
The proof of the two equalities can be found in \cite{smadi2014eco} (Lemma 7.1) as the expression is the same for $n_A \in I_\eps^K$ or
$ dist (n_A,I_\eps^K)=1$ (where $dist$ is the canonical distance on $\R$).
The only difference is that we consider two neutral loci and have to exclude the double recombination case.
Indeed, if there are simultaneous recombinations the alleles located at SL and N2 in the newborn originate from the same parent.
The expressions of $p_{A\alpha}^{(c,j)}(n)$ in the case where $n_A \in I_\eps^K$ are also stated in \cite{smadi2014eco} (Lemma 7.1),
and from the definition of $\tilde{N}$ in \eqref{deftildeN} we get that when $ dist (n_A,I_\eps^K)=1$,
$p_{AA}^{(c,j)}(n)={2}/{n_A^2}$ and $p_{Aa}^{(c,j)}(n)=0$. This ends the proof.
\end{proof}

Next we focus on the recombination probabilities:

\begin{lem}\label{lemma:recprob}
Let $\alpha \in \mathcal{A}$, $n=(n_A,n_a) \in \N^\mathcal{A}$ such that $n_a \leq \lfloor \eps K \rfloor$, $n_A \in I_\eps^K ±1$ and
$j \in \{1,2,(1,2)\}$. Then
there exist two finite constants $c$ and $\eps_0$ such that for every $\eps\leq \eps_0$,
\begin{align*}
 p_{aa}^{(j)}(n) =\frac{r_j^* f_{a}(n_{a}-1)}{(n_a+1)(f_An_A+f_an_a)}, \quad p_{a A}^{(j)}(n)
=\frac{r_j^* f_An_A}{(n_a+1)(f_An_A+f_an_a)},
\end{align*}
\begin{equation}
  \label{rqpr2} p_{Aa}^{(j)}(n)\leq \frac{c\eps}{K\log K} \quad \text{and} \quad
 (1-c\eps )\frac{r_2}{n_A} \leq p_{AA}^{(2)}(n_A,k) \leq \frac{r_2}{n_A}, \ k \leq \lfloor \eps K \rfloor
\end{equation}
\end{lem}

\begin{proof}
The second equality is stated in \cite{smadi2014eco} Equation $(7.2)$.

Conditionally on the birth of an $a$-individual and the state of the process at the $(m-1)$-th jump, the probability of
picking the newborn when choosing an individual at random amongst the $a$-individuals is equal
to $1/(n_{a}+1)$. A recombination before the locus $Nj$ {(or before locus $N1$ and locus $N2$ if $j=(1,2)$)}
happens with probability $r_j^*$, independent of all other events. Finally, the probability that the second parent
is an $a$-individual but is different from the first parent is equal to $f_{a}(n_{a}-1)/(f_An_A + f_an_a)$. This proves the first equality.

When $n_A \in I_\eps^K$ we get similarly that $$p_{AA}^{(j)}(n) =\frac{r_j^* f_{A}(n_{A}-1)}{(n_A+1)(f_An_A+f_an_a)}\quad \text{and}\quad
p_{Aa}^{(j)}(n)
=\frac{r_j^* f_an_a}{(n_A+1)(f_An_A+f_an_a)},$$
and from the definition of $\tilde{N}$ in \eqref{deftildeN} we obtain that when $ dist (n_A,I_\eps^K)=1$,
$p_{AA}^{(2)}(n)=r_2(n_A-1)/n_A^2$ and $p_{Aa}^{(j)}(n)=0$. Condition \eqref{assrK} completes the proof.
\end{proof}

\begin{rem}\label{rem:recprob}
Let us recall the definition of $I_\eps^K$ in \eqref{compact}.
Then there exist three finite constants $c$, $\eps_0$ and $K_0$ such that for $\eps\leq \eps_0$,
$K\geq K_0$, $j \in \{1,2,(1,2)\}$, $n_A \in I_\eps^K±1$ and $k<\lfloor \eps K \rfloor$,
\begin{equation} \label{rqpr}
 (1-c\eps)\frac{r_j^*}{k+1}\leq   p_{aA}^{(j)}(n_A,k)\leq \frac{r_j^*}{k+1}\quad \text{and}\quad
 p_{aa}^{(2)}(n_A,k) \leq \frac{f_a}{f_A}\frac{r_2}{n_A}\leq \frac{c}{K\log K}.
\end{equation}
\end{rem}

Recalling the definitions of the $m$th jump time and the number of jumps in \eqref{deftpssauts}
 and \eqref{defJK1}, we define for $j \in \{1,2,(1,2)\}$, $m \in \N$
and an individual $i$ uniformly picked at the end of the first phase,
\begin{multline}\label{defassalpha}
 (\alpha ij)_{m}:=\{m \leq J^K(1) \text{ and the $j$-th locus/loci of the $i$-th individual is/are associated} \\
 \text{to an allele $\alpha$ at the $m$-th jump time}\}.
\end{multline}
The notation $(\alpha i1)_m, (\alpha' i2)_m$ here implies that the two neutral loci of individual $i$ are associated to two distinct
individuals at the $m$th jump time, for any $\alpha,\alpha'\in \mathcal{A}$.\\

To approximate the genealogy of the neutral alleles sampled at the end of the first phase we will focus on the recombinations
and coalescences which may happen during this time interval. 
Keep in mind that when looking at coalescing neutral loci, the parent's type may differ from the type of up to one child.
We first prove that we can neglect some event combinations.
Sample $2d$ distinct individuals uniformly at the end of the first phase (maximal number of ancestors for the
$2d$ neutral alleles sampled at the end of the sweep) and define:
\begin{enumerate}
 \item[$aAa$:] a neutral allele recombines from the $a$-population to the $A$-population, and then (backwards in time) back into the $a$-population
 \item[$CR$:] two neutral alleles coalesce in the $a$-population, and then (backwards in time) recombine into the $A$-population
 \item[$CA$:] two neutral alleles coalesce and at least one of them carries the allele $A$ at the time of coalescence
 \item[$2R$:] a neutral allele takes part in a double recombination (i.e. a recombination before $N1$ and a recombination
 before $N2$ at the same birth event)
\item[$R2a$:] a recombination separates the two neutral loci of an individual within the $a$-population
\end{enumerate}
We can bound the probability of these events as follows:

\begin{lem}\label{lemma_negevents}
 There exist three positive finite constants $c$, $K_0$ and $\eps_0$ such that for $\eps\leq \eps_0$ and $K\geq K_0$
$$ \P^{(1)}(aAa)+\P^{(1)}(CR)+\P^{(1)}(2R) +\P^{(1)}(R2a)\leq \frac{c}{\log K}, \quad \text{and} \quad \P^{(1)}(CA)\leq \frac{c\log K}{K}. $$
\end{lem}

\begin{proof}
 The probabilities of events $aAa$, $CR$ and $CA$ are bounded in \cite{smadi2014eco} Lemma $7.3$ and Equation (7.19) for the process $N$.
But according to Lemmas \ref{lempcoal} and \ref{lemma:recprob} the coalescence and recombination probabilities for the process $\tilde{N}$
are very close or even smaller when $ dist (n_A,I_\eps^K)=1$ than when $N$ and $\tilde{N}$ are equal.
 Hence we just have to bound the probability of $2R$ and $R2a$.
If a neutral allele experiences a double recombination, it happens either when it is associated with an allele $a$,
or with an allele $A$. From Lemma \ref{lemma:recprob} and the fact that $r_1$ and $r_2$ are of order $1/\log K$ we get for
$k <\lfloor \eps K \rfloor$:
$$ \sup_{n_A \in I_\eps^K±1}\Big( p_{aa}^{(1,2)}+p_{aA}^{(1,2)}\Big)(n_A,k)\leq \frac{c}{(k+1)\log^2 K} \ \
\text{and}
\ \sup_{n_A \in I_\eps^K±1}\Big( p_{Aa}^{(1,2)}+p_{AA}^{(1,2)}\Big)(n_A,k)\leq \frac{c}{K\log^2 K}. $$
Recall the definitions of $U_k^K(1)$ and $\mathcal{U}_k^K(1)$ in \eqref{Umjk} and \eqref{mathcalUk} respectively. As a birth of an $\alpha$-individual is needed to have a recombination from the $\alpha$- to the $\alpha'$-population, we can bound the probability to have a double recombination by:
$$ \P^{(1)}(2R)\leq \frac{c}{\log^2 K} \E^{(1)} \Big[\sum_{k=1}^{\lfloor \eps K \rfloor -1} \Big( \frac{U_k^K(1)}{k+1}+ \frac{\mathcal{U}_k^K(1)}{K} \Big) \Big].$$
By applying inequality \eqref{expupa2} and Lemma \ref{approxsummathcalUk} we succeed in bounding $\P^{(1)}(2R)$ by a constant over $\log K$. It remains to consider the event $R2a$ of a recombination within the $a$-population. 
Define the first time (with respect to the backwards in time process) that this event happens:
\begin{align}
\begin{aligned}\label{Raa(i)}
 R^{(1)}_{aa}(i) := \sup \{ m ,&\  m \leq J^{K}(1) \text{ and both neutral loci of the $i$-th individual are}\\
&\ \text{ associated to distinct $a$-individuals at the $(m-1)$th jump},
\end{aligned}
\end{align}
where $R^{(1)}_{aa}(i)=-\infty$ if the event does not happen during the first phase of the sweep. Then,
\begin{multline*}
 \P^{(1)}(R^{(1)}_{aa}(i)\geq 0) = \sum_{l=1}^{\lfloor \eps K \rfloor -1} \P^{(1)}(R^{(1)}_{aa}(i)\geq 0, \tilde{N}_a(\tau^{K}_{R^{(1)}_{aa}(i)})=l)\\
=\sum_{l=1}^{\lfloor \eps K \rfloor -1} \sum_{m< \infty} \P^{(1)}(m \leq J^{K}(1), \tilde{N}_a(\tau^{K}_{m-1})=l,\tilde{N}_a(\tau^{K}_{m})=l+1, 
(ai1)_m,(ai2)_m,\forall m'> m: (ai12)_{m'})\\
\leq \sum_{l=1}^{\lfloor \eps K \rfloor -1} \sum_{m< \infty} \sup_{n_A\in I_\eps^K±1}\Big( p_{aa}^{(2)} (n_A,l)
 \P^{(1)}_{(n_A,l+1)}(\forall m\geq 0: (ai12)_{m})\Big) \P^{(1)}(m \leq J^{K}(1), \tilde{N}_a(\tau^{K}_{m-1})=l,\tilde{N}_a(\tau^{K}_{m})=l+1)\\
\leq \sum_{l=1}^{\lfloor \eps K \rfloor -1} \frac{c}{K\log K}\E^{(1)} [U_l^K(1)] \leq \frac{c}{\log K}, \end{multline*}
by \eqref{expupa2} and \eqref{rqpr}.
\end{proof}
To simplify the notations we will denote the union of all negligible events by
\begin{align}\label{def:negevents}
 NE := aAa \cup CR \cup CA \cup 2R  \cup R2a.
\end{align}

\subsection{The two loci of one individual separate within the $A$-population}\label{subsec:2locisep}
Having excluded events of small probability, there are exactly two ways for the
neutral alleles of an individual sampled at the end of the first phase to originate from two distinct $A$-individuals.
The two possibilities were already described on page \pageref{defgenealogy} and represented in Figure \ref{schemagendefevent}.
The ideas which are pursued in this section are similar to the ones from \cite{brink2014multsweep}, but there are extra difficulties due
to the randomness of the population size.

\subsubsection{Event $[2,1]^{rec}_{A,i}$}
 The aim of this section is to prove the following approximation:

\begin{pro}\label{lemma_21Arec}
Let $i$ be an $a$-individual sampled uniformly at the end of the first phase.
There exist two finite constants $c$ and $\eps_0$ such that for $\eps \leq \eps_0$,
$$ \underset{K \to \infty}{\limsup}\hspace{.1cm}\Big| \P^{(1)}([2,1]^{rec}_{A,i})-
\Big[ \frac{r_2}{r_1+r_2}-e^{-\frac{r_1}{s}\log \lfloor \eps K\rfloor }+\frac{r_1}{r_1+r_2}e^{-\frac{r_1+r_2}{s}\log \lfloor \eps K\rfloor} \Big]\Big|\leq c\sqrt{\eps}. $$
\end{pro}
We first give a preliminary Lemma before proving Proposition \ref{lemma_21Arec}.
Recall \eqref{defJK1} and define for $j \in \{1,2,(1,2)\}$ and $m \in \N$,
\begin{multline}\label{def:Rij}
 R(i,j):= \sup \{m , m \leq J^K(1) \text{ and the $j$-th locus/loci of the $i$-th individual} \\
 \text{is/are associated to an allele $A$ at the $(m-1)$th jump time}\} ,
\end{multline}
the last jump (forwards in time) when the $j$-th locus/loci of the $i$-th individual belongs to the A-population
(with $\sup \emptyset=-\infty$).
 To prove Proposition \ref{lemma_21Arec} the idea is to decompose the event $[2,1]^{rec}_{A,i}$ according to the different possible $a$-population sizes when
 the first
(backwards in time) recombination between $N1$ and $N2$ occurs.
\begin{multline} \label{decotheo1}
 \P^{(1)}([2,1]^{rec}_{A,i}) = \P^{(1)}(R(i,2)>R(i,1) \geq 0)\\
= \sum_{l=1}^{\lfloor\eps K \rfloor} \P^{(1)}(R(i,1) \geq 0, R(i,2)>R(i,1),\tilde{N}_a({\tau^K_{R(i,2)}})=l) \\
= \sum_{l=1}^{\lfloor\eps K \rfloor-1} \P^{(1)}(R(i,2)>R(i,1),\tilde{N}_a(\tau^K_{R(i,2)})=l) \P^{(1)}(R(i,1) \geq 0 | R(i,2)>R(i,1),\tilde{N}_a({\tau^K_{R(i,2)}})=l).
\end{multline}
In the following Lemma, which then gives rise to the proof of Proposition \ref{lemma_21Arec}, we consider separately
the two probabilities of the above product:

\begin{lem}\label{claim1}
There exist three finite constants $c$, $K_0$ and $\eps_0$ such that for $K \geq K_0$, $\eps \leq \eps_0$ and $l<\lfloor \eps K \rfloor$,
 \begin{multline}\label{541}
\Big| \P^{(1)}(R(i,2) > R(i,1), \tilde{N}_a(\tau^K_{R(i,2)})=l)-
  r_2\frac{1-(1-s)^{\lfloor \eps K \rfloor-l}-(1-s)^{l+1}}{s(l+1)}e^{- \frac{r_1+r_2}{s}\log\frac{\lfloor \eps K \rfloor}{l}}
  \Big|\leq \frac{c\sqrt{\eps}}{l\log K}
 \end{multline}
and
 \begin{equation}\label{542} \Big| \P^{(1)}(R(i,1) \geq 0 | R(i,2)>R(i,1),N_a^K(\tau^K_{R(i,2)})=l) -
\sum_{k=1}^{l-1} \frac{r_1}{s(k+1)} e^{ -\frac{r_1}{s} \log \frac{l-1}{k} }\Big|\leq c \sqrt{\eps}.\end{equation}
\end{lem}

\begin{proof}[Proof of Proposition \ref{lemma_21Arec}]

From Lemma \ref{claim1} and Equation \eqref{decotheo1} we get the existence of a finite $c$ such that for
$K$ large enough and $\eps$ small enough,
\begin{equation}\label{firstineq}
  \P^{(1)}([2,1]^{rec}_{A,i}) \leq
  \sum_{l=1}^{\lfloor\eps K \rfloor-1}\Big[ \frac{r_2}{s(l+1)}
e^{ -\frac{r_1+r_2}{s}\log\frac{\lfloor \eps K \rfloor}{l}}+\frac{c\sqrt{\eps}}{l \log K}\Big] \Big[\sum_{k=1}^{l-1} \frac{r_1}{s(k+1)} e^{ -\frac{r_1}{s} \log  \frac{l-1}{k} }+{c\sqrt{\eps}}\Big].
\end{equation}
Rewriting the second term in brackets and applying Lemma \ref{equivalent} with $c_N/\log N=r_1/s$ yields:
\begin{eqnarray*}
  e^{-\frac{r_1}{s}\log (l-1)} \frac{r_1}{s}\sum_{k=1}^{l-1} \frac{1}{k+1} e^{\frac{r_1}{s}\log k}+{c\sqrt{\eps}}& \leq &
  e^{-\frac{r_1}{s}\log (l-1)}\Big(e^{\frac{r_1}{s}\log l}-1+c\frac{r_1}{s}\Big)+c\sqrt{\eps}\\
& \leq & 1-e^{-\frac{r_1}{s}\log l } + c\sqrt{\eps},
\end{eqnarray*}
for $K$ large enough, $\eps$ small enough and a finite $c$, whose value can change from line to line and which can be chosen independently of $l$.
We use in the last inequality Condition \eqref{assrK} which claims that $\limsup_{K \to \infty}r_1\log K<\infty$.
Including the last inequality in \eqref{firstineq} gives
\begin{eqnarray*}
  \P^{(1)}([2,1]^{rec}_{A,i}) &\leq &
  \sum_{l=1}^{\lfloor\eps K \rfloor-1} \frac{r_2}{s(l+1)}e^{-\frac{r_1+r_2}{s}\log \lfloor\eps K \rfloor}\Big(e^{\frac{r_1+r_2}{s}\log l}-e^{\frac{r_2}{s}\log l} \Big) + c\sqrt{\eps},
\end{eqnarray*}
for a finite $c$, $K$ large enough and $\eps$ small enough, where we again use \eqref{assrK} which ensures that exponential
terms are bounded away from zero and infinity in the following sense:
$$ \frac{1}{c}\leq \liminf_{K \to \infty} e^{-\frac{r_1+r_2}{s}\log \lfloor \eps K\rfloor}\leq
\limsup_{K \to \infty} e^{\frac{r_1+r_2}{s}\log \lfloor \eps K\rfloor}\leq c $$
for a positive and finite $c$. Applying again Lemma \ref{equivalent}, we get:
$$  \P^{(1)}([2,1]^{rec}_{A,i}) \leq  \Big( \frac{r_2}{r_1+r_2}-e^{-\frac{r_1}{s}\log \lfloor\eps K \rfloor}+\frac{r_1}{r_1+r_2}e^{-\frac{r_1+r_2}{s}\log \lfloor\eps K \rfloor} \Big)+c\sqrt{\eps}. $$
The lower bound is obtained in the same way. Notice that it is a little bit more involved as we need to use \eqref{lemma3.5} in addition.
\end{proof}

The end of this section is devoted to the proof of Lemma \ref{claim1}.

\begin{proof}[Proof of Equation \eqref{541}]
We can decompose the event
$ \{ R(i,2) > R(i,1), \tilde{N}_a(\tau^K_{R(i,2)})=l \} $ according to the jump number of the (backwards in time) first recombination.
Recall the definition of $NR(i)^{(1)}$ on page \pageref{defgenealogy}. We will use this event with a different
initial condition for $(\tilde{N}_A,\tilde{N}_a)$, which will not necessarily be $(\lfloor \bar{n}_AK \rfloor,1)$.
It will however still correspond to the absence of
any recombination before the end of the first phase.
We recall conventions \eqref{convP} and \eqref{defP1}.
With the definition of $(\alpha ik)_m$ in \eqref{defassalpha} we get
\begin{multline} \label{decompproba1}
 \P^{(1)}(R(i,2) > R(i,1), N_a(\tau^K_{R(i,2)})=l) \\
=\sum_{m>1} \P^{(1)}(m\leq  J^{K}(1), \tilde{N}_a(\tau_{m-1}^{K})=l-1, \tilde{N}_a(\tau_{m}^{K})=l,
 (ai1)_{m-1}, (Ai2)_{m-1}, \forall m\leq m'\leq  J^{K}(1) : (ai12)_{m'})\\
\leq \sum_{m>1} \underset{n_A \in I_\eps^K±1}{\sup} \Big\{p_{aA}^{(2)}(n_A,l-1)\P^{(1)}_{(n_A,l)}(NR(i)^{(1)}) \Big\}
\P^{(1)}( m\leq J^{K}(1), \tilde{N}_a(\tau_{m-1}^{K})=l-1, \tilde{N}_a(\tau_{m}^{K})=l)  \\
 = \underset{n_A \in I_\eps^K±1}{\sup} \Big\{p_{aA}^{(2)}(n_A,l-1)\P^{(1)}_{(n_A,l)}(NR(i)^{(1)}) \Big\}\E^{(1)}[U_{l-1}^K(1)],
\end{multline}
and the same expression with the infimum on $n_A \in I_\eps^K±1$ for a lower bound.
Adding \eqref{rqpr} and \eqref{lemPNRml} yields,
\begin{eqnarray*}
 \P^{(1)}(R(i,2) > R(i,1), \tilde{N}_a(\tau^K_{R(i,2)})=l)
&\leq &(1+c{\eps}) \frac{r_2}{l+1}(e^{- \frac{r_1+r_2}{s}\log\frac{\lfloor \eps K \rfloor}{l}}+c\sqrt{\eps})\E^{(1)}[U_{l-1}^K(1)]
\\
&\leq &(1+c\sqrt{\eps}) \frac{r_2}{l+1}e^{- \frac{r_1+r_2}{s}\log\frac{\lfloor \eps K \rfloor}{l}}
\E^{(1)}[U_{l-1}^K(1)],
\end{eqnarray*}
for a finite $c$, $\eps$ small enough and $K$ large enough, where we used that
$(r_1+r_2)\log K$ is bounded. We similarly get a lower bound and end up the proof of Equation
\eqref{541} by applying \eqref{expupa2}.
\end{proof}

\begin{proof}[Proof of Equation \eqref{542}]
We will decompose the event considered here according to the value of $\tilde{N}_a$ when the first (backwards in time) recombination
occurs.
Let us denote by  $\zeta_k^K{(1)}$ the jump number of the last hitting of $k\leq \lfloor  \varepsilon K \rfloor$
 by $\tilde{N}_a$ during the first phase,
\begin{equation}\label{zeta} \zeta_k^K{(1)}: = \sup \{ m, \tau_m^K \leq \tilde{T}_\eps^K,
\tilde{N}_a(\tau_m^K)=k \}, \end{equation}
and recall \eqref{defsigma}. Then we can
define the events
\begin{multline}\label{NRlsigmai} NR(l,\xi,i):=\{\text{the first locus of individual $i$ sampled at jump time $\tau^K_{\xi}$}\\
\text{does not recombine from the $a$- to the $A$-population between $0$ and $\tau_{\xi}^K$}\} \end{multline}
where $\xi \in \{\zeta_l^K(1),\sigma_l^K(1)\}$.
Similarly as in \eqref{decompproba1}, Bayes' rule leads to:
\begin{eqnarray} \label{decompoproba2}
&& \P^{(1)}(R(i,1) \geq 0 \mid R(i,2)>R(i,1),\tilde{N}_a(\tau^K_{R(i,2)})=l)\\
&&\quad =\sum_{k=1}^{\lfloor\eps K \rfloor}\P^{(1)}(R(i,1) \geq 0, \tilde{N}_a(\tau^K_{R(i,1)}){=k} \mid {R(i,2)>R(i,1), }\tilde{N}_a(\tau^K_{R(i,2)})=l),\nonumber\\
&&\quad \leq \sum_{k=1}^{\lfloor\eps K \rfloor} \Big(\underset{n_A \in I_\eps^K±1}{\sup} p_{aA}^{(1)}(n_A,k-1)
 \P^{(1)}_{(n_A,k)}(NR(l,\sigma,i))\Big)\mathcal{S}(k,l),\nonumber
\end{eqnarray}
where for the sake of simplicity we have introduced the notation
$$ \mathcal{S}(k,l):= \sum_{m<\infty} \P^{(1)}(m< R(i,2),  \tilde{N}_a(\tau_{m-1}^{K})=k-1, \tilde{N}_a(\tau_{m}^{K})=k \mid \tilde{N}_a(\tau^K_{R(i,2)})=l). $$
The lower bound is obtained by taking the infimum for $n_A$ in $I_\eps^K±1$ and replacing $\sigma$ by $\zeta$.
To lighten the proof, we bound the probability in the brackets for both $\sigma$ and $\zeta$ in Lemma \ref{probeventphase1}, Equation \eqref{prob:ai1_m'}.

First we prove that with a probability close to one the $a$-population size is bigger when the (backwards in time) first recombination occurs than when the second, of locus $(i,1)$, occurs.
Note that by \eqref{expupa2} and Lemma \ref{excunder}, there exists a finite $c$ such that for every
$l< k < \lfloor \eps K \rfloor$:
$$
\mathcal{S}(k,l)
 \leq \E^{(1)}[U_{l}^{K}(1)]\sup_{n_A \in I_\eps^K±1}\hspace{.1cm} \E^{(1)}_{(n_A,l+1)}[U_{n_A,l,k-1}^K(1)|\sigma_l^K(1)<\infty]
\leq c \mu_\eps^{k-l},
$$
where we recall that $\mu_\eps<1$ for $\eps$ small enough.
Hence, recalling \eqref{decompoproba2} and \eqref{rqpr}, we obtain for $k >l $
\begin{equation*}
 \P^{(1)}(R(i,1) \geq 0, \tilde{N}_a^K(\tau^K_{R(i,1)})\geq l | R(i,2)>R(i,1),\tilde{N}_a(\tau^K_{R(i,2)})=l)\leq
cr_1 \sum_{k=l+1}^{\lfloor\eps K \rfloor} \frac{\mu_\eps^{k-l}}{k}
  \leq  \frac{c}{\log K},
\end{equation*}
for a finite $c$ and $\eps$ small enough, which entails
\begin{multline*}
\P^{(1)}(R(i,1) \geq 0 \mid R(i,2)>R(i,1),\tilde{N}_a(\tau^K_{R(i,2)})=l)\\
\leq \sum_{k=1}^{l} \Big(\underset{n_A \in I_\eps^K±1}{\sup} p_{aA}^{(1)}(n_A,k-1)
 \P^{(1)}_{(n_A,k)}(NR(l,\sigma,i))\Big)\mathcal{S}(k,l) + O\left( \frac{1}{\log K} \right).
\end{multline*} We therefore can ignore all $k> l$ in the sum in
\eqref{decompoproba2} and continue with the case $k\leq l$.
In this setting, we can bound the sum $\mathcal{S}(k,l)$ as follows:
$$
\E^{(1)}[U_{k-1}^{K}(1)] - \underset{n_A \in I_\eps^K±1}{\sup}\hspace{.1cm} \E^{(1)}_{(n_A,l-1)}[U_{n_A,l,k-1}^{K}(1)]\E^{(1)}[U_{l}^{K}(1)]
\leq \mathcal{S}(k,l)\leq \E^{(1)}[U_{k-1}^{K}(1)].$$
Bounding the difference between the two bounds above within Equation \eqref{decompoproba2} then yields
\begin{equation*}
 \sum_{k=1}^{l} \frac{r_1}{k}   \underset{n_A \in I_\eps^K±1}{\sup}\hspace{.1cm} \E^{(1)}_{(n_A,l-1)}[U_{n_A,l,k-1}^{K}(1)]
 \E^{(1)}[U_{l}^{K}(1)]\\
\leq cr_1 \sum_{k=1}^{l} \frac{\mu_\eps^{l-k}}{k} \leq \frac{c }{\log K},
\end{equation*}
 for a finite $c$ by \eqref{rqpr}, \eqref{expupa2} and Lemma \ref{excunder}.
As a consequence,
$$
  \sum_{k=1}^{l} \Big(\underset{n_A \in I_\eps^K±1}{\sup} p_{aA}^{(1)}(n_A,k-1)
 \P^{(1)}_{(n_A,k)}(NR(l,\sigma,i))\Big)\Big|\mathcal{S}(k,l)-\E^{(1)}[U_{k-1}^{K}(1)]\Big|\leq  O\left( \frac{1}{\log K} \right),
$$
and thus we can work with $\E^{(1)}[U_{k-1}^{K}(1)]$ as an approximation for
the sum $\mathcal{S}(k,l)$:
\begin{multline*}
\P^{(1)}(R(i,1) \geq 0 \mid R(i,2)>R(i,1),\tilde{N}_a(\tau^K_{R(i,2)})=l)\\
\leq \sum_{k=1}^{l} \Big(\underset{n_A \in I_\eps^K±1}{\sup} p_{aA}^{(1)}(n_A,k-1)
 \P^{(1)}_{(n_A,k)}(NR(l,\sigma,i))\Big)\E^{(1)}[U_{k-1}^{K}(1)] + O\left( \frac{1}{\log K} \right).
\end{multline*}

Reasoning in the same way to get a lower bound and using \eqref{rqpr} and \eqref{prob:ai1_m'}
 we get the existence of a
finite $c$ such that for $K$ large
enough and $\eps$ small enough,
$$\Big| \P^{(1)}(R(i,1) \geq 0 | R(i,2)>R(i,1),\tilde{N}_a^K(\tau^K_{R(i,2)})=l) -
\sum_{ k=1}^{l-1} \frac{r_1}{k} e^{ -\frac{r_1}{s} \log  \frac{l-1}{k} }\E^{(1)}[U_{k}^{K}(1)]\Big|\leq c \sqrt{\eps}.$$
Applying \eqref{expupa2} and \eqref{lemma3.5} yields Equation \eqref{542}.
Notice that we have replaced $1/k$ by $1/(k+1)$. We used Condition \eqref{assrK} to do this.
\end{proof}

\subsubsection{Event $[12,2]^{rec}_{A,i}$} Recall the definition of $[12,2]^{rec}_{A,i}$ on page \pageref{defgenealogy}.
This section is devoted to the proof of the following result:

\begin{pro}\label{lemma_122Arec}
Let $i$ be an individual sampled uniformly at the end of the first phase.
There exist two finite constants $c$ and $\eps_0$ such that for $\eps\leq \eps_0$,
$$
\limsup_{K \to \infty}\Big|\P^{(1)}([12,2]^{rec}_{A,i})-r_1\Big[\frac{1-e^{- \frac{r_1+r_2}{s}\log\lfloor \eps K \rfloor}}{r_1+r_2}
+ \frac{e^{- \frac{r_1+r_2}{s}\log\lfloor \eps K \rfloor}-e^{- \frac{f_Ar_2}{f_as}\log\lfloor \eps K \rfloor}}{r_1+r_2(1-f_A/f_a)}\Big]\Big| \leq c \sqrt{\eps}.
$$
\end{pro}

\begin{proof}
As the proof is very similar to the proof of Proposition \ref{lemma_21Arec} we will be very brief here and only give
the ingredients.
Let us introduce for $l<\lfloor \eps K \rfloor$ the event:
\begin{equation}\label{defRAli}
 RA(l,i):=\{ [12,2]^{rec}_{A,i} \mid R(i,2)=R(i,1)\geq 0, \tilde{N}_a(\tau_{R(i,1)}^{K})=l \}.
\end{equation}
Then we can rewrite the probability of $[12,2]^{rec}_{A,i}$ as follows:
\begin{equation}\label{P122}
   \P^{(1)}([12,2]^{rec}_{A,i})
=\sum_{l=1}^{\lfloor\eps K \rfloor} \P^{(1)}(RA(l,i))\P^{(1)}(R(i,2)=R(i,1)\geq 0, \tilde{N}_{a}^K(\tau_{R(i,1)}^{K})=l).
\end{equation}
Apart from the point of recombination, the second probability in the above sum coincides with the probability studied
in \eqref{541} and we obtain for $\eps$ small enough and $K$ large enough,
 \begin{multline}\label{approxmirror}
  \sup_{l \leq \lfloor \eps K \rfloor}\ l {\cdot}\Big| \P^{(1)}(R(i,2)=R(i,1)\geq 0, N_{a}^K(\tau_{R(i,1)}^{K})=l)-\\
  \frac{r_1(1-(1-s)^{\lfloor \eps K \rfloor-l}-(1-s)^{l+1})}{s(l+1)}e^{- \frac{r_1+r_2}{s}\log\frac{\lfloor \eps K \rfloor}{l}}
  \Big|\leq c \frac{\sqrt{\eps}}{\log K},
 \end{multline}
for a finite $c$, when substituting $r_2$ by $r_1$ in the fraction which mirrors the recombination probability.
The probability of $RA(l,i)$ is derived in Lemma \ref{probeventphase1}.
Inserting \eqref{approxmirror} and \eqref{approxhatPRAl} into \eqref{P122} yields
\begin{eqnarray*}
   \P^{(1)}([12,2]^{rec}_{A,i})
&\leq &  \sum_{l=1}^{\lfloor\eps K \rfloor} ( 1-e^{- \frac{f_A}{f_a} \frac{r_2}{s} \log l })\frac{r_1}{l+1}e^{-\frac{r_1+r_2}{s} \log\frac{\lfloor \eps K \rfloor}{ l}}+c\sqrt{\eps}
\\
&\leq &r_1e^{- \frac{r_1+r_2}{s}\log\lfloor \eps K \rfloor}\Big[ \frac{e^{\frac{r_1+r_2}{s}\log\lfloor \eps K \rfloor} -1}{r_1+r_2}  -\frac{e^{\frac{r_1+r_2-f_Ar_2/f_a}{s}\log\lfloor \eps K \rfloor} -1}{r_1+r_2-f_Ar_2/f_a}\Big]+c\sqrt{\eps}
\end{eqnarray*}
where we again applied Lemma \ref{equivalent} to express the sum in a different way,
and used the {finiteness} of $\limsup_{K \to \infty}(r_1+r_2)\log K$ assumed in Condition \eqref{assrK}. Reasoning similarly for the lower
bound and rearranging the terms end the proof of
Proposition \ref{lemma_122Arec}.
\end{proof}

\subsection{Proof of Proposition \ref{prop:1phase_probg1}}\ 

\noindent \textbf{Event $R2(i)^{(1)}$:} By definition and from Lemma \ref{lemma_negevents},
$$ \P^{(1)}(R2(i)^{(1)})
 =\P^{(1)}(R(i,2)\geq 0) - \P^{(1)}(R(i,1)\geq 0)+ O\Big(\frac{\log K}{K}\Big), $$
where $R(i,1)$ and $R(i,2)$ have been defined in \eqref{def:Rij}. But these probabilities have already been derived in
\cite{smadi2014eco} Lemma 7.4,
and we get:
$$ \P^{(1)}(R2(i)^{(1)})
 =(1-q_1q_2) - (1-q_1)+O_K(\eps)=q_1(1-q_2)+O_K(\eps), $$
 where $O_K(\eps)$ satisfies \eqref{defOKeps}.

\noindent \textbf{Event $R1|2(i)^{(1,ga)}$:} By definition (see page \pageref{defgenealogy})
 $$ \P^{(1)}(R1|2(i)^{(1,ga)})=\P^{(1)}([2,1]^{rec}_{A,i})+\P^{(1)}([12,2]^{rec}_{A,i}) . $$
 The result then follows from Propositions \ref{lemma_21Arec} and \ref{lemma_122Arec}.\\

\noindent \textit{Event $R12(i)^{(1)}$:}
From Definition \eqref{defRAli} and Equation \eqref{approxhatPRAl} we obtain for $K$ large enough,
\begin{eqnarray*}
\P^{(1)}( R12(i)^{(1)})&=&
\sum_{l=1}^{\lfloor\eps K \rfloor}(1-\P^{(1)}(RA(l,i))) \P^{(1)}(R(i,1)=R(i,2)\geq 0, \tilde{N}_a(\tau_{R(i,2)}^{K})=l)\\
&=& r_1\sum_{l=1}^{\lfloor\eps K \rfloor} e^{- \frac{f_A}{f_a} \frac{r_2}{s} \log l } \frac{1-(1-s)^{\lfloor \eps K \rfloor-l}-(1-s)^{l+1}}{s(l+1)}
e^{- \frac{r_1+r_2}{s}\log\frac{\lfloor \eps K \rfloor}{l}}+O_K(\sqrt{\eps})\\
&= &\frac{r_1}{r_1+r_2-f_Ar_2/f_a}\Big(e^{-\frac{r_2}{s}\frac{f_A}{f_a}\log\lfloor \eps K \rfloor}-e^{- \frac{r_1+r_2}{s}\log\lfloor \eps K \rfloor}\Big)+O_K(\sqrt{\eps}),
\end{eqnarray*}
where  we again used the statement of Lemma \ref{equivalent} to substitute
the sum, as well as Equation \eqref{lemma3.5}.\\

\noindent \textit{Event $NR(i)^{(1)}$:}
From Lemma \ref{lemma_negevents},
$$ \P^{(1)}(NR(i)^{(1)})=1-\P^{(1)}(R2(i)^{(1)})-\P^{(1)}( R12(i)^{(1)})-\P^{(1)}(R2(i)^{(1)})+O\Big(\frac{\log K}{K} \Big). $$
This ends up the proof of Proposition \ref{prop:1phase_probg1}. \hfill $\square$

\section{Second and third phases}\label{secondphase}

This section is devoted to the proofs of Propositions \ref{prosecondphase} and \ref{prothirdphase}.

\subsection{Proof of Proposition \ref{prosecondphase}}

We need to show that two distinct lineages picked uniformly at the end of the second phase coalesce or recombine during that phase only with negligible probability.
Let us recall the definition of the jumps $\tau_m^K$ in \eqref{deftpssauts} and
denote by $U^K(2)$ the number of upcrossings of the $a$-population during the second phase:
\begin{equation}\label{defU2K} U^K(2) :=\#\{ m, T_\eps^K<\tau_m^K\leq T_\eps^K + t_\eps, N_{a}({\tau}_{m+1}^K)-N_a({\tau}_{m}^K)=1 \}. \end{equation}
Let us introduce the event $C_\eps^K$:
$$ C_\eps^K:= \{T_\eps^K \leq S_\eps^K\}\cap \{N_a^K(t) \geq \eps^2K/4, \forall \ T_\eps^K \leq t \leq T_\eps^K +t_\eps\}.$$
In particular on the event $C_\eps^K$, for $T_\eps^K \leq \tau_m^K \leq T_\eps^K+t_\eps$ and $j \in \{1,2\}$
$$p_{aA}^{(j)}(N(\tau_m^K))\leq \frac{8r_j}{\eps^2 K}\quad \text{and} \quad
 p_{aa}^{(c,j)}(N(\tau_m^K))\leq \frac{32}{\eps^4 K^2}. $$
Then if we recall the definition of $NR(i)^{(2)}$ on page \pageref{defgenealogy3} we have for $m \in \N$,
\begin{equation}\label{stepprop2} \P^{(1)}(NR(i)^{(2)}|U^K(2)=m,C_\eps^K )
\geq \Big( 1- \frac{8(r_1+r_2)}{\eps^2 K} \Big)^m. \end{equation}
But for $K$ large enough, $\log (1-8(r_1+r_2)/(\eps^2 K))\geq -10(r_1+r_2)/(\eps^2 K)$ and hence
\begin{eqnarray*}
 \P^{(1)}(NR(i)^{(2)}|C_\eps^K)&\geq & \Big(1- \P^{(1)}(U^K(2)>  K\log\log K|C_\eps^K)\Big)
 e^{ K\log\log K\log (1-\frac{8(r_1+r_2)}{\eps^2 K})}\\
&\geq & \Big(1- \P^{(1)}(U^K(2)>  K\log\log K|C_\eps^K)\Big)
 e^{ -\frac{10(r_1+r_2) \log\log K}{\eps^2 }}.
\end{eqnarray*}
According to Condition \eqref{assrK} the exponential term is equivalent to $1$ when $K$ is large.
Moreover, by \eqref{result_champa}, $N_a^K$ is smaller than $2\bar{n}_aK$ on the time interval $[T_\eps^K, T_\eps^K+t_\eps]$
with probability close to $1$. When this property holds, we can bound the birth number $U^K(2)$ by
the sum of $2\bar{n}_aK$ iid Poisson random variables with parameter $f_at_\eps$. The strong law of large numbers then yields
$$ \underset{K \to\infty}{\lim}\P^{(1)}(U^K(2)>  K\log\log K|C_\eps^K)=0. $$
Applying again \eqref{result_champa} to get
$ {\lim}_{K \to\infty}\P(C_\eps^K|T_\eps^K<\infty)=1 $
finally gives
$$  \underset{K \to\infty}{\lim}\P(NR(i)^{(2)}|T_\eps^K<\infty)=1.$$
The coalescence part in Proposition \ref{prosecondphase} can be proven in the same way.

\subsection{Proof of Proposition \ref{prothirdphase}}
The proof of the asymptotic probability of $R2(i)^{(3,ga)}$ is the same as for \eqref{approxhatPRAl}, except that the roles of
$A$ and $a$ are exchanged. Hence we do not give more details.
Note however that it extensively uses Lemma \ref{lemespvarU}.
Let us now focus on the event $NR(i)^{(3)}$,
and introduce
$$NRA(i)^{(3)}:=\{\text{no neutral allele of individual $i$ recombines from the $a$ to the $A$ population}\}.$$
Recall the definitions of $\P^{(3)}$ and $U^K(3)$ in \eqref{defP1} and \eqref{UDk3phase} respectively. We decompose the probabilities according to the
number of upcrossings of $\tilde{\tilde{N}}_a$ during the third phase and get in the same way as in \eqref{stepprop2}, for $m\in \N$
$$ \P^{(3)}(NRA(i)^{(3)}|U^K(3)=m,\{\tilde{\tilde{T}}_0^{(K,A)}< \tilde{\tilde{T}}_\eps^{(K,A)}\wedge S_\eps^{(K,a)}\})
\geq \Big( 1- \frac{f_A(r_1+r_2)\eps}{f_a(\bar{n}_a-M''\eps)^2 K} \Big)^m, $$
where we recall that
$\tilde{\tilde{T}}_0^{(K,A)}$ and $\tilde{\tilde{T}}_\eps^{(K,A)}$ are the analogs of
 ${T}_0^{(K,A)}$ and $T_\eps^{(K,A)}$ (defined in \eqref{T0K}) for the process $\tilde{\tilde{N}}$.
But for $K$ large enough and $\eps$ small enough,
$$\log \Big(1-\frac{f_A(r_1+r_2)\eps}{f_a(\bar{n}_a-M''\eps)^2 K}\Big)\geq
-2f_A\frac{(r_1+r_2)\eps}{f_a\bar{n}_a^2 K}.$$
Hence we get for a finite constant $c$ and $\eps$ small enough:
\begin{eqnarray*}
\P^{(3)}(NRA(i)^{(3)})& \geq &
 \Big(1- \P^{(3)}\Big(U^K(3)>  \frac{K\log K}{\sqrt{\eps}}\Big)\Big)
\exp\Big( -\frac{2f_A(r_1+r_2)\sqrt{\eps} \log K}{f_a\bar{n}_a^2 } \Big)\\
& \geq & \Big(1- \frac{\sqrt{\eps}\E^{(3)}[U^K(3)]}{  K\log K }\Big)\Big( 1-\frac{2f_A(r_1+r_2)\sqrt{\eps} \log K}{f_a\bar{n}_a^2 } \Big)
 \geq  (1- {c\sqrt{\eps}})^2,
\end{eqnarray*}
where we used  Lemma \ref{lemespvarU} and that
$(r_1+r_2)\log K$ is bounded (Condition \eqref{assrK}).\\

The proof of the last part of Proposition \ref{prothirdphase} is very similar to that of Proposition \ref{prosecondphase}.
 The key arguments
are that the expectation of
the birth number of $a$-individuals during the third phase under $\P^{(3)}$ is of order $K\log K$ (Lemma \ref{lemespvarU}),
 whereas the probability for two neutral alleles associated with an allele
$a$ to coalesce is of order $1/K^2$ at each birth of an $a$-individual (Lemma \ref{lempcoal}).

{\section{Independence of neutral lineages}\label{proofindep}
This section is dedicated to the proof of Proposition \ref{proindi}.
We sample $d$ distinct individuals uniformly at the end of the first phase.
We recall the definitions of the genealogical events during the first phase on page \pageref{defgenealogy}
 and introduce:
$$
 R(1|2):=\underset{1\leq i \leq d}{\sum} \mathbf{1}_{R1|2(i)^{(1,ga)}}, \quad R(1):=R(1|2)+ \underset{1\leq i \leq d}{\sum} \mathbf{1}_{R12(i)^{(1)}}
\quad \text{and} \quad
 R(2):=R(1)+ \underset{1\leq i \leq d}{\sum} \mathbf{1}_{R2(i)^{(1)}}
$$
From Proposition \ref{prop:1phase_probg1} we know that
$R(1)$, $R(2)$ and $R(1|2)$ are sufficient to describe the neutral genealogies at the end of the first phase up to a probability negligible
with respect to one for large $K$.
Let $j,k,l$ be three integers such that $l\leq j$ and $j+k\leq d.$
We aim at approximating
\begin{align}\label{decompoindep}
 p(j,k,l):&=\P(R(1)=j,R(2)=j+k,{R(1|2)}=l|T_\eps^K\leq S_\eps^K)\\
&=\P(R(1)=j|T_\eps^K\leq S_\eps^K)\P(R(2)=j+k|T_\eps^K\leq S_\eps^K,R(1)=j)\nonumber\\
& \hspace{1cm}\P({R(1|2)}=l|T_\eps^K\leq S_\eps^K,R(1)=j,R(2)=j+k).\nonumber
\end{align}
The approximations of the two first probabilities are direct adaptations of Lemma 5.2 and the proof of Proposition 2.6 in
\cite{schweinsberg2005random}, pp 1623-1624. More precisely, Lemma 7.3 in \cite{smadi2014eco} which states that with high probability
two neutral lineages do not coalesce and then recombine (backwards in time) allows us to get an equivalent of Lemma 5.2 (with $J=0$) in
\cite{schweinsberg2005random}:
$$ \Big| \P(R(1)=j|T_\eps^K\leq S_\eps^K)-{d\choose j} \E[F_1^j(1-F_1)^{n-j}|T_\eps^K \leq S_\eps^K] \Big|\leq c\Big(\frac{1}{\log K}+\eps\Big), $$
for $\eps$ small enough, where $c$ is a finite constant,
$$F_1:= \P(R(i,1)\geq 0|((N_A,N_a)(\tau_n^K),n\leq J^K(1)),T_\eps^K \leq S_\eps^K), $$
and $R(i,1)$ is defined in \eqref{def:Rij}.
Then Equations (7.21), (7.23), (7.24) and (7.26) of \cite{smadi2014eco} yield
$$ \underset{K \to \infty}{\limsup}\ \Big| \E[F_1^j(1-F_1)^{d-j}|T_\eps^K \leq S_\eps^K] -(1-q_1)^jq_1^{(n-j)} \Big| \leq c \eps, $$
where $q_1$ has been defined in \eqref{defq1q2}, which allows to conclude
\begin{equation} \label{indepr1}   \underset{K \to \infty}{\limsup} \ \Big| \P(R(1)=j|T_\eps^K\leq S_\eps^K)-
{d\choose j}(1-q_1)^jq_1^{(d-j)} \Big|\leq c\eps, \end{equation}
for $\eps$ small enough where $c$ is a finite constant.\\

The derivation of the second probability, $\P(R(2)=j+k|T_\eps^K\leq S_\eps^K,R(1)=j)$, follows the same outline.
The lineages where $N1$ does not escape the sweep can be seen as
lineages where $SL$ and $N1$ are the same locus and the recombination probability between $SL-N1$ and $N2$ is $r_2$.
This is due to the independence of the recombinations
between $SL$ and $N1$ and between $N1$ and $N2$. Hence we can rewrite
the probability as follows:
$$ \P(R(2)=j+k|T_\eps^K\leq S_\eps^K,R(1)=j)=\P(R(2)-R(1)=k|T_\eps^K\leq S_\eps^K,d-R(1)=d-j). $$
We can then directly apply the result \eqref{indepr1} for the law of $R(1)$ and get:
\begin{equation}\label{indepr2}  \underset{K \to \infty}{\limsup} \ \Big|  \P(R(2)=j+k|T_\eps^K\leq S_\eps^K,R(1)=j)-
{d-j\choose k}(1-q_2)^kq_2^{(d-j-k)} \Big|\leq c\eps, \end{equation}
for $\eps$ small enough where $c$ is a finite constant and $q_2$ has been defined in \eqref{defq1q2}.\\

The derivation of the last probability in \eqref{decompoindep}
is more involved but follows the same spirit.
First note that we only have to focus on genealogies where $N1$ escapes the sweep. Hence the derivation
of the probability comes down to the derivation of
$ \P(R(1|2)=l|T_\eps^K\leq S_\eps^K,R(1)=j)$.
The idea is to propose an alternative construction of the process with the same law and where we add the recombinations between
$N1$ and $N2$ at the end:
\begin{enumerate}
 \item[$\bullet$] First we construct a trait population process $(N_A,N_a)$ with birth and death rates defined in \eqref{def:totalbd}
 \item[$\bullet$] Second we "add" the recombinations between $SL$ and $N1$: at each birth event we draw a Bernoulli variable with parameter $r_1$ to decide whether there
 is a recombination or not.  If there is a recombination, the parent giving its neutral allele at $N1$ is chosen with a probability proportional to
 its fertility ($f_A$ or $f_a$). {\\After} this step of the construction we know the genealogies of the $d$ neutral alleles at $N1$ sampled at the end of the sweep.
We label $(i_1,...,i_j)$ the $j$ sampled neutral alleles at $N1$ which experience a recombination between $SL$ and $N1$ in their genealogy.
 \item[$\bullet$] Third we "add" the recombinations between $N1$ and $N2$ sequentially in the lineages where there is already a recombination between $SL$ and $N1$:
 we first follow {backward} in time the lineage of $i_1$ and at each birth event we draw a Bernoulli variable with parameter $r_2$ to decide whether there
 is a recombination or not, and choose the parent of neutral allele at $N2$ as in the second step. Then we do the same with the lineage of $i_2$, and so on until
 the lineage of $i_j$. 
 \item[$\bullet$] Finally we "add" the recombinations between $N1$ and $N2$ in {those lineages which were not marked with any recombination between $SL$ and $N1$}.
\end{enumerate}
Such a construction generates a process distributed as the original process and facilitates the study of the dependencies between lineages
$(i_1,...,i_j)$.
According to Lemma \ref{lemma_negevents}, with high probability
there is no recombination between $SL$ and $N1$ after
(backwards in time) a coalescence at locus $N1$ among the $d$ sampled individuals.
In the same way, there is no coalescence at locus $N1$ after a recombination between $SL$ and $N1$
in the $A$-population
(this is due to the large number of $A$-individuals;
similar proof as for the last probability of Proposition \ref{prothirdphase}.
Hence if we introduce
$$ NC(j):=\{\text{there is no coalescence between lineages
$(i_1,...,i_j)$ at locus $N1$}\},$$
 we get:
$$ \P({R(1|2)}=l|T_\eps^K\leq S_\eps^K,R(1)=j)= \P({R(1|2)}=l|T_\eps^K\leq S_\eps^K,R(1)=j,NC(j))+O\Big( \frac{\log K}{K}\Big).
$$
{With the construction of the alternative process we can also define sequentially }for $1 \leq k \leq j$:
\begin{align*}
 NC(j,k):=\{&\text{there is no coalescence between lineages } (i_1,...,i_k)\text{ after completion of the}\\
&\text{process of adding the recombinations between $N1$ and $N2$ in the lineage $i_k$}\}.
\end{align*}
Then, if we introduce for $1\leq k \leq j$ and $\delta \in \{0,1\}$
\begin{align*} \{r_{i_k}=\delta\}:=& \{ \text{there is $\delta$ recombination between $N1$ and $N2$ in the lineage $i_k$} \},\end{align*}
then for $(\delta_1,...,\delta_j) \in \{0,1\}^j$
\begin{multline*}
 \P(r_{i_k}=\delta_k, 1\leq k \leq j|T_\eps^K\leq S_\eps^K,R(1)=j)=\\
 \underset{1 \leq k \leq j}{\prod}
 \P(r_{i_k}=\delta_k|T_\eps^K\leq S_\eps^K,R(1)=j,NC(j),NC(j,1),...,NC(j,k-1))+ O\Big( \frac{\log K}{K}\Big).
\end{multline*}
Indeed, the probability that the event $NC(j,k)$ is not realized after {witnessing} the recombinations between $N1$ and $N2$ in lineage
$i_k$ has order $\log K/K$ according to Lemma \ref{lemma_negevents}.
But for $1 \leq k \leq j$,
\begin{multline}  \P(r_{i_k}=\delta_k|T_\eps^K\leq S_\eps^K,R(1)=j,NC(j),...,NC(j,k-1))\\
=\frac{ \P(r_{i_k}=\delta_k,R(1)=j,NC(j),...,NC(j,k-1)|T_\eps^K\leq S_\eps^K) }
{\P(R(1)=j,NC(j),...,NC(j,k-1)|T_\eps^K\leq S_\eps^K) }\\
=\frac{ \P(r_{i_k}=\delta_k,R(1)=j|T_\eps^K\leq S_\eps^K)- \P(r_{i_k}=\delta_k,R(1)=j,(NC(j)\cap...\cap NC(j,k-1))^c|T_\eps^K\leq S_\eps^K) }
{ \P(R(1)=j|T_\eps^K\leq S_\eps^K)- \P(R(1)=j,(NC(j)\cap...\cap NC(j,k-1))^c|T_\eps^K\leq S_\eps^K) },
\end{multline}
and according to Lemma \ref{lemma_negevents} and Coupling \eqref{couplage2.1}, there exists a finite $c$ such that for $K$ large enough
and $\eps$ small enough,
$$ \P((NC(j)\cap...\cap NC(j,k-1))^c|T_\eps^K\leq S_\eps^K)\leq c\Big(\frac{\log K}{K}+\eps\Big). $$
As $ \P(r_{i_k}=\delta_k,R(1)=j|T_\eps^K\leq S_\eps^K) $ does not go to $0$ when $K$ goes to infinity, we get
\begin{multline*}
\P(r_{i_k}=\delta_k|T_\eps^K\leq S_\eps^K,R(1)=j,NC(j),...,NC(j,k-1))=
 \P(r_{i_k}=\delta_k|T_\eps^K\leq S_\eps^K,R(1)=j)+O\Big( \frac{\log K}{K}+\eps\Big) \\
 =  \delta_k\frac{\P(R1|2(i_k)^{(1,ga)}|T_\eps^K\leq S_\eps^K)}{\P(R(i_k,1)\geq 0|T_\eps^K\leq S_\eps^K)}+
(1-\delta_k)\Big(1-\frac{\P(R1|2(i_k)^{(1,ga)}|T_\eps^K\leq S_\eps^K)}{\P(R(i_k,1)\geq 0|T_\eps^K\leq S_\eps^K)}\Big)
+O\Big( \frac{\log K}{K}+\eps\Big)\\
= \delta_k\frac{1-q_1-q_3}{1-q_1}+
(1-\delta_k)\Big(1-\frac{1-q_1-q_3}{1-q_1}\Big)
+O\Big( \frac{\log K}{K}+\eps\Big),
\end{multline*}
where we recall the definition of $R(i_k,1)$ in \eqref{def:Rij}, the definition of $R1|2(i_k)^{(1,ga)}$ on page \pageref{defgenealogy}, and we used
Proposition \ref{prop:1phase_probg1}.
Adding Equations \eqref{indepr1} and \eqref{indepr2} we finally obtain:
\begin{eqnarray}
 p(j,k,l)&=&{n\choose j}(1-q_1)^{j}q_1^{(n-j)}{n-j\choose k}(1-q_2)^kq_2^{(n-j-k)}{j\choose l}
 \Big(1-\frac{q_3}{1-q_1}\Big)^l\Big(\frac{q_3}{1-q_1}\Big)^{j-l}+ O_K(\eps)\nonumber \\
 & =& \frac{n!}{l!(j-l)!k!(n-j-k)!}(q_1q_2)^{n-j-k}(q_1(1-q_2))^kq_3^{j-l}(1-q_1-q_3)^l+ O_K(\eps).
\end{eqnarray}
This ends the proof of the independence between
genealogies during the first phase.

The derivation of the asymptotic independence of neutral lineages during the third phase is an easy adaptation of Lemma 5.2 and the proof of Proposition 2.6 in
\cite{schweinsberg2005random}, pp 1623-1624 as with high probability two lineages do not coalesce during this phase.} \hfill $\square$

\appendix

\section{Lemma \ref{probeventphase1}}\label{appA}

Recall the definition of $NR(i)^{(1)}$ on page \pageref{defgenealogy}, and Definitions \eqref{NRlsigmai} and \eqref{defRAli}. Then we have
the following approximations for large $K$.

\begin{lem}\label{probeventphase1}
 There exist three finite constants $c$, $K_0$ and $\eps_0$ such that for every $K \geq K_0$ and $\eps \leq \eps_0$
\begin{equation}\label{lemPNRml} \sup_{n_A \in I_\eps^K±1,l \leq \lfloor \eps K \rfloor} \Big|  \P^{(1)}_{(n_A,l)}(NR(i)^{(1)})-
\exp\Big(-\frac{r_1+r_2}{s} \log \frac{ \lfloor \eps K \rfloor}{l}  \Big)\Big|
\leq c \sqrt{\eps},
\end{equation}
\begin{equation}\label{prob:ai1_m'}
 \sup_{\tau \in \{\zeta,\sigma\}}\sup_{n_A \in I_\eps^K±1,k \leq l \leq \lfloor \eps K \rfloor}\Big|\P^{(1)}_{(n_A,k)}(NR(l,\tau,i)  -\exp\Big( -\frac{r_1}{s} \log  \frac{l-1}{k}  \Big)\Big|
\leq c\sqrt{\eps},
\end{equation}
\begin{equation} \label{approxhatPRAl}
 \sup_{l \leq \lfloor \eps K \rfloor} \Big|\P^{(1)}(RA(l,i)) -\Big( 1-\exp\Big(- \frac{f_A}{f_a} \frac{r_2}{s} \log l \Big)\Big)\Big| \leq c \sqrt{\eps}.
\end{equation}
\end{lem}

\begin{proof}
Let us introduce the sigma-algebra generated by the trait population process
\begin{equation*}\label{deftribu} \mathcal{F}:= \sigma \Big( (\tilde{N}_A,\tilde{N}_a)(\tau_m^K),\tau_m^K\leq \tilde{T}_{\eps}^K \Big). \end{equation*}
 We use some ideas developed in \cite{schweinsberg2005random} and extended in \cite{brink2014multsweep} towards the two-locus case. The proof, although quite technical, can be summarized easily:
for $(g,b,c,d,f) \in \R_+^5$, the Triangle
Inequality and the
Mean Value Theorem imply
$$ |g-e^{-b}|\leq |g-e^{-c}|+|c-d|+|d-f|+|f-b| .$$
Hence for every random variables $(X_1,X_2) \in \R_+^2$ and measurable event $C$:
\begin{multline*}
 \Big| \P^{(1)}(C| \mathcal{F})-e^{ -\frac{r_1+r_2}{s} \log \frac{\lfloor \eps K \rfloor}{l} }\Big|
 \leq \Big|\P^{(1)}(C| \mathcal{F})-e^{-X_1}\Big| +
 \Big|X_1-X_2\Big|\\
+\Big|X_2-\E^{(1)}[X_2]\Big|+
 \Big|\E^{(1)}[X_2]-\frac{r_1+r_2}{s} \log \frac{\lfloor \eps K \rfloor}{l} \Big|.
\end{multline*}
By taking the expectation and applying Jensen and Cauchy-Schwarz Inequalities, we obtain:
\begin{multline}\label{ineqX1X2}
 \Big| \P^{(1)}(C)-e^{ -\frac{r_1+r_2}{s} \log \frac{\lfloor \eps K \rfloor}{l}}\Big|\leq
\E^{(1)} \Big|\P^{(1)}(C|\mathcal{F})-e^{-X_1}\Big| +
\E^{(1)} \Big|X_1-X_2 \Big|\\+\sqrt{\var(X_2)}+
 \Big|\E^{(1)}[X_2]-\frac{r_1+r_2}{s} \log \frac{\lfloor \eps K \rfloor}{l}\Big|.
\end{multline}
Hence the idea is to find the appropriate random variables $(X_1,X_2) \in \R_+^2$ to get small quantities on the right
hand side.\\

\noindent \textit{Proof of Equation \eqref{lemPNRml}:} The first step consists in working conditionally on $\mathcal{F}$, describing
this probability as a product of conditional probabilities close to one, as well as in deriving a Poisson approximation.
To this aim, we define for $m\in \N$:
$$ \theta^{(12)} (m):=
\mathbf{1}_{\{\tau_m^K\leq \tilde{T}_\eps^K\}}\mathbf{1}_{ \{\tilde{N}_a(\tau_{m}^{K})-\tilde{N}_a(\tau_{m-1}^{K})=1\}}
(p_{aA}^{(1)}+p_{aA}^{(2)})(\tilde{N}_A,\tilde{N}_a)(\tau_{m-1}^{K}) , $$
where we recall the definition of the $p_{\alpha \alpha'}^{(i)}$ in Definition \ref{defpalphaalpha}. Notice that
 Remark \ref{rem:recprob} p. \pageref{rem:recprob} implies that for $\rho \in \{1,2\}$, $n_A \in I_\eps^K±1$ and $l<\lfloor \eps K \rfloor$,
 \begin{equation}\label{encadtheta}
(1-c\eps)(r_1+r_2)^\rho\Big( \sum_{k=1}^{l-1} \frac{\E^{(1)}_{(n_A,l)}U_{k}^{K}(1)}{(k+1)^\rho}\Big)\leq
\E^{(1)}_{(n_A,l)} \Big[\sum_{m=1}^\infty(\theta^{(12)} (m)\mathbf{1}_{\{\tilde{N}_a(\tau_{m}^{K}) < l\}})^\rho \Big]
\leq
(r_1+r_2)^\rho\Big( \sum_{k=1}^{l-1} \frac{\E^{(1)}_{(n_A,l)}U_{k}^{K}(1)}{(k+1)^\rho}\Big).
 \end{equation}
Then, similarly as in \cite{schweinsberg2005random}, we have for $n_A \in I_\eps^K±1$ and $l<\lfloor \eps K \rfloor$
$$\P^{(1)}_{(n_A,l)}(NR(i)^{(1)}|\mathcal{F})
=\prod_{m=1}^\infty(1-\theta^{(12)} (m)), \quad \P^{(1)}_{(n_A,l)}-\text{a.s.} $$
If we introduce the variable,
\begin{equation*}
 \eta^{(12)} := \sum_{m=1}^{\infty} \theta^{(12)} (m),
\end{equation*}
which will play the role of $X_1$ in \eqref{ineqX1X2}, we get by following the path of Lemma $3.6$ in \cite{schweinsberg2005random}:
\begin{equation} \label{diffcondexp}
\E^{(1)}_{(n_A,l)} \Big|\prod_{m=1}^\infty(1-\theta^{(12)} (m))-\exp( -\eta^{(12)})\Big| \leq
\E^{(1)}_{(n_A,l)} \Big[\sum_{m=1}^\infty(\theta^{(12)} (m))^2 \Big] \\
 \leq  \frac{c}{\log^2 K  },
\end{equation}
for K large enough, $n_A \in I_\eps^K±1, l<\lfloor \eps K \rfloor$ and a finite $c$
(which can be chosen independently of $l$), where we used Equations
\eqref{expupa2} \eqref{expupa3} and \eqref{encadtheta}, and Condition \eqref{assrK} for the last inequality.
Next we introduce an approximation of the random variable $\eta^{(12)}$, namely
\begin{align} \label{defetatilde12}
 \tilde{\eta}^{(12)} :=  \sum_{m=1}^{\infty} \theta^{(12)} (m)\mathbf{1}_{\{\tilde{N}_a(\tau_{m}^{K}) \geq \tilde{N}_a(0)\}},
\end{align}
which will play the role of $X_2$ in \eqref{ineqX1X2}. For $n_A \in I_\eps^K±1$ and $l \leq \lfloor \eps K \rfloor$:
\begin{align} \label{diffX1X2}
0\leq \E^{(1)}_{(n_A,l)}[ \eta^{(12)} - \tilde{\eta}^{(12)}] = &
\E^{(1)}_{(n_A,l)}\Big[\sum_{m=1}^{\infty} \theta^{(12)} (m) \mathbf{1}_{\{\tilde{N}_a(\tau_{m}^{K}) < l\}}\Big]\\
 \leq & \frac{r_1+r_2}{s_+(\eps)s_-^2(\eps)} \sum_{k=1}^{l-1} \frac{(1-s_-(\eps))^{l-k}}{k+1} \leq c \frac{(r_1+r_2) }{l},
\end{align}
for a finite $c$ and $\eps$ small enough, where we used
\eqref{encadtheta} and
\eqref{expupa3} for the first inequality, and \eqref{lemma3.5} for the second one.
This latter ensures that $c$ can be chosen independently of $l$.
The expected value of $\tilde{\eta}^{(12)}$ can be bounded by using \eqref{encadtheta},
\eqref{expupa2} and \eqref{lemma3.5}
\begin{align}\label{approxtildeeta1}
 \E^{(1)}_{(n_A,l)}[\tilde{\eta}^{(12)}]
& \geq   (1-c\eps)(r_1+r_2)\sum_{k=l}^{\lfloor\eps K \rfloor -1}\frac{1}{k+1}\Big(\frac{1-(1-s)^{\lfloor \eps K \rfloor -k}-(1-s)^{k+1}}{s}-c\eps \Big)
\nonumber\\
&\geq(1-c\eps) \frac{r_1 +r_2}{s}  \log  \frac{\lfloor\eps K \rfloor}{l} - \frac{c}{\log K},
\end{align}
for a finite $c$ and $\eps$ small enough.
For the upper bound we get similarly,
\begin{equation}\label{approxtildeeta2}
 \E^{(1)}_{(n_A,l)}[\tilde{\eta}^{(12)}] \leq
(1+c\eps) \frac{r_1 +r_2}{s}  \log  \frac{\lfloor\eps K \rfloor}{l}.
\end{equation}
The last step consists in bounding the variance of $\tilde{\eta}^{(12)}$. As the calculation of this variance is quite involved,
we introduce an approximation of $\tilde{\eta}^{(12)}$, namely
\begin{eqnarray*} \tilde{\tilde{\eta}}^{(12)} :=
  \sum_{m=1}^{\infty} \mathbf{1}_{\{\tilde{N}_a(\tau_{m-1}^K)\geq \tilde{N}_a(0)\}}
 \mathbf{1}_{\{\tilde{N}_a(\tau_{m}^K)-\tilde{N}_a(\tau_{m-1}^K)=1\}}\frac{r_1+r_2}{\tilde{N}_a(\tau_{m-1}^K)+1}
= \sum_{k=\tilde{N}_a(0)}^{\lfloor\eps K \rfloor -1} \frac{r_1+r_2}{k+1}  U_{k}^K(1).
\end{eqnarray*}
Equation \eqref{rqpr} yields
$ (1-c\eps)\tilde{\tilde{\eta}}^{(12)}\leq {\tilde{\eta}}^{(12)}\leq \tilde{\tilde{\eta}}^{(12)}$
for a finite $c$ and $\eps$ small enough. Hence
\begin{multline}
\label{tildeeta}
\Big|{\var}^{(1)}_{(n_A,l)}\tilde{\eta}^{(12)}-{\var}^{(1)}_{(n_A,l)}\tilde{\tilde{\eta}}^{(12)}\Big|
  \leq  c \varepsilon \E^{(1)}_{(n_A,l)}\Big[\Big(\tilde{\tilde{\eta}}^{(12)} \Big)^2\Big] \\
 \leq  c \varepsilon(r_1+r_2)^2 \sum_{k,k'=l}^{\lfloor \eps K \rfloor-1}\frac{\E^{(1)}[(U_{k}^K(1))^2]+\E^{(1)}[(U_{k'}^K(1))^2]}{(k+1)(k'+1)}
\leq {c\eps} ,
\end{multline}
where we used \eqref{compaU'|Uj} and \eqref{minqk} which ensure that
 $U_{k}^K(1)$ is smaller than a geometric random variable with parameter $q_k^{(s_-(\eps),s_+(\eps))}  \geq s_-(\eps)$.
Thus it is enough to bound ${\var}^{(1)}_{(n_A,l)}\tilde{\tilde{\eta}}^{(12)}$. Thanks to \eqref{covUmlk} and
Condition \eqref{assrK} we get:
\begin{eqnarray*}\label{varetatildetilde}
 {\var}^{(1)}_{(n_A,l)}\tilde{\tilde{\eta}}^{(12)}&=&
 { (r_1+r_2)^2 }\sum_{k,k'=l}^{\lfloor \eps K \rfloor -1}\frac{\cov^{(1)}_{(n_A,l)}(U_{k}^K(1),U_{k'}^K(1))}{(k+1)(k'+1)} \nonumber\\
 & \leq & 2 (r_1+r_2)^2 \sum_{k\leq k'=l}^{\lfloor \eps K \rfloor -1}\frac{\lambda_\eps^{(k'-k)/2}+\eps}{(k+1)(k'+1)}
 \leq  c \frac{\log \lfloor \eps K \rfloor}{\log^2 K}(c+\eps \log \lfloor \eps K \rfloor).
\end{eqnarray*}
Recalling \eqref{tildeeta} and again Condition \eqref{assrK}, we finally obtain
\begin{equation}\label{vartildeta}
 \limsup_{K \to \infty} {\var}^{(1)}_{(n_A,l)}\tilde{{\eta}}^{(12)}\leq c \eps,
\end{equation}
for a finite $c$ independent of $l$ and $\eps$ small enough.
Applying \eqref{ineqX1X2}
with $X_1=\eta^{(12)}$ and $X_2=\tilde{\eta}^{(12)}$ yields
\begin{multline*}
\Big|\P^{(1)}_{(n_A,l)}(NR(i)^{(1)})-e^{-\frac{r_1+r_2}{s}\log \frac{\lfloor \eps K \rfloor}{l}} \Big|\leq
 \E^{(1)}_{(n_A,l)} \Big|\prod_{m=1}^\infty(1-\theta^{(12)} (m))-\exp( -\eta^{(12)})\Big|\\
+ \E^{(1)}_{(n_A,l)}[ \eta^{(12)} - \tilde{\eta}^{(12)}]+\sqrt{{\var}^{(1)}_{(n_A,l)}\tilde{{\eta}}^{(12)}}+
\Big|\E^{(1)}_{(n_A,l)}[\tilde{\eta}^{(12)}] - \frac{r_1 +r_2}{s}  \log  \frac{\lfloor\eps K \rfloor}{l}\Big|.
\end{multline*}
We end the proof of Equation \eqref{lemPNRml} with Inequalities \eqref{diffcondexp}, \eqref{diffX1X2}, \eqref{vartildeta}, \eqref{approxtildeeta1} and
\eqref{approxtildeeta2}.\\

\noindent \textit{Proof of \eqref{prob:ai1_m'}:} There is a supplementary difficulty due to the randomness of $\tilde{N}_a(\tau^K_{R(i,2)})$.
In the previous case we were interested in
an event before the first hitting of $\lfloor \eps K \rfloor$, while in the current case, the conditioning on the value of
$\tilde{N}_a(\tau^K_{R(i,2)})$ does not tell us how many times $\tilde{N}_a$ has hit this value before.
This is why we have introduced $NR(l,\sigma,i)$ and $NR(l,\zeta,i)$ in \eqref{NRlsigmai}.
Define for $m\geq 1$,
$$ \theta^{(1)} (m):=
\mathbf{1}_{\{\tau_m^K\leq \tilde{T}_\eps^K\}}\mathbf{1}_{ \{\tilde{N}_a(\tau_{m}^{K})-\tilde{N}_a(\tau_{m-1}^{K})=1\}}p_{aA}^{(1)}
((\tilde{N}_A,\tilde{N}_a)(\tau_{m}^{K})).  $$
We again condition on the trait population process and get for $n_A \in I_\eps^K ±1$ and $k\leq l<\lfloor \eps K \rfloor$,
\begin{equation}\label{lowupbound_P}\P^{(1)}_{(n_A,k)}(NR(l,\sigma,i)|\mathcal{F})
=\prod_{m=1}^{\sigma_{l}^{K}(1)}(1-\theta^{(1)} (m)), \quad \P^{(1)}_{(n_A,k)}-\text{a.s.},\end{equation}
and the same expression with $\sigma$ replacing $\zeta$. We define the corresponding parameters for the Poisson approximation as follows:
\begin{align*}
\eta^{(1),-}_{l} := \sum_{m=1}^{\sigma_{l}^K(1)} \theta^{(1)} (m) ,\ \text{ and } \
 \eta^{(1),+}_{l} := \sum_{m=1}^{\zeta_l^K(1)} \theta^{(1)} (m).
\end{align*}
They will play the role of $X_1$ in \eqref{ineqX1X2}.
We will show that both can be approximated by:
\begin{align}\label{defetatilde1}
 \tilde{\eta}^{(1)}_{l} := \sum_{m=1}^{\zeta_l^K(1)} \theta^{(1)} (m)
 \mathbf{1}_{\{\tilde{N}_a(0)\leq \tilde{N}_a(\tau_{m}^{K}) \leq l\}},
\end{align}
which will play the role of $X_2$ in \eqref{ineqX1X2}.
Recall Definitions \eqref{Umjk}, \eqref{Dk} and \eqref{defjumpexc}. On the one hand, for $n_A \in I_\eps^K±1$ and $k<\lfloor \eps K \rfloor$,
\begin{align}\label{approx+}
\E^{(1)}_{(n_A,k)}[\eta^{(1),+}_{l}-\tilde{\eta}^{(1)}_{l} ] &= \E^{(1)}_{(n_A,k)} \Big[\sum_{m=1}^{\zeta_l^K(1)}
\theta^{(1)} (m) (\mathbf{1}_{\{ N_a^K(\tau_{m}^{K}) <k\}} +  \mathbf{1}_{\{ N_a^K(\tau_{m}^{K}) > l\}})\Big]\nonumber\\
&\leq \E^{(1)}[D^{K}_{k}(1)] \sum_{j=1}^{k-1} \sup_{n_A \in I_\eps^K}\hspace{.1cm}p_{aA}^{(1)}(n_A,j)
\sup_{n_A \in I_\eps^K±1}\hspace{.1cm}\E^{(1)}_{(n_A,k-1)}[U_{n_A,k,j}^K(1)]\nonumber\\
&+ \E^{(1)}[U^{K}_{l}(1)] \sum_{j=l+1}^{ \lfloor \eps K \rfloor} \sup_{n_A \in I_\eps^K} \hspace{.1cm} p_{aA}^{(1)}(n_A,j)
 \sup_{n_A \in I_\eps^K±1}\hspace{.1cm}\E^{(1)}_{(n_A,l+1)}[U_{n_A,l,j}^K(1)|\sigma_l^K(1)<\infty],
\end{align}
where we used that in the first phase, under $\P^{(1)}$, the number of excursions below $k$ (resp. above $l$) is equal to $D_k^K(1)$
(resp. $U_l^K(1)-1$). Applying Inequality \eqref{expupa2}, Lemma \ref{excunder}, and Equation \eqref{rqpr},
we get the existence of a finite $c$ such that for $\eps$ small enough:
\begin{equation*}
 \E^{(1)}_{(n_A,k)}[\eta^{(1),+}_{l}-\tilde{\eta}^{(1)}_{l} ] \leq
 cr_1 \sum_{j=1}^{ \lfloor \eps K \rfloor} \frac{\mu_\eps^{|j-l|}}{j+1} \leq  \frac{c}{\log K},
\end{equation*}
as $\mu_\eps\in(0,1)$ for $\eps$ small enough and by Condition \eqref{assrK}.
On the other hand, by using the same results as in \eqref{approx+}, we get
\begin{eqnarray*}
\E^{(1)}_{(n_A,k)}[| \eta^{(1),-}_{l}-\tilde{\eta}^{(1)}_{l} | ]
&\leq & \E^{(1)}_{(n_A,k)} \Big[\sum_{m=1}^{\sigma_{l}^K(1)} \theta^{(1)} (m) \mathbf{1}_{\{\tilde{N}_a(\tau_{m}^{K})<k\}}
+ \sum_{m=\sigma_{l}^{K}(1)+1}^{\zeta_l^K(1)} \theta^{(1)} (m)\mathbf{1}_{\{k\leq \tilde{N}_a(\tau_{m}^{K}) \leq l\}} \Big]\\
 &\leq & cr_1 \Big(\sum_{j=1}^{k-1}\frac{\mu_\eps^{k-j}}{j+1} +\sum_{j=k}^{l-1} \frac{\mu_\eps^{l-j}}{j+1} \Big)
 \leq \frac{c}{\log K}.\end{eqnarray*}
This shows that it is sufficient to use
$\tilde{\eta}^{(1)}_{l}$ for the Poisson approximation. From \eqref{diffcondexp} we deduce that
this approximation holds true up to terms of order $1/\log^{2} K$. Recalling once again \eqref{ineqX1X2}, we see that it only
remains to calculate the expected value of $\tilde{\eta}^{(1)}_{l}$ and to bound its variance.
The expectation can be approximated in the same way as the expected value of $\tilde{\eta}_{l}^{(12)}$ from
the previous part in \eqref{approxtildeeta1} and \eqref{approxtildeeta2}:
\begin{align}
(1-c\eps) \frac{r_1}{s} \log \frac{l-1}{k} - \frac{c}{\log K}\leq
 \E^{(1)}_{(n_A,k)}[\tilde{\eta}^{(1)}_{l}] \leq(1+c\eps) \frac{r_1}{s} \log \frac{l-1}{k}.
\end{align}
A comparison of the definitions of $\tilde{\eta}^{(1)}_{l}$ in \eqref{defetatilde1} and $\tilde{\eta}^{(12)}$
in \eqref{defetatilde12} shows that the variance of $\tilde{\eta}^{(1)}_{l}$ can be bounded by the same expression, that is,
a constant times $\eps$. This ends the proof of Equation \eqref{prob:ai1_m'}.\\

\noindent \textit{Proof of Equation \eqref{approxhatPRAl}}
It can be done in a similar way as for Equations \eqref{lemPNRml} and \eqref{prob:ai1_m'}. We have the following
lower and upper bounds:
\begin{multline}
\label{lowupbound_P_AA}
 \prod_{m=1}^{\zeta_l^K(1)} \Big[ 1- p_{AA}^{(2)}(\tilde{N}_A,\tilde{N}_a)(\tau_m^K) \Big]
\leq
1-\P^{(1)}(RA(l,i)|\mathcal{F})
 \leq \prod_{m=1}^{\sigma_l^K(1)} \Big[ 1- p_{AA}^{(2)}(\tilde{N}_A,\tilde{N}_a)(\tau_m^K) \Big].
\end{multline}
Once again we aim at deriving a Poisson approximation. As a birth event in the $A$-population is needed
to see a recombination within the $A$-population,
bounds on the expected number of jumps will concern the process $\tilde{N}_A$ and we have to use Lemma \ref{lemespvarmathcalU}.
\end{proof}

\section{Technical results}

This section is dedicated to technical results needed in the proofs.
First we recall a well known result on the hitting times of birth and death processes which can be found in \cite{schweinsberg2005random} Lemma 3.1:

\begin{pro}
Let $Z=(Z_t)_{t \geq 0}$ be a birth and death process with individual birth and death rates $b$ and $d $. For $i \in \Z^+$,
$T_i=\inf\{ t\geq 0, Z_t=i \}$ and $\P_i$ is the law of $Z$ when $Z_0=i$. Then
for $(i,j,k) \in \Z_+^3$ such that $j \in (i,k)$,
\begin{equation} \label{hitting_times} \P_j(T_k<T_i)=\frac{1-(d/b)^{j-i}}{1-(d/b)^{k-i}} .\end{equation}
\end{pro}

We also recall Lemma 3.5 in \cite{schweinsberg2005random}
and the first part of Equation (A.16) in \cite{smadi2014eco} which are used several times:

\begin{lem}  \
\begin{enumerate}
 \item[$\bullet$] If $a>1$ there is a $C$ such that for every $N \in \N$,
\begin{equation} \label{lemma3.5}\sum_{j=1}^N  \frac{a^j}{j}\leq \frac{Ca^N}{N}.\end{equation}
 \item[$\bullet$] Recall Definition \eqref{defqks1s2}. Then for $(s_1, s_2) \in (0,1)^2$ and $k< \lfloor \eps K \rfloor$,
\begin{equation} \label{minqk} q^{(s_1\wedge s_2,s_1 \vee s_2)}_k\geq s_1 \wedge s_2. \end{equation}
\end{enumerate}
\end{lem}


Finally, we state two technical results. The first one can be proven by using characteristic functions, the proof of the second Lemma is given below:

\begin{lem}\label{lemgeom}
Let $V$ be a geometric random variable with parameter $p_1$ and $(G^i, i \in \N)$ a sequence of independent geometric random variables with parameter $p_2$, independent of $V$.
Then the random variable:
$$ Z:= \underset{i \leq V}{\sum}G^i $$
is geometrically distributed with parameter $p_1p_2$.
\end{lem}

\begin{lem}\label{equivalent}
 Let $(c_N, N \in \N)$ be a bounded sequence of $\R$.
 Then there exists a finite constant $c$ such that
$$\limsup_{N \to \infty}\sup_{k \leq N}\Big|
\sum_{l=1}^{k-1}\frac{e^{\frac{c_N }{\log N}\log l}}{l+1}-\frac{\log N}{c_N} (e^{\frac{c_N }{\log N}\log k}-1) \Big| \leq c.$$
\end{lem}

\begin{proof}
We prove the Lemma for a sequence $(c_N, N \in \N)$ in $\R^*$ and extend the result by using the convention
$$ \Big( \frac{\log N}{c_N} (e^{\frac{c_N }{\log N}\log k}-1)\Big)_{|c_N=0}=\log k.$$
 The idea is to compare the sum with the integral
$$ \int_1^k x^{\frac{c_N}{\log N}-1}dx= \frac{\log N}{c_N} (e^{\frac{c_N }{\log N}\log k}-1).$$
Let $l$ be in $\{1,...,N-1\}$. Then we have
\begin{eqnarray*}
 \int_l^{l+1} x^{\frac{c_N}{\log N}-1}dx -\frac{l^{\frac{c_N }{\log N}}}{l+1} &=& \frac{\log N}{c_N} \Big((l+1)^{\frac{c_N }{\log N}}-l^{\frac{c_N }{\log N}}  -\frac{c_N}{\log N}\frac{l^{\frac{c_N }{\log N}}}{l+1} \Big)\\
&=& \frac{\log N}{c_N}l^{\frac{c_N }{\log N}} \Big(\Big(1+\frac{1}{l}\Big)^{\frac{c_N }{\log N}}-1  -\frac{c_N}{(l+1)\log N}\Big).
\end{eqnarray*}
An application of the Taylor-Lagrange formula yields that
$$ \Big(1+\frac{1}{l}\Big)^{\frac{c_N }{\log N}}-1=\frac{c_N}{l\log N}+\frac{c_N}{\log N}\Big(\frac{c_N}{\log N}-1 \Big) \frac{1}{2l^2}\Big(1+x\Big)^{\frac{c_N }{\log N}-2} $$
where $x$ belongs to $[0,1/l]$.
As the sequence $(c_N, N \in \N)$ is bounded, we deduce that there exists a finite constant $c$ such that
$$  \Big|\int_l^{l+1} x^{\frac{c_N}{\log N}-1}dx -\frac{l^{\frac{c_N }{\log N}}}{l+1}\Big|\leq \frac{c}{l^2}. $$
This ends up the proof of Lemma \ref{equivalent}.
\end{proof}

{\bf Acknowledgements:} {\sl The authors would like to thank Jean-François Delmas, Sylvie M\'el\'eard and Anja Sturm for their careful
reading of this paper. 
They also want to thank an anonymous reviewer for several suggestions and improvements.
This work  was partially funded by project MANEGE `Mod\`eles
Al\'eatoires en \'Ecologie, G\'en\'etique et \'Evolution' of the French national
research agency ANR-09-BLAN-0215, Chair Mod\'elisation Math\'ematique et Biodiversit\'e Veolia Environnement-
Ecole Polytechnique-Museum National d'Histoire Naturelle-Fondation X
and the French national research agency ANR-11-BSV7-013-03,
the DFG through SPP priority programme 1590 and the RTG 1644, `Scaling Problems in Statistics'.}

\bibliographystyle{abbrv}
\bibliography{biblio_hitch}

\end{document}